%% file: das.tex
\documentclass[reqno,11pt]{amsart}
\input{macros}
\title{Yangians and degenerate affine Schur algebras}
\author{Jonathan Brundan}
\address[J.B.]{Department of Mathematics, University of Oregon, Eugene, OR, USA
}
\urladdr{\href{https://pages.uoregon.edu/brundan}{https://pages.uoregon.edu/brundan}, \textrm{\textit{ORCiD}:} \href{https://orcid.org/0009-0009-2793-216X}{0009-0009-2793-216X}}
\email{brundan@uoregon.edu}

\author{Viacheslav Ivanov}
\address[V.I.]{Department of Mathematics, University of Oregon, Eugene, OR, USA
}
\email{vivanov@uoregon.edu}

\thanks{This research was supported in part by NSF grant DMS-2348840.
Some of the material is based upon work supported by the National Science Foundation under Grant No. DMS-1929284 while the first author was in residence at the Institute for Computational and Experimental Research in Mathematics in Providence, RI, during the Categorification and Computation in Algebraic Combinatorics semester program in Fall 2025.
}



\begin{document}

\begin{abstract}
Drinfeld's degenerate affine analog of Schur-Weyl duality
relates representations of the degenerate affine Hecke algebra $\aH_r$ to representations of the Yangian $\Yn$. One way to understand the construction is to introduce an intermediate algebra $\aS(n,r)$, the {\em degenerate affine Schur algebra}, which appears both as the endomorphism algebra
of an induced tensor space over $\aH_r$,
and as the image of a homomorphism $\Drinfeld_{n,r}:\Yn\twoheadrightarrow \aS(n,r)$.
In this paper, we describe $\Drinfeld_{n,r}$ 
using a diagrammatic calculus. 
Then we use a theorem of Drinfeld to 
compute $\ker\Drinfeld_{n,r}$ when $n > r$,
thereby giving a presentation of $\aS(n,r)$ in these cases.
We formulate a conjecture in the remaining cases.
Finally, we apply results of Arakawa 
to develop some of the representation theory of $\aS(n,r)$.
\end{abstract}

\maketitle
\tableofcontents

\input{s1-intro}

\input{s2-reminders}

\input{s3-basis}
\input{s4-pres}
\input{s5-rels}
\input{s6-center}
\input{s7-drinfeld}

\input{s8-diags}
\input{s9-presenting}
\input{s10-reps}
\bibliographystyle{alphaurl}
\bibliography{das}
\end{document}

%% file: macros.tex
\usepackage{
    amsmath,
    amsfonts,
    amssymb,
    amsthm,
    amscd,
    comment,
    enumitem,
    etoolbox,
    textcomp,
    gensymb,
    mathtools,
    mathdots,
    dsfont,
    setspace
}
\usepackage[dvipsnames]{xcolor}
\usepackage[colorlinks=true, linkcolor=black, citecolor=purple, urlcolor=blue, breaklinks=true]{hyperref}

\usepackage{tikz}
\usetikzlibrary{arrows,patterns}
\usetikzlibrary{arrows.meta}
\usetikzlibrary{shapes.geometric}
\usetikzlibrary{decorations.markings,decorations.pathreplacing}
\usetikzlibrary{cd}
\usetikzlibrary{patterns}
\usetikzlibrary{calc}
\tikzset{S/.style={baseline={([yshift=-0.5ex]current bounding box.center)}}}
\tikzset{anchorbase/.style={baseline={([yshift=-0.5ex]current bounding box.center)}}}
\tikzset{centerzero/.style={baseline={([yshift=-0.5ex]{0,#1})}},
    centerzero/.default={0}}

\tikzset{S/.style={thin,black,centerzero}}

\newcommand{\circltimes}
{\mathchoice   
{\begin{tikzpicture}[baseline=-3pt]
\node at (0,0) {$\otimes$};
\draw[line width=0.27pt] (-.087,-.08) to (-.087,.08);
\end{tikzpicture}}     
{\begin{tikzpicture}[baseline=-3pt]
\node at (0,0) {$\otimes$};
\draw[line width=0.27pt] (-.087,-.08) to (-.087,.08);
\end{tikzpicture}}
{\!\begin{tikzpicture}[baseline=-3pt]
\node at (0,0) {$\scriptstyle\otimes$};
\draw[line width=0.26pt] (-.068,-.064) to (-.068,.064);
\end{tikzpicture}\!}
{\!\begin{tikzpicture}[baseline=-1.9pt]
\node at (0,0) {$\scriptscriptstyle\otimes$};
\draw[line width=0.24pt] (-.055,-.05) to (-.055,.05);
\end{tikzpicture}\!}
}

\newcommand{\circltimesintikzcd}
{\mathchoice   
{\!\begin{tikzpicture}[baseline=-1pt]
\node at (0,0) {$\otimes$};
\draw[line width=0.27pt] (-.087,.01) to (-.087,.18);
\end{tikzpicture}\!}     
{\!\begin{tikzpicture}[baseline=-1pt]
\node at (0,0) {$\otimes$};
\draw[line width=0.27pt] (-.087,.01) to (-.087,.18);
\end{tikzpicture}\!}
{\!\!\begin{tikzpicture}[baseline=-.5pt]
\node at (0,0) {$\scriptstyle\otimes$};
\draw[line width=0.26pt] (-.068,.01) to (-.068,.13);
\end{tikzpicture}\!\!}
{\!\!\begin{tikzpicture}[baseline=-.5pt]
\node at (0,0) {$\scriptscriptstyle\otimes$};
\draw[line width=0.24pt] (-.055,.01) to (-.055,.1);
\end{tikzpicture}\!\!}
}

\tikzset{gcolor/.style={green!60!black,semithick}}
\tikzset{pinhead/.style={thin,gray,fill=yellow!40!white}}
\tikzset{bulb/.style={thin,black,fill=black}}
\tikzset{bulbinverse/.style={thin,black,fill=white}}
\tikzset{internalbubble/.style={thin,black,fill=blue!30!white}}
\tikzset{fakebubble/.style={semithick,black,fill=red!20!white}}
\tikzset{zeroline/.style={thin,densely dotted}}

\colorlet{darkg}{green!50!black}

\newcommand\dotlabel[1]{$\color{teal}\scriptstyle{#1}$}

\newcommand\strandlabel[1]{$\color{olive}\scriptstyle{#1}$}
\newcommand\botlabel[1]{node[inner sep=0.5pt,anchor=north] {\strandlabel{#1}}}
\newcommand\toplabel[1]{node[anchor=south,inner sep=0.5pt] {\strandlabel{#1}}}
 
\newcommand\strand[2]{\node at (#1) {\strandlabel{#2}}} 

\newcommand\closeddot[2][black]{
    \node at (#2) {\color{#1}$\scriptstyle{\bullet}$}
}

\newcommand\multcloseddot[3]{
    \closeddot{#1};
    \draw (#1) node[anchor=#2] {\dotlabel{#3}}
}

\newcommand{\pin}[3]{
    \path (#1) node[inner sep=1.6pt] (x) {} to (#2) node[rectangle,rounded corners,draw,pinhead,inner sep=2.5pt](y) {$\color{black}\scriptstyle#3$};
    \draw[Triangle Cap-,thick,gray] (x)--(y);
    \closeddot{#1}
}

\newcommand{\pinpinpinpinpin}[7]{
\path (#1) node[inner sep=1.6pt] (x) {} to (#6) node[rectangle,rounded corners,draw,pinhead,inner sep=2.5pt](y) {$\color{black}\scriptstyle#7$};
    \draw[Triangle Cap-,thick,gray] (x)--(#2)--(#3)--(#4)--(#5)--(y);
\closeddot{#5}; \closeddot{#4}; \closeddot{#3}; \closeddot{#2}; \closeddot{#1}
}

\newcommand{\pinpinpinpinpinpin}[8]{
\path (#1) node[inner sep=1.6pt] (x) {} to (#7) node[rectangle,rounded corners,draw,pinhead,inner sep=2.5pt](y) {$\color{black}\scriptstyle#8$};
    \draw[Triangle Cap-,thick,gray] (x)--(#2)--(#3)--(#4)--(#5)--(#6)--(y);
\closeddot{#6}; \closeddot{#5}; \closeddot{#4}; \closeddot{#3}; \closeddot{#2}; \closeddot{#1}
}

\newcommand\spot[1]{\filldraw (#1) circle (.9pt)} 

\newcommand{\circled}[2]{\node[circle,draw,bulb,inner sep=.85pt] at (#1) {$\hspace{0.06pt}\color{white}\scriptscriptstyle \pmb{#2}$}}

\newcommand{\ovaled}[2]{\node[ellipse,draw,bulb,inner sep=.5pt] at (#1) {$\hspace{0.06pt}\color{white}\scriptscriptstyle\pmb{\!#2\!}$}}

\newcommand{\circledinverse}[2]{\node[circle,draw,bulbinverse,inner sep=.85pt] at (#1) {$\hspace{0.06pt}\color{black}\scriptscriptstyle #2$}}

\newcommand{\ovaledinverse}[2]{\node[ellipse,draw,bulbinverse,inner sep=.5pt] at (#1) {$\hspace{0.06pt}\color{black}\scriptscriptstyle\!#2\!$}}

\usepackage{zref-clever}

\zcsetup{
cap,nameinlink=true,
countertype = {
    enumi=item, enumii=item, enumiii=item, enumiv=item},
counterresetby = {
    enumii=enumi, enumiii=enumii, enumiv=enumiii}}
    
\zcRefTypeSetup{item}{
    Name-sg = {} ,
    name-sg = {} ,
    Name-pl = {} ,
    name-pl = {} ,
}

\zcRefTypeSetup{corollary}{
    Name-sg = Corollary ,
    name-sg = corollary ,
    Name-pl = Corollaries ,
    name-pl = corollaries ,
}
\zcRefTypeSetup{definition}{
    Name-sg = Definition ,
    name-sg = definition ,
    Name-pl = Definitions ,
    name-pl = definitions ,
}
\zcRefTypeSetup{example}{
    Name-sg = Example ,
    name-sg = example ,
    Name-pl = Examples ,
    name-pl = examples ,
}
\zcRefTypeSetup{equation}{
    Name-sg = {} ,
    name-sg = {} ,
    Name-pl = {} ,
    name-pl = {} ,
}
\zcRefTypeSetup{lemma}{
    Name-sg = Lemma ,
    name-sg = lemma ,
    Name-pl = Lemmas ,
    name-pl = lemmas ,
}
\zcRefTypeSetup{remark}{
    Name-sg = Remark ,
    name-sg = remark ,
    Name-pl = Remarks ,
    name-pl = remarks ,
}
\zcRefTypeSetup{conjecture}{
    Name-sg = Conjecture ,
    name-sg = conjecture ,
    Name-pl = Conjectures ,
    name-pl = conjectures ,
}
\zcsetup{countertype={subsection=subsection}}
\zcRefTypeSetup{subsection}{
    Name-sg = Subsection ,
    name-sg = subsection ,
    Name-pl = Subsections ,
    name-pl = subsections ,
}
\zcRefTypeSetup{table}{
    Name-sg = Table ,
    name-sg = table ,
    Name-pl = Tables ,
    name-pl = tables ,
}
\zcRefTypeSetup{theorem}{
    Name-sg = Theorem ,
    name-sg = theorem ,
    Name-pl = Theorems ,
    name-pl = theorems ,
}

\newcommand{\cref}[1]{\zcref{#1}}
\newcommand{\Cref}[1]{\zcref[S]{#1}}

\AddToHook{env/conjecture/begin}{\zcsetup{countertype={theorem=conjecture}}}
\AddToHook{env/corollary/begin}{\zcsetup{countertype={theorem=corollary}}}
\AddToHook{env/definition/begin}{\zcsetup{countertype={theorem=definition}}}
\AddToHook{env/example/begin}{\zcsetup{countertype={theorem=example}}}
\AddToHook{env/lemma/begin}{\zcsetup{countertype={theorem=lemma}}}
\AddToHook{env/remark/begin}{\zcsetup{countertype={theorem=remark}}}
\AddToHook{env/table/begin}{\zcsetup{countertype={theorem=table}}}

\makeatletter

\newtheorem{theorem}{Theorem}[section]

\newtheorem{lemma}[theorem]{Lemma}
\newtheorem{corollary}[theorem]{Corollary}
\newtheorem{conjecture}[theorem]{Conjecture}
\theoremstyle{definition}
\newtheorem{definition}[theorem]{Definition}
\newtheorem{remark}[theorem]{Remark}
\newtheorem{example}[theorem]{Example}

\numberwithin{equation}{section}
\allowdisplaybreaks

\newcommand\one{\mathds{1}}

\newcommand{\Hom}{\operatorname{Hom}}
\newcommand{\Z}{\mathbb{Z}}
\newcommand{\C}{\mathbb{C}}
\newcommand{\Q}{\mathbb{Q}}
\def\inc{\operatorname{inc}}
\newcommand{\N}{\mathbb{N}}
\newcommand{\Res}{\operatorname{Res}}
\newcommand{\id}{\operatorname{id}}
\renewcommand{\Im}{\operatorname{Im}}

\newcommand{\End}{\operatorname{End}}
\newcommand{\rev}{\mathrm{rev}}
\newcommand{\op}{\mathrm{op}}

\newcommand{\g}{\mathfrak{g}}
\newcommand{\gl}{\mathfrak{gl}}
\renewcommand{\sl}{\mathfrak{sl}}
\newcommand\lround{(\!(}
\newcommand\rround{)\!)}
\newcommand\llbracket{[\![}
\newcommand\rrbracket{]\!]}
\newcommand{\kk}{\Bbbk}

\newcommand{\gr}{\operatorname{gr}}
\def\ev{\operatorname{ev}}

\def\Mat#1#2{\mathrm{Mat}(#1,#2)}
\def\KK{\mathbb{K}}
\def\S{\mathrm{S}}
\def\G{\mathrm{G}}
\def\T{\mathrm{T}}
\def\GL{\mathrm{GL}}
\def\resS{\mathrm{S}^{[p]}}

\def\Schur{\mathbf{S}\hspace{-.25pt}\mathbf{chur}}
\def\Sym{\mathbf{S}\hspace{-.25pt}\mathbf{ym}}
\def\aSym{\mathbf{A}\hspace{-.3pt}\mathbf{S}\hspace{-.25pt}\mathbf{ym}}
\def\qSchur{\mathbf{S}\hspace{-.25pt}\mathbf{chur}_q}
\def\aSchur{\mathbf{A}\hspace{-.3pt}\mathbf{S}\hspace{-.25pt}\mathbf{chur}}
\def\gSchur{\mathbf{A}\hspace{-.3pt}\mathbf{S}\hspace{-.25pt}\mathbf{chur}_0}

\def\aS{\mathrm{A}\hspace{-.2pt}\mathrm{S}}
\def\qS{\mathrm{A}\hspace{-.2pt}\mathrm{S}_q}
\def\oneS{\mathrm{A}\hspace{-.2pt}\mathrm{S}_1}
\def\aH{\mathrm{A}\hspace{-.2pt}\mathrm{H}}
\def\gS{\mathrm{A}\hspace{-.2pt}\mathrm{S}_0}

\def\Y{\mathrm{Y}}
\def\SY{\mathrm{S}\hspace{-.5pt}\mathrm{Y}}
\def\Yn{\mathrm{Y}(\gl_n)}
\def\resYn{\mathrm{Y}^{[p]}(\gl_n)}
\def\SYn{\mathrm{Y}(\sl_n)}
\def\U{\mathrm{U}}
\def\SU{\mathrm{S}\hspace{-.4pt}\mathrm{U}}
\def\resUn{\mathrm{U}^{[p]}(\gl_n)}
\def\Un{\mathrm{U}(\gl_n)}
\def\SUn{\mathrm{U}(\sl_n)}
\def\P{\mathrm{P}}
\def\Q{\mathrm{Q}}
\def\v{m}

\def\gYn{\mathrm{U}(\gl_n[x])}
\def\diag{\operatorname{diag}}

\def\eps{\varepsilon}
\def\ba{\underline{a}}
\def\bb{\underline{b}}
\def\bh{\text{\boldmath$h$}}
\def\bi{\text{\boldmath$i$}}
\def\bj{\text{\boldmath$j$}}

\def\mod{\!\operatorname{-mod}}

\def\transpose{\mathtt{T}}
\def\dash{}
\def\bully{\diamond}

\def\vid{\text{\boldmath$\scriptstyle\cdot|\cdot$}}
\def\barotimes{\;\bar\otimes\;}
\def\drinfeld{\mathtt{d}}
\def\Drinfeld{\mathtt{D}}
\def\SDrinfeld{\mathtt{SD}}
\def\sdrinfeld{\mathtt{sd}}
\def\Functor{\mathtt{F}}
\def\functor{\mathtt{f}}
\def\DRINFELD{\widetilde{\mathtt{D}}}

\def\Dist{\operatorname{Dist}}
\DeclareSymbolFont{symbolsC}{U}{pxsyc}{m}{n}
\SetSymbolFont{symbolsC}{bold}{U}{pxsyc}{bx}{n}
\DeclareFontSubstitution{U}{pxsyc}{m}{n}
\DeclareMathSymbol{\medcirc}{\mathbin}{symbolsC}{7}
\def\star{*}
\def\ostar{\circledast}

%% file: s1-intro.tex
\setcounter{section}{0}

\section{Introduction}

There has been some interest recently in the development of diagrammatic tools for working with Schur algebras and related objects appearing in representation theory. For example, the classical Schur algebra $\S(n,r)$ has a standard basis indexed by certain minimal length double coset representatives in the symmetric group $\S_r$.
These double coset representatives may be represented graphically by {\em double coset
diagrams} with $n$ vertical strings of total thickness $r$
at the top and bottom boundaries, like in the following example which is a picture of
a minimal length double coset representative for the subgroups $\S_5 \times \S_3 \times \S_8 \times \S_4$  and $\S_6 \times \S_2 \times \S_7 \times \S_5$ 
in the symmetric group $\S_{20}$:
$$
\begin{tikzpicture}[centerzero,scale=1.3]
\draw[line width=1pt] (-1.17,-.6) to (-1.17,-.5) to (-.015,.5) to (-.015,.6);
\draw (.588,.6) to (.588,.5) to (.018,-.5) to (.018,-.6);
\draw(-.018,-.6) to (-.018,-.5) to (-.59,.5) to (-.59,.6);
\draw[line width=.7pt] (-.605,-.6)\botlabel{2} to (-.605,.6)\toplabel{3};
\draw[ultra thick] (0,-.6)\botlabel{7} to (0,.6)\toplabel{8};
\draw[line width=.9pt] (.6,.6)\toplabel{4} to (.6,-.6)\botlabel{5};
\draw[line width=.6pt] (-1.17,.6) to (-1.17,.5) to (.585,-.5) to (.585,-.6);
\draw[line width=.9pt] (-1.18,.6)\toplabel{5} to (-1.18,-.6)\botlabel{6};
\strand{.27,.15}{1};
\strand{-.25,-.3}{1};
\strand{-.12,.15}{5};
\strand{-.85,-.39}{3};
\strand{-1.3,0}{3};
\strand{.73,0}{3};
\strand{-.9,.2}{2};
\end{tikzpicture}.
$$
The numerical labels indicate the thicknesses of the individual strings.
The same diagrams can be used to represent corresponding standard basis vectors in the Schur algebra; the example is a 
vector in $\S(4,20)$.
Then Schur's formula for computing products of standard basis vectors can be reinterpreted in terms of local relations on string diagrams which allow non-reduced diagrams to be simplified algorithmically.

Working over $\C$, there is a surjective algebra homorphism from the universal enveloping algebra of $\gl_n(\C)$ to $\S(n,r)$
defined on the generators $d_i := e_{i,i}$, 
$e_i := e_{i,i+1}$ and $f_i := e_{i+1,i}$ by
\begin{align*}
d_i&\mapsto\!
\sum_{\lambda \in \Lambda(n,r)}\!
\lambda_i 
\begin{tikzpicture}[centerzero,scale=.9]
\draw[ultra thick] (-1,-.6)\botlabel{\lambda_1} to (-1,.6);
\node at (-.58,0) {$\dots$};
\draw[ultra thick] (-.2,-.6)\botlabel{\lambda_i} to (-.2,.6);
\node at (.22,0) {$\dots$};
\draw[ultra thick] (.6,-.6)\botlabel{\lambda_n} to (.6,.6);
\end{tikzpicture},\end{align*}\begin{align*}
e_i &\mapsto \!
\sum_{\substack{\mu \in \Lambda(n,r)\\\mu_{i+1}>0}}
\begin{tikzpicture}[centerzero,scale=.9]
\draw[ultra thick] (-1.5,-.6)\botlabel{\mu_1} to (-1.5,.6);
\node at (-1.08,0) {$\dots$};
\draw[ultra thick] (-.7,-.6)\botlabel{\mu_{i\!-\!1}\ } to (-.7,.6);
\draw[ultra thick] (-.2,-.6)\botlabel{\mu_i} to (-.2,.6);
\draw[ultra thick] (0.5,-.6)\botlabel{\mu_{i\!+\!1}\ } to (0.5,.6);
\draw[ultra thick] (1,-.6)\botlabel{\ \mu_{i\!+\!2}} to (1,.6);
\draw (.47,-.6) to (.47,-.3) to (-0.17,.3) to (-.17,.6);
\node at (1.42,0) {$\dots$};
\draw[ultra thick] (1.8,-.6)\botlabel{\mu_n} to (1.8,.6);
\end{tikzpicture},&
f_i&\mapsto
\!\sum_{\substack{\mu \in \Lambda(n,r)\\\mu_{i}>0}}
\begin{tikzpicture}[centerzero,scale=.9]
\draw[ultra thick] (-1.5,-.6)\botlabel{\mu_1} to (-1.5,.6);
\node at (-1.08,0) {$\dots$};
\draw[ultra thick] (-.7,-.6)\botlabel{\mu_{i\!-\!1}\ } to (-.7,.6);
\draw[ultra thick] (-.2,-.6)\botlabel{\mu_i} to (-.2,.6);
\draw[ultra thick] (0.5,-.6)\botlabel{\mu_{i\!+\!1}\ } to (0.5,.6);
\draw[ultra thick] (1,-.6)\botlabel{\ \mu_{i\!+\!2}} to (1,.6);
\draw (.47,.6) to (.47,.3) to (-0.17,-.3) to (-.17,-.6);
\node at (1.42,0) {$\dots$};
\draw[ultra thick] (1.8,-.6)\botlabel{\mu_n} to (1.8,.6);
\end{tikzpicture},
\end{align*}
where $\Lambda(n,r)$ denotes the set of compositions $\lambda = (\lambda_1,\dots,\lambda_n) \in \N^n$ whose parts sum to $r$. The unlabelled diagonal strings here are of thickness 1.
It is natural to want to view the diagrams for the images of $e_i$ and $f_i$ as compositions of their top and bottom halves, but the half diagrams themselves do not make sense as elements of $\S(n,r)$ since the slice across the middle cuts $n+1$ rather than $n$ strings. This suggests that it might be desirable to
pass from the Schur algebra to a more general object where there are fewer constraints. 

These ideas were developed systematically 
in \cite{BEEO}, defining the {\em Schur category} $\Schur$ 
to be a strict monoidal category with objects given by compositions, and morphisms represented by string diagrams with strings of appropriate thicknesses.
Tensor product is defined on objects by concatenation of compositions and on morphisms by horizontal stacking of string diagrams.
The main families of generating morphism are the merges
$\begin{tikzpicture}[anchorbase,centerzero,scale=.7]
\draw[-,line width=1pt] (0.28,-.3) to (0.08,0.04);
\draw[-,line width=1pt] (-0.12,-.3) to (0.08,0.04);
\draw[-,line width=2pt] (0.08,.4) to (0.08,0);
\strand{-.22,-.5}{a};
\strand{0.35,-.5}{b};
\end{tikzpicture}$, the splits $\begin{tikzpicture}[anchorbase,centerzero=.1,scale=.7]
\draw[-,line width=2pt] (0.08,-.3) to (0.08,0.04);
\draw[-,line width=1pt] (0.28,.4) to (0.08,0);
\draw[-,line width=1pt] (-0.12,.4) to (0.08,0);
\strand{-0.22,.6}{a};
\strand{0.36,.6}{b};
\end{tikzpicture}$, and the thick crossings $\begin{tikzpicture}[baseline=-1.5mm,scale=.7]
\draw[-,line width=1.2pt] (-0.3,-.34) to (.3,.34);
\draw[-,line width=1.2pt] (0.3,-.34) to (-.3,.34);
\strand{0.3,-.55}{b};
\strand{-0.3,-.55}{a};
\end{tikzpicture}$,
which satisfy relations which can be expressed in a very economical way; see \cref{assrel,mergesplit}. 
Then $\S(n,r)$ is the 
path algebra of the full subcategory of $\Schur$ with object set $\Lambda(n,r)$.
There is also a quantum analog $\qSchur$ of $\Schur$ which was defined both by generators and relations and with explicit bases for morphism spaces in \cite{BSchur}; one replaces the (singular) thick crossing with the
positive and negative thick crossings
$\begin{tikzpicture}[baseline=-1.5mm,scale=.7]
\draw[-,line width=1.2pt] (0.3,-.34) to (-.3,.34);
\draw[-,line width=4pt,white] (-0.3,-.34) to (.3,.34);
\draw[-,line width=1.2pt] (-0.3,-.34) to (.3,.34);
\strand{0.3,-.55}{b};
\strand{-0.3,-.55}{a};
\end{tikzpicture}$ and
$\begin{tikzpicture}[baseline=-1.5mm,scale=.7]
\draw[-,line width=1.2pt] (-0.3,-.34) to (.3,.34);
\draw[line width=4pt,white] (.3,-.34) to (-.3,.34);
\draw[-,line width=1.2pt] (0.3,-.34) to (-.3,.34);
\strand{0.3,-.55}{b};
\strand{-0.3,-.55}{a};
\end{tikzpicture}$.

The papers \cite{BEEO, BSchur} are quite recent, so of course they 
rest on many previous works.
The excellent idea that string diagrams provide a useful tool for working in Schur-like categories was probably first suggested by Stroppel and Webster; see \cite[Sec.~3.3]{SW}. Splits and merges can already be seen in the MOY diagrams from \cite{MOY}, which were being used at about the same time as \cite{SW} in the development of diagrammatic
categorifications of Schur algebras; see \cite{MSV} and references therein.
In \cite[Sec.~3.2]{MakS}, Maksimau and Stroppel
pioneered the use of
diagrammatics similar to \cite{BEEO,BSchur}
with the addition of coupons on thick strings labelled by symmetric Laurent polynomials in order 
to represent elements of the affine $q$-Schur algebra of Green and Vignéras \cite{Green-affine,vigneras}; see also \cite{MS}.
This work included the case of roots of unity and also considered cyclotomic quotients\footnote{They also consider an extended tensor product version with additional red strands, which we will not say anything about here.}, 
establishing isomorphisms to 
cyclotomic quotients of the quiver Schur algebras of \cite{SW}.
A generators and relations description of some of the algebras in \cite{MakS} was given later in \cite{ShenSongWang}.

In another influential paper \cite{CKM}, certain diagrams called 
{\em webs} were used to present a monoidal category closely related to $\qSchur$.
This terminology goes back to work of Kuperberg \cite{Kuperberg}, but we find it is a little misleading in the Schur algebra context---Kuperberg's webs are certain labelled planar graphs which are not the same as our double coset diagrams, although they are related.
Unlike the situation for $\qSchur$, it is not easy to find explicit bases for morphism spaces in the Cautis-Kamnitzer-Morrison web category; see \cite{Elias} which constructed bases for a closely related variant, and \cite[Th.~8.1]{BSchur} for another approach which involves taking the quotient of $\qSchur$ by a cell ideal.

This paper was inspired instead by the recent work of Song and Wang \cite{SongWang}, who introduced a strict monoidal category defined by generators and relations which they called ``affine web category.'' We prefer to call it the
{\em degenerate affine Schur category}, denoted
$\aSchur$.
The path algebra of the full subcategory of $\aSchur$
with object set $\Lambda(n,r)$
is the {\em degenerate affine Schur algebra} $\aS(n,r)$, which is the degenerate analog of the affine $q$-Schur algebra mentioned already. 
Letting $V$ be the natural representation of $\gl_n(\C)$, the algebra $\aS(n,r)$ can be constructed more directly as the endomorphism algebra
$$
\aS(n,r) = \End_{\dash\aH_r}\big(V^{\otimes r} \otimes_{\C \S_r} \aH_r\big)
$$
of the {\em induced tensor space} $V^{\otimes r} \otimes_{\C \S_r} \aH_r$, which is a right module for the {\em degenerate affine Hecke algebra} $\aH_r$.
Song and Wang also consider cyclotomic quotients, which they show are related to the Schur algebras of higher levels from \cite{BKschur}.
A generalization in a different direction was considered independently in \cite{DKMZ1,DKMZ2}.

In the first half of the paper, we reprove some of the results of Song and Wang about $\aSchur$ (but none of their later results about cyclotomic quotients). A key difference in our exposition is that we allow strings of thickness $r$ to be decorated by symmetric polynomials in $\C[x_1,\dots,x_r]^{\S_r}$, similar to what was done already in the quantum case in \cite{MS, MakS}. We point out one useful relation:  we have that
$$
\begin{tikzpicture}[anchorbase,scale=1.1,centerzero=.4]
\draw[-,line width=1.1pt] (0.58,0) to (.02,1);
\draw[-,line width=1.1pt] (0.02,0) to (.58,1);
\strand{0,-.1}{a};
\strand{0.6,-.1}{b};
\pin{.14,.22}{-.3,.3}{e_d};
\end{tikzpicture}
=
\sum_{s=0}^{\min(a,b,d)}
s!
\begin{tikzpicture}[anchorbase,scale=1.1,centerzero=.4]
\draw[-,thick] (0.58,0) to (0.58,.2) to (.02,.8) to (.02,1);
\draw[-,thick] (0.02,0) to (0.02,.2) to (.58,.8) to (.58,1);
\draw[-,thin] (0,0) to (0,1);
\draw[-,thin] (0.6,0) to (0.6,1);
\strand{0,-.1}{a};
\strand{0.6,-.1}{b};
\strand{-0.1,.5}{s};
\strand{0.75,.5}{s};
\pin{.45,.65}{.3,1.25}{e_{d-s}};
\end{tikzpicture}
$$
for $a,b,d \geq 1$,
where $e_d$ pinned to a string of thickness $a$ denotes a coupon labelled by the $d$th elementary symmetric polynomial in variables $x_1,\dots,x_a$.
This relation allows symmetric polynomials to be commuted past crossings in double coset diagrams. Song and Wang use it only in the special case that $d=a$.
To prove it for smaller values of $d$, we 
work in terms of generating functions, using the diagrammatic shorthands
\begin{align*}
\begin{tikzpicture}[centerzero]
\draw[ultra thick] (0,-.4)\botlabel{r} to (0,.4);
\circledinverse{0,0}{u};
\end{tikzpicture}
&:=
\begin{tikzpicture}[centerzero]
\draw[ultra thick] (0,-.4)\botlabel{r} to (0,.4);
\pin{0,0}{1.4,0}{(u-x_1)\cdots (u-x_r)};
\end{tikzpicture}\ ,
&
\begin{tikzpicture}[centerzero]
\draw[ultra thick] (0,-.4)\botlabel{r} to (0,.4);
\circled{0,0}{u};
\end{tikzpicture}
&:=
\begin{tikzpicture}[centerzero]
\draw[ultra thick] (0,-.4)\botlabel{r} to (0,.4);
\pin{0,0}{1.4,0}{\frac{1}{(u-x_1)\cdots (u-x_r)}};
\end{tikzpicture}
\end{align*}
for coupons labelled by the generating functions
for elementary and complete symmetric polynomials.
These are elements of $\End_{\aSchur}\big((r)\big)\lround u^{-1}\rround$ where $u$ is a formal variable.
The above relation for commuting elementary symmetric polynomials past crossings follows from 
$$
\begin{tikzpicture}[anchorbase,scale=1.1,centerzero=.4]
\draw[-,line width=1.1pt] (0.58,0) to (.02,1);
\draw[-,line width=1.1pt] (0.02,0) to (.58,1);
\strand{0,-.1}{a};
\strand{0.6,-.1}{b};
\circledinverse{.14,.22}{u};
\end{tikzpicture}
=
\sum_{s=0}^{\min(a,b)}
(-1)^s s!
\begin{tikzpicture}[anchorbase,scale=1.1,centerzero=.4]
\draw[-,thick] (0.58,0) to (0.58,.2) to (.02,.8) to (.02,1);
\draw[-,thick] (0.02,0) to (0.02,.2) to (.58,.8) to (.58,1);
\draw[-,thin] (0,0) to (0,1);
\draw[-,thin] (0.6,0) to (0.6,1);
\strand{0,-.1}{a};
\strand{0.6,-.1}{b};
\strand{-0.1,.5}{s};
\strand{0.75,.5}{s};
\circledinverse{.45,.65}{u};
\end{tikzpicture},
$$
which is easier to prove; see \cref{bingley}.
We use this relation as one of the 
defining relations in a monoidal presentation for $\aSchur$ which is equivalent to the presentation originally derived in \cite{SongWang} but more convenient since it fully incorporates symmetric polynomials;
see \cref{songwangpresentation,tobecontinued}. We also prove several complementary results which are not surprising, but are missing in the existing literature. For example, in \cref{weiqiang}, we determine the centers of each of the endomorphism algebras $\End_{\aSchur}(\lambda)$, proving a conjecture from \cite{SongWang}. 

The second half of the paper is concerned with the Yangian $\Yn$
associated to $\gl_n(\C)$, and its subalgebra $\SYn$ which is the Yangian of $\sl_n(\C)$. 
In \cite{Drinfeld}, Drinfeld defined a functor 
$$
V^{\otimes r}\otimes_{\C S_r} -:\aH_r\mod\rightarrow \Yn\mod
$$
which can be used to study finite-dimensional representations of $\Yn$ in the same way that the classical Schur functor is used in the context of representation theory of symmetric and general linear groups.
His main result about this functor is as follows:

\vspace{2mm}
\noindent
\textbf{Theorem} (Drinfeld)\textbf{.}
{\em Assuming $n > r$, 
the composite functor $\Res^{\Yn}_{\SYn} 
\circ (V^{\otimes r} \otimes_{\kk S_r} -)$
defines an equivalence
of categories between $\aH_r\mod$ and the full subcategory
of $\SYn\mod$ consisting of modules whose restriction to $\mathfrak{sl}_n(\C)$ are polynomial representations of degree $r$.}

\vspace{2mm}

\noindent
An analogous result in the quantum setting was proved by Chari and Pressley in \cite{CP}.
The Drinfeld functor was studied further by Arakawa \cite{Arakawa}, including in the case that $n \leq r$.

Applying Drinfeld's functor to the regular representation of $\aH_r$
produces an action of $\Yn$ on the
induced tensor space $V^{\otimes r} \otimes_{\kk S_r} \aH_r$, making it
into a $(\Yn,\aH_r)$-bimodule. This action induces a homomorphism 
$$
\Drinfeld_{n,r}:\Yn \rightarrow \aS(n,r)
$$
which we call the {\em Drinfeld homomorphism}. 
In \cref{SlavaTheorem}, we give
an explicit formula expressing the images under $\Drinfeld_{n,r}$
of the RTT generators $T_{i,j}^{(d)}$ of $\Yn$ in terms of standard bases of $\aS(n,r)$. The result can also be understood diagrammatically; see \cref{SlavaExample}.

It turns out to be much easier to describe $\Drinfeld_{n,r}$ on 
another well-known family of generators for $\Yn$
denoted by $D_i^{(d)}, E_i^{(d)}$ and $F_i^{(d)}$, which are closely related to the Drinfeld generators from \cite{Dnew}.
The generating functions $D_i(u) = 1+\sum_{d \geq 1} D_i^{(d)} u^{-d}$,
$E_i(u) =\sum_{d \geq 1} E_i^{(d)} u^{-d}$
and $F_i(u) = \sum_{d \geq 1} F_i^{(d)} u^{-d}$ arise as entries of the Gauss factorization of the matrix $(T_{i,j}(u))_{1 \leq i,j \leq n}$
of generating functions $T_{i,j}(u) = \delta_{i,j}+\sum_{d\geq 1} T_{i,j}^{(d)} u^{-d}$ for the RTT generators; see \cref{gfact}.
In \cref{def}, we show that $\Drinfeld_{n,r}$ maps
\begin{align*}
D_i(u)&\mapsto
\sum_{\lambda \in \Lambda(n,r)}
\begin{tikzpicture}[centerzero,scale=.9]
\draw[ultra thick] (-1.5,-.6)\botlabel{\lambda_1} to (-1.5,.6);
\node at (-1.08,0) {$\dots$};
\draw[ultra thick] (-.7,-.6)\botlabel{\lambda_{i\!-\!1}\ } to (-.7,.6);
\draw[ultra thick] (0,-.6)\botlabel{\lambda_i} to (0,.6);
\draw[ultra thick] (.7,-.6)\botlabel{\ \lambda_{i\!+\!1}} to (.7,.6);
\node at (1.12,0) {$\dots$};
\draw[ultra thick] (1.5,-.6)\botlabel{\lambda_n} to (1.5,.6);
\ovaledinverse{0,.18}{u+i};
\ovaled{0,-.18}{u+i-1};
\end{tikzpicture},\end{align*}\begin{align*}
E_i(u)&\mapsto
\sum_{\substack{\mu \in \Lambda(n,r)\\\mu_{i+1}>0}}
\begin{tikzpicture}[centerzero,scale=.9]
\draw[ultra thick] (-1.5,-.6)\botlabel{\mu_1} to (-1.5,.6);
\node at (-1.08,0) {$\dots$};
\draw[ultra thick] (-.7,-.6)\botlabel{\mu_{i\!-\!1}\ } to (-.7,.6);
\draw[ultra thick] (-.2,-.6)\botlabel{\mu_i} to (-.2,.6);
\draw[ultra thick] (0.8,-.6)\botlabel{\mu_{i\!+\!1}\ } to (0.8,.6);
\draw[ultra thick] (1.3,-.6)\botlabel{\ \mu_{i\!+\!2}} to (1.3,.6);
\draw (.77,-.6) to (.77,-.3) to (-0.17,.3) to (-.17,.6);
\node at (1.72,0) {$\dots$};
\draw[ultra thick] (2.1,-.6)\botlabel{\mu_n} to (2.1,.6);
\ovaled{0.3,0}{u+i};
\end{tikzpicture}\ ,&
F_i(u)&\mapsto
\sum_{\substack{\mu \in \Lambda(n,r)\\\mu_{i}>0}}
\begin{tikzpicture}[centerzero,scale=.9]
\draw[ultra thick] (-1.5,-.6)\botlabel{\mu_1} to (-1.5,.6);
\node at (-1.08,0) {$\dots$};
\draw[ultra thick] (-.7,-.6)\botlabel{\mu_{i\!-\!1}\ } to (-.7,.6);
\draw[ultra thick] (-.2,-.6)\botlabel{\mu_i} to (-.2,.6);
\draw[ultra thick] (0.8,-.6)\botlabel{\mu_{i\!+\!1}\ } to (0.8,.6);
\draw[ultra thick] (1.3,-.6)\botlabel{\ \mu_{i\!+\!2}} to (1.3,.6);
\draw (.77,.6) to (.77,.3) to (-0.17,-.3) to (-.17,-.6);
\node at (1.72,0) {$\dots$};
\draw[ultra thick] (2.1,-.6)\botlabel{\mu_n} to (2.1,.6);
\ovaled{0.3,0}{u+i};
\end{tikzpicture}.
\end{align*}
The image of the diagonal generator $D_i(u)$ involves some troublesome inhomogeneous symmetric polynomials. 
These can be seen already in the case $n=i=1$, 
when $\aS(1,r)$ is $\C[x_1,\dots,x_r]^{\S_r}$ and the image 
of $D_1(u)$ under the Drinfeld homomorphism is
$$
\left(1+\frac{1}{u-x_1}\right)
\left(1+\frac{1}{u-x_2}\right)\cdots
\left(1+\frac{1}{u-x_r}\right).
$$
The coefficient of $u^{-d-1}$ in the expansion of this as a formal power series in $u^{-1}$ is a symmetric polynomial $\tilde p_d(x_1,\dots,x_r)$ which we call the {\em deformed power sum}
since it is equal to the power
sum $p_d(x_1,\dots,x_r)=x_1^d+\cdots+x_r^d$ plus lower degree terms; see \cref{jonsid}.

Over the complex numbers still, it is well known that $\Drinfeld_{n,r}$ is surjective.
It is natural to ask for explicit generators for
its kernel.
In \cref{s9-presenting}, we formulate a 
precise conjecture about this,
proving our conjecture 
in the case $n > r$ using Drinfeld's theorem.
Surjectivity of $\Drinfeld_{n,r}$ implies that the category of left $\aS(n,r)$-modules
is identified with a full subcategory of $\Yn\mod$ consisting of what we call {\em polynomial representations} of $\Yn$ of degree $r$.
Arakawa's work mentioned above gives a great deal of information about this category. In the final \cref{s10-reps}, we reinterpret his results in terms of the algebra $\aS(n,r)$. In particular, in \cref{laugh}, we classify irreducible representations of $\aS(n,r)$; they are naturally indexed by $n$-tuples
$\lambda(u) = (\lambda_1(u),\lambda_2(u),\dots,\lambda_n(u))$ of monic polynomials in $\C[u]$ whose degrees sum to $r$ with
$$
\lambda_n(u)\ \big|\ \lambda_{n-1}(u)\ \big|\ \cdots\ \big|\ \lambda_1(u).
$$
It seems reasonable to hope that sequences of monic polynomials of this form also parametrize irreducible representations of $\aS(n,r)$ over algebraically closed fields of positive characteristic.
We point out also that a $q$-analog of these sequences
can be seen in the classification of irreducible polynomial representations of $U_q\widehat{\mathfrak{gl}}_n$ obtained by Frenkel and Mukhin in \cite[Sec.~4.3]{FM}.

\vspace{2mm}
\noindent
{\em Conventions.} 
In the remainder of the article, we work over a commutative ground ring $\kk$.
We are mainly interested in the case that $\kk$ is an algebraically closed field of characteristic 0,
but most of the constructions make sense more generally.
We use $\otimes$ for tensor product over $\kk$.

\vspace{2mm}
\noindent
{\em Acknowledgements.} The first author would like to thank Steve Doty for helpful discussions.

%% file: s2-reminders.tex
\setcounter{section}{1}

\section{Reminders about double cosets and Schur algebras}\label{s2-reminders}

A {\em composition} $\lambda = (\lambda_1,\dots,\lambda_n)$ 
of $r$
is a finite sequence of natural numbers (including 0) whose sum is $r$. Its {\em length}
$\ell(\lambda)$ is the number $n$ of parts, and 
$|\lambda|$ denotes $r=\lambda_1+\cdots+\lambda_n$.
Another useful shorthand: $\lambda_{< i}$ denotes $\lambda_1+\cdots+\lambda_{i-1}$ and $\lambda_{\leq i} := \lambda_{< i}+\lambda_i$.
We also adopt the following notation:
\begin{itemize}
\item
Let $\Lambda(n,r)$ be the set
of all compositions of $r$ of length $n$.
\item
Let $\Lambda(n) := \N^n = \coprod_{r \geq 0} \Lambda(n,r)$
be the set of all compositions with $n$ parts.
\item
Let $X(n)$ be the Abelian group $\Z^n$. It contains $\Lambda(n)$ as a sub-monoid.
\end{itemize}
We use $\eps_i$ to denote the element of $X(n)$ that has $1$ in its $i$th entry and 0 in all other positions, and $\alpha_i := \eps_i-\eps_{i+1}$.
This notation depends implicitly on the value of $n$, but we do not think it will cause confusion subsequently.

We denote the symmetric group acting on the left on 
$\{1,\dots,r\}$ by $\S_r$. It is generated by the basic transpositions $s_i := (i\:\:i\!+\!1)$ for $i=1,\dots,r-1$. Let $\ell:\S_r \rightarrow \N$ be the usual length function,
and $\leq$ be the Bruhat order.
For $\lambda\in \Lambda(n,r)$, we write
$\S_\lambda$ for the parabolic subgroup $\S_{\lambda_1}\times\cdots\times \S_{\lambda_n}$ of  $\S_r$.
Given also $\mu \in \Lambda(m,r)$,
let 
$(\S_\lambda \backslash \S_r)_{\min}$, $(\S_r / \S_\mu)_{\min}$
and 
\begin{equation}
(\S_\lambda \backslash \S_r / \S_\mu)_{\min} =
(\S_\lambda \backslash \S_r)_{\min}\cap (\S_r / \S_\mu)_{\min}
\end{equation}
be the sets of minimal length right, left and double coset representatives.

Let $\Mat{\lambda}{\mu}$ be the set
of $\ell(\lambda)\times\ell(\mu)$-matrices with entries in $\N$ whose row sums are the parts of $\lambda$ and whose column sums are the parts of $\mu$. 
An element $A \in \Mat{\lambda}{\mu}$ can be visualized by means of its {\em double coset diagram}, so-called because it gives rise to a well-known bijection 
\begin{equation}
\Mat{\lambda}{\mu} \stackrel{\sim}{\rightarrow}
(\S_\lambda \backslash \S_r / \S_\mu)_{\min},
\qquad
A \mapsto d_A.
\end{equation} 
We give an example in lieu of the formal definition:
\begin{equation}\label{myeg}
\begin{tikzpicture}[anchorbase,scale=1.55]
\draw[-,line width=.6mm] (.212,.5) to (.212,.39);
\draw[-,line width=.75mm] (.595,.5) to (.595,.39);
\draw[-,line width=.15mm] (0.0005,-.396) to (.2,.4);
\draw[-,line width=.3mm] (0.01,-.4) to (.59,.4);
\draw[-,line width=.3mm] (.4,-.4) to (.607,.4);
\draw[-,line width=.45mm] (.79,-.4) to (.214,.4);
\draw[-,line width=.15mm] (.8035,-.398) to (.614,.4);
\draw[-,line width=.3mm] (.4006,-.5) to (.4006,-.395);
\draw[-,line width=.6mm] (.788,-.5) to (.788,-.395);
\draw[-,line width=.45mm] (0.011,-.5) to (0.011,-.395);
\draw[-] (.2,.4) to (0.395,-.4);
\strand{.22,.12}{0};
\strand{.05,.05}{1};
\strand{.76,.05}{1};
\strand{.35,.35}{3};
\strand{.2,-.26}{2};
\strand{.5,-.26}{2};
\end{tikzpicture}
\leftrightarrow
\begin{tikzpicture}[centerzero=.6,scale=1.2]
\draw[ultra thick,cyan] (-.01,-.02) to (0.41,-.02);
\strand{.2,-.13}{3};
\draw[ultra thick,cyan] (.59,-.02) to (0.81,-.02);
\strand{.7,-.13}{2};
\draw[ultra thick,cyan] (.99,-.02) to (1.61,-.02);
\strand{1.3,-.13}{4};
\draw[ultra thick,cyan] (-.01,1.02) to (0.61,1.02);
\strand{.3,1.13}{4};
\draw[ultra thick,cyan] (.79,1.02) to (1.61,1.02);
\strand{1.2,1.13}{5};
\draw[-] (0,0) to (0,1);
\draw[-] (0.2,0) to (.8,1);
\draw[-] (0.4,0) to (1,1);
\draw[-] (0.6,0) to (1.2,1);
\draw[-] (.8,0) to (1.4,1);
\draw[-] (1,0) to (.2,1);
\draw[-] (1.2,0) to (.4,1);
\draw[-] (1.4,0) to (.6,1);
\draw[-] (1.6,0) to (1.6,1);
\end{tikzpicture}\leftrightarrow
d_A =(2\:5\:8\:4\:7\:3\:6)
\in (\S_{(4,5)} \backslash \S_9 / \S_{(3,2,4)})_{\min}
\leftrightarrow
A=\begin{bmatrix} 1&0&3\\2&2&1\end{bmatrix}.
\end{equation}
Here, $\lambda = (4,5)$ and $\mu = (3,2,4)$, these being the row and column sums of the matrix $A$.
The double coset diagram for $A$ is the diagram on the left hand side. It has strings at the top of thickness given by the parts of $\lambda$, and strings at the bottom of thickness given by the parts of $\mu$. These strings split into thinner {\em propagating strings}, with the one joining the $i$th string at the top  to the $j$th string at the bottom being of thickness $a_{i,j}$.
The minimal length double coset representative $d_A$ indexed by this matrix may be obtained by expanding the thick strings in the double coset diagram into parallel thin strings, then reading off the permutation encoded by the resulting string diagram.

Generally, in string diagrams,
we use a dotted line without a thickness label as a shorthand for a string of thickness 0, and we use a thin solid line without a thickness label to denote a string of thickness 1.
In fact, it is usually harmless to simply omit propagating
strings of thickness zero from diagrams. With these conventions, we have that
\begin{equation}\label{conventions}
\begin{tikzpicture}[anchorbase,scale=1.55]
\draw[-,line width=.6mm] (.212,.5) to (.212,.39);
\draw[-,line width=.75mm] (.595,.5) to (.595,.39);
\draw[-,line width=.15mm] (0.0005,-.396) to (.2,.4);
\draw[-,line width=.3mm] (0.01,-.4) to (.59,.4);
\draw[-,line width=.3mm] (.4,-.4) to (.607,.4);
\draw[-,line width=.45mm] (.79,-.4) to (.214,.4);
\draw[-,line width=.15mm] (.8035,-.398) to (.614,.4);
\draw[-,line width=.3mm] (.4006,-.5) to (.4006,-.395);
\draw[-,line width=.6mm] (.788,-.5) to (.788,-.395);
\draw[-,line width=.45mm] (0.011,-.5) to (0.011,-.395);
\draw[-] (.2,.4) to (0.395,-.4);
\strand{.22,.12}{0};
\strand{.05,.05}{1};
\strand{.76,.05}{1};
\strand{.35,.35}{3};
\strand{.2,-.26}{2};
\strand{.5,-.26}{2};
\end{tikzpicture}
\ =\ \begin{tikzpicture}[anchorbase,scale=1.55]
\draw[-,line width=.6mm] (.212,.5) to (.212,.39);
\draw[-,line width=.75mm] (.595,.5) to (.595,.39);
\draw[-,line width=.15mm] (0.0005,-.396) to (.2,.4);
\draw[-,line width=.3mm] (0.01,-.4) to (.59,.4);
\draw[-,line width=.3mm] (.4,-.4) to (.607,.4);
\draw[-,line width=.45mm] (.79,-.4) to (.214,.4);
\draw[-,line width=.15mm] (.8035,-.398) to (.614,.4);
\draw[-,line width=.3mm] (.4006,-.5) to (.4006,-.395);
\draw[-,line width=.6mm] (.788,-.5) to (.788,-.395);
\draw[-,line width=.45mm] (0.011,-.5) to (0.011,-.395);
\draw[zeroline] (.2,.4) to (0.395,-.4);
\strand{.35,.35}{3};
\strand{.2,-.26}{2};
\strand{.5,-.26}{2};
\end{tikzpicture}\ =\ 
\begin{tikzpicture}[anchorbase,scale=1.55]
\draw[-,line width=.6mm] (.212,.5) to (.212,.39);
\draw[-,line width=.75mm] (.595,.5) to (.595,.39);
\draw[-,line width=.15mm] (0.0005,-.396) to (.2,.4);
\draw[-,line width=.3mm] (0.01,-.4) to (.59,.4);
\draw[-,line width=.3mm] (.4,-.4) to (.607,.4);
\draw[-,line width=.45mm] (.79,-.4) to (.214,.4);
\draw[-,line width=.15mm] (.8035,-.398) to (.614,.4);
\draw[-,line width=.3mm] (.4006,-.5) to (.4006,-.395);
\draw[-,line width=.6mm] (.788,-.5) to (.788,-.395);
\draw[-,line width=.45mm] (0.011,-.5) to (0.011,-.395);
\strand{.35,.35}{3};
\strand{.2,-.26}{2};
\strand{.5,-.26}{2};
\end{tikzpicture}\ .
\end{equation}

For $A \in \Mat{\lambda}{\mu}$, we define its {\em left redundancy}
$\lambda(A)$ and its {\em right redundancy} $\mu(A)$ to be the compositions obtained by reading the entries of the matrix in order along rows starting with the top row, or by reading the entries of the matrix in order down columns starting with the leftmost column, respectively.
In the example, $\lambda(A) = (1,0,3,2,2,1)$
and $\mu(A) = (1,2,0,2,3,1)$.
The parts of $\lambda(A)$ are the thicknesses of the propagating strings in the double coset diagram above all of the crossings, and the parts of $\mu(A)$ are their thicknesses below all of the crossings.
Also observe that $\S_{\lambda(A)} \leq \S_\lambda$ and $\S_{\mu(A)} \leq \S_\mu$.

The following lemma is fundamental. Parts (1) and (2) are formulated this way in \cite[Lem.~2.1]{BSchur} and proofs can be extracted from \cite[Lem.~1.6]{DJgl}. Part (3) is also well known; see \cite{BLM}.

\begin{lemma}\label{dylan}
Let $A,B \in \Mat{\lambda}{\mu}$.
\begin{enumerate}
\item
We have that $d_A \S_{\mu(A)} = \S_{\lambda(A)} d_A$.
The isomorphism
$\S_{\mu(A)} \stackrel{\sim}{\rightarrow} \S_{\lambda(A)},
w \mapsto d_A w d_A^{-1}$
preserves length and Bruhat order.
\item
Any element $w$ of the double coset $\S_\lambda d_A \S_\mu$
can be written as $w = x d_A y$ for unique elements
$x \in \S_\lambda$ and $y \in (\S_{\mu(A)} \backslash \S_\mu)_{\min}$, or as $w = x d_A y$
for unique elements $x \in (\S_\lambda / \S_{\lambda(A)})_{\min}$ and $y \in \S_\mu$.
In both situations, $\ell(w) = \ell(x)+\ell(d_A) + \ell(y)$.
\item
$\displaystyle
d_A \leq d_B \:\Leftrightarrow\:
\left(
\sum_{i=1}^s \sum_{j=1}^t  a_{i,j} \geq \sum_{i=1}^s \sum_{j=1}^t b_{i,j}
\text{ for all $1 \leq s \leq \ell(\lambda)$ 
and $1 \leq t \leq \ell(\mu)$}\right).
$
\end{enumerate}
\end{lemma}

The double coset combinatorics just described is used classically in the construction of the {\em Schur algebra}; e.g., see \cite{Green}.
To set some notation, we briefly recall one of the many equivalent definitions of $\S(n,r)$: it is
the endomorphism algebra
\begin{equation}\label{schurdef}
\S(n,r) := \End_{\dash \kk \S_r}
\left(\bigoplus_{\lambda \in \Lambda(n,r)} M(\lambda)\right)
\end{equation}
where $M(\lambda)$ is the (right) {\em permutation module}
$\kk_\lambda \otimes_{\kk \S_\lambda} \kk \S_r$
induced from the trivial right $\kk \S_\lambda$-module $\kk_\lambda$. Denoting the vector $1 \otimes 1 \in M(\lambda)$ by $\v_\lambda$, $M(\lambda)$ has the standard basis
$\left\{\v_\lambda x\:\big|\:x \in (\S_\lambda \backslash \S_r)_{\min}\right\}$.
Denoting the idempotent in $\S(n,r)$ defined by the projection onto $M(\lambda)$ by $1_\lambda$, we have that
$$
1_\lambda \S(n,r) 1_\mu = 
\Hom_{\dash \kk \S_r}(M(\mu), M(\lambda)).
$$
This is a free $\kk$-module with basis $\{\xi_A\:|\:A \in \Mat{\lambda}{\mu}\}$ in which $\xi_A$ is the unique $\kk \S_r$-module homomorphism
\begin{equation}\label{standardbasis}
\xi_A:
M(\mu) \rightarrow M(\lambda),
\qquad
\v_\mu \mapsto 
\sum_{y \in (\S_{\mu(A)} \backslash \S_\mu)_{\min}}
\v_\lambda d_A y.
\end{equation}
To see that such a homomorphism exists, it suffices to show that 
$\sum_{y \in (\S_{\mu(A)} \backslash \S_\mu)_{\min}}
\v_\lambda d_A y$ is invariant under right multiplication by any simple reflection $s_i \in \S_\mu$, which is easily checked; see the proof of \cref{bits}.
Note also that $1_\lambda =\xi_{\diag(\lambda_1,\dots,\lambda_n)}$.

To make the connection between the Schur algebra and the general linear group, let $\G$ be the group scheme $\GL_n$ over $\kk$,
$V$ be its natural representation with standard basis
$v_1,\dots,v_n$, and $\gl_n$ be its Lie algebra.
Let $I(n,r)$ denote the set of {\em multi-indices}
$\bi = (i_1,\dots,i_r)$
with $1 \leq i_1,\dots,i_r \leq n$.
This set indexes the obvious basis of the tensor space
$V^{\otimes r}$ consisting of the monomials
$v_\bi := v_{i_1}\otimes\cdots\otimes v_{i_r}$.
Tensor space 
is a $(\G,\kk \S_r)$-bimodule with $w \in \S_r$
acting by permuting tensors, i.e.,
\begin{equation}\label{permtensoraction}
v_\bi\; w := v_{\bi \cdot w}\qquad\text{
where}\qquad\text{$\bi \cdot w := (i_{w(1)},\dots,i_{w(r)})$}.
\end{equation}
Let $\T$ be the maximal torus of diagonal matrices in $\G$, identifying its character group with $X(n)$
so that $\eps_i$ is the character $\diag(t_1,\dots,t_n) \mapsto t_i$.
The vector $v_\bi$
is of weight $\eps_{i_1}+\cdots+\eps_{i_r} \in \Lambda(n,r)$.
We may also refer to this as the weight of the multi-index $\bi$.

The right $\kk \S_r$-module
$\bigoplus_{\lambda \in \Lambda(n,r)} M(\lambda)$
appearing in \cref{schurdef} may be identified with $V^{\otimes r}$ so that $\v_\lambda \in M(\lambda)$ corresponds to the tensor $v_{\bi^\lambda}$
indexed by
\begin{equation}\label{bilambda}
\bi^\lambda := (1^{\lambda_1},2^{\lambda_2},\dots,n^{\lambda_n}),
\end{equation}
this being the unique increasing multi-index of weight $\lambda$.
With this identification, we have that
\begin{equation}
\S(n,r) = \End_{\dash \kk \S_r}(V^{\otimes r}).
\end{equation}
We then have for any $A \in \bigcup_{\lambda,\mu \in \Lambda(n,r)} \Mat{\lambda}{\mu}$
and $\bj \in I(n,r)$ that
\begin{equation}\label{schurformula}
\xi_A v_\bj = \sum_{\substack{\bi \in I(n,r)\text{ such that}\\ a_{i,j} = \big|\{k=1,\dots,n\:|\:i_k = i, j_k =j\}\big|\\\text{for }i,j=1,\dots,n}}
v_\bi.
\end{equation}
This formula originates with Schur: it shows that $\xi_A$
is a sum of matrix units over an $\S_r$-orbit on $I(n,r)\times I(n,r)$.
For example, by Schur's formula, we have that
\begin{equation}\label{fooly}
\xi_{\diag(\mu_1,\dots,\mu_n)+e_{i,j}-e_{j,j}} v_\bj = \sum_{\substack{1 \leq p \leq r \\ j_p=j}}
v_{j_1}\otimes \cdots \otimes v_{j_{p-1}}\otimes v_i \otimes v_{j_{p+1}}\otimes\cdots\otimes v_{j_r}.
\end{equation}
for $1 \leq i,j\leq n$ with $i \neq j$, 
$\mu \in \Lambda(n,r)$ with $\mu_j > 0$, and
$\bj \in I(n,r)$ of weight $\mu$.
The notation $\diag(\mu_1,\dots,\mu_n)+e_{i,j}-e_{j,j}$ in \cref{fooly}
denotes the $n\times n$ matrix obtained from the diagonal matrix
$\diag(\mu_1,\dots,\mu_n)$ by adding 1 to the $(i,j)$-entry and subtracting 1 from the $(j,j)$-entry.
Its double coset diagram has $n$ vertical strings of thicknesses
$\mu_1,\dots,\mu_n$ at the bottom, plus a diagonal string of thickness 1 connecting the top of the $i$th string to the bottom of $j$th string:
\begin{equation}\label{muscles}
\begin{tikzpicture}[centerzero,scale=.9]
\draw[ultra thick] (-1.5,-.6)\botlabel{\mu_1} to (-1.5,.6);
\node at (-1.08,0) {$\dots$};
\node at (.6,0) {$\dots$};
\node at (-.25,0) {$\dots$};
\draw[ultra thick] (-.7,-.6)\botlabel{\mu_i} to (-.7,.6);
\draw[ultra thick] (.15,-.6) to (.15,.6);
\draw[ultra thick] (1,-.6)\botlabel{\mu_j} to (1,.6);
\draw (.97,-.6) to (.97,-.3) to (-0.67,.3) to (-.67,.6);
\node at (1.42,0) {$\dots$};
\draw[ultra thick] (1.8,-.6)\botlabel{\mu_n} to (1.8,.6);
\end{tikzpicture}
\quad\text{if $i < j$,}
\qquad\qquad
\begin{tikzpicture}[centerzero,scale=.9]
\draw[ultra thick] (-1.5,-.6)\botlabel{\mu_1} to (-1.5,.6);
\node at (-1.08,0) {$\dots$};
\node at (.6,0) {$\dots$};
\node at (-.25,0) {$\dots$};
\draw[ultra thick] (-.7,-.6)\botlabel{\mu_j} to (-.7,.6);
\draw[ultra thick] (.15,-.6) to (.15,.6);
\draw[ultra thick] (1,-.6)\botlabel{\mu_i} to (1,.6);
\draw (-.67,-.6) to (-.67,-.3) to (0.97,.3) to (.97,.6);
\node at (1.42,0) {$\dots$};
\draw[ultra thick] (1.8,-.6)\botlabel{\mu_n} to (1.8,.6);
\end{tikzpicture}
\quad\text{if $i > j$.}
\end{equation}

The derived action of $\gl_n$ on $V^{\otimes r}$ induces an algebra homomorphism 
\begin{equation}\label{notdrinfeld}
\drinfeld_{n,r}:
\Un \twoheadrightarrow \S(n,r),
\end{equation}
where $\Un$ denotes the universal enveloping algebra of
$\gl_n$.
By \cref{fooly}, the image of the matrix unit $e_{i,j} \in \gl_n$
is
\begin{equation}
\begin{dcases}
\sum_{\lambda \in \Lambda(n,r)} 
\lambda_i 1_\lambda&\text{if $i=j$}\\
\sum_{\substack{\mu \in \Lambda(n,r) \\ \mu_j > 0}}\xi_{\diag(\mu_1,\dots,\mu_n)+e_{i,j}-e_{j,j}}&\text{if $i \neq j$.}
\end{dcases}\label{folly}
\end{equation}
The double coset diagram
for $\xi_{\diag(\mu_1,\dots,\mu_n)+e_{i,j}-e_{j,j}}\:(i \neq j)$
is as displayed in \cref{muscles}.

When $\kk$ is a field of characteristic 0, it is well known that 
$\drinfeld_{n,r}$ is surjective. 
In \cite{doty-giaquinto}, Doty and Giaquinto also determined the kernel of $\drinfeld_{n,r}$ explicitly, thereby giving a Serre-type presentation for the (semisimple!) algebra $\S(n,r)$ over a field of characteristic 0. 
In the next paragraph, we reformulate their result in a way that
is relevant for a construction in \cref{s9-presenting}.
(For other ground rings, $\drinfeld_{n,r}$ need not be surjective, 
but the analogous statement with $\Un$ replaced
by the algebra of distributions $\Dist(\G)$ 
always holds, 
as does the result of Doty and Giaquinto with appropriate modifications; see \cite{doty} which proves an even more general result.)

The adjoint action of $\T$ on $\Un$ defines a weight
decomposition
$\Un = \bigoplus_{\alpha \in X(n)} \U_\alpha$,
with the $\alpha$-weight space $\U_\alpha$ being $\{0\}$
unless $\alpha$ is in the root lattice. 
Fixing $r \geq 0$, let
\begin{equation}\label{lotsofrain}
\KK := \bigoplus_{\lambda \in \Lambda(n,r)} \kk 1_\lambda
\end{equation}
be the direct sum of copies of $\kk$ indexed by the set $\Lambda(n,r)$, so $\{1_\lambda\:|\:\lambda \in \Lambda(n,r)\}$ are mutually orthogonal idempotents whose sum is the identity in $\KK$. We view 
\begin{equation}
\widetilde{\U}_{n,r} := \bigoplus_{\lambda,\mu \in \Lambda(n,r)}
\U_{\lambda-\mu}
\end{equation}
as a $(\KK, \KK)$-bimodule so
that $1_\lambda a 1_\mu$ is the projection $a_{\lambda,\mu}$
of $a = \sum_{\lambda,\mu \in \Lambda(n,r)} a_{\lambda,\mu} 
\in \widetilde{U}_{n,r}$ onto the $(\lambda,\mu)$th summand.
Then we define $\U_{n,r}$ to be the quotient of the tensor algebra 
\begin{equation}\label{tellies}
T_\KK(\widetilde{\U}_{n,r})
= \KK \ \oplus\  \widetilde{\U}_{n,r} \ \oplus\  \widetilde{\U}_{n,r} \otimes_{\KK} \widetilde{\U}_{n,r}\  \oplus\  \widetilde{\U}_{n,r} \otimes_{\KK} \widetilde{\U}_{n,r} \otimes_{\KK} \widetilde{\U}_{n,r}\  \oplus\  \cdots
\end{equation}
by the two-sided ideal generated by the relations
\begin{align}\label{rellies}
1_\lambda a 1_\mu \otimes 1_\mu b 1_\nu
&= 1_\lambda ab1_\nu,&
1_\lambda d_i 1_\lambda &= \lambda_i 1_\lambda,
\end{align}
for all $\lambda,\mu,\nu \in \Lambda(n,r), a \in \U_{\lambda-\mu}, b \in \U_{\mu-\nu}$ and $i=1,\dots,n$.
Equivalently, $\U_{n,r}$ is the quotient of Lusztig's modified form $\dot\U(\gl_n)$ by the two-sided ideal generated by the idempotents $1_\lambda$ for $\lambda \notin \Lambda(n,r)$.
We denote the image of an element $1_\lambda a 1_\mu$ of $\widetilde{\U}_{n,r}$ in $\U_{n,r}$ by $1_\lambda \bar a 1_\mu$.
The main result of \cite{doty-giaquinto} can be reformulated as follows:

\begin{theorem}[Doty-Giaquinto]\label{DGthm}
When $\kk$ is a field of characteristic 0, 
there is an algebra isomorphism
\begin{align}\label{dg}
\bar\drinfeld_{n,r}:
\U_{n,r} &\stackrel{\sim}{\rightarrow}\S(n,r),
&
1_\lambda \bar a 1_\mu&\mapsto1_\lambda \drinfeld_{n,r}(a) 1_\mu
\end{align}
for $\lambda,\mu \in \Lambda(n,r)$
and $a \in \U_{\lambda-\mu}$.
\end{theorem}

\begin{remark}
Surjectivity of $\drinfeld_{n,r}$ follows from injectivity of the natural map $\kk[\G] \rightarrow \Dist(\G)^*$, which is a consequence of Krull's intersection theorem.
Assuming this, \cref{DGthm} is equivalent to the statement that 
a left $\Un$-module is a polynomial representation of $\GL_n$ of degree $r$ if and only if it is a
weight module with all weights belonging to $\Lambda(n,r)$. 
\end{remark}

%% file: s3-basis.tex
\setcounter{section}{2}

\section{The degenerate affine Schur algebra as an endomorphism algebra}\label{s3-basis}

The {\em degenerate affine Hecke algebra} $\aH_r$ is the $\kk$-algebra with
generators $x_1,\dots,x_r$ and $s_1,\dots,s_{r-1}$
subject to the following relations. The generators $x_1,\dots,x_r$ commute with each other,
the generators $s_1,\dots,s_{r-1}$ satisfy the usual Coxeter 
relations of the basic transpositions in the symmetric group $\S_r$,
and\footnote{This is different from the defining relation
$s_i x_{i} = x_{i+1}s_i-1$ used
in \cite{Kleshchev} but mapping $x_i \mapsto -x_i$ gives an isomorphism between the two versions.}
\begin{align}\label{daharels}
s_i x_{i} &= x_{i+1}s_i+1,
&
x_{i} s_i &= s_i x_{i+1}+1,
&
x_i s_j &= s_j x_i\text{ if $i \neq j,j+1$.}
\end{align}
The definition of $\aH_r$ makes sense even 
if $r=0$, when it is $\kk$.
Letting $\P_r$ be the polynomial algebra $\kk[x_1,\dots,x_r]$,
the linear map 
$\kk \S_r \otimes \P_r \rightarrow \aH_r$ 
defined by multiplication is a $\kk$-module isomorphism.
We will simply identify $\kk \S_r$ and $\P_r$ with subalgebras of $\aH_r$ from now on.
For $w \in \S_r$ and $f \in \P_r$, we use the notation $w(f)$
to denote the image of $f$ under the usual action of $\S_r$ on $\P_r$ permuting the generators.
For any $f \in \P_r$, we have in $\aH_r$ that
\begin{align}\label{poof}
s_i f &= s_i(f) s_i + \partial_i(f),
&
f s_i &= s_i\ s_i(f) + \partial_i(f)
\end{align}
where $\partial_i$ is the {\em Demazure operator} defined from
\begin{equation}\label{demazureoperator}
\partial_i(f) := \frac{f - s_i(f)}{x_i-x_{i+1}}.
\end{equation}
The 
{\em left polynomial representation} of $\aH_r$ is the left $\aH_r$-module\footnote{Secretly, it is $\aH_r \otimes_{\kk \S_r} \kk$ for the trivial action of $\S_r$ on $\kk$.}
$\P_r$ with $x_i$ acting by multiplication and $\S_r$ acting by $\bully$ defined so that
\begin{equation}\label{polyaction}
s_i \bully f := s_i(f) + \partial_i(f).
\end{equation}
Note also that the {\em center} $Z(\aH_r)$ is the subalgebra 
\begin{equation}\label{bernsteincenter}
\P^{(r)}
:= \P_r^{S_r} = 
\left\{f \in \P_r\:\big|\:w(f) = f\text{ for all }w \in \S_r\right\}
=
\left\{f \in \P_r\:\big|\:w\bully f = 
f\text{ for all }w \in \S_r\right\}
\end{equation}
of $\P_r$ consisting of symmetric polynomials. 
More generally, for $\lambda \in \Lambda(n,r)$, there is the
parabolic subalgebra $\aH_\lambda$ of $\aH_r$, which is the image of 
$\kk \S_\lambda \otimes \P_r$ under the multiplication map.
The center of $\aH_\lambda$ is
\begin{equation}\label{Plambda}
\P^\lambda := \P_r^{S_\lambda} = 
\left\{f \in \P_r\:\big|\:w(f) = f\text{ for all }w \in \S_\lambda\right\}=
\left\{f \in \P_r\:\big|\:w\bully f = f\text{ for all }w \in \S_\lambda\right\}.
\end{equation}
For proofs of these basic results and further 
background, see \cite{Kleshchev}.

\begin{definition}
For $\lambda \in \Lambda(n,r)$, the induced module
$M(\lambda)\otimes_{\kk \S_r} \aH_r$ 
is a cyclic $\aH_r$-module 
generated by the vector $\v_\lambda \otimes 1$.
The {\em degenerate 
affine Schur algebra} $\aS(n,r)$ is the endomorphism algebra
\begin{equation}\label{affdef}
\aS(n,r) := \End_{\dash \aH_r}\left(
\bigoplus_{\lambda \in \Lambda(n,r)} M(\lambda)\otimes_{\kk \S_r} \aH_r
\right).
\end{equation}
Like for the Schur algebra,
there are distinguished idempotents
$1_\lambda \in \aS(n,r)$ for each $\lambda \in \Lambda(n,r)$ defined by the evident 
projections onto the summands. 
\end{definition}

\begin{remark}
The degenerate affine Schur algebra has not received so much attention in the literature, but there is also the 
{\em affine $q$-Schur algebra} $\qS(n,r)$, which may be constructed in a similar way replacing the degenerate affine Hecke algebra
with the actual affine Hecke algebra. 
The affine $q$-Schur algebra has been thoroughly studied; e.g., see \cite{Green-affine,vigneras,du-fu,MS}.
Specializing $q$ to $1$ in the affine $q$-Schur algebra produces also the {\em affine Schur algebra} $\oneS(n,r)$, which was introduced in \cite[Sec.~3]{green-doty} and is 
different from the degenerate affine Schur algebra $\aS(n,r)$; see \cite[Sec.~2.4]{Antor} for a clear exposition when over a field of characteristic 0. 
\end{remark}

The quantum analog of the following theorem is 
proved in \cite[Th.~2.2.4]{Green-affine} and 
\cite[4.2.13]{vigneras}; see \cite[Sec.~2.5]{Antor} where the definition is explained in terms of the Bernstein presentation.
In the degenerate case, we regard the result 
as folklore. There are several proofs in the recent literature; e.g., see \cite[Prop.~7.5]{LM} or \cite[Cor.~11.3.2]{DKMZ2} which prove more general results, both of which include the result needed here as a special case.
We include a self-contained proof based on an application of the Mackey theorem.

\begin{theorem}\label{bits}
For $\lambda,\mu \in \Lambda(n,r)$, the $\kk$-module
$$1_\lambda \aS(n,r) 1_\mu
= \Hom_{\dash \aH_r}\big(M(\mu)\otimes_{\kk \S_r} \aH_r ,M(\lambda)\otimes_{\kk \S_r} \aH_r\big)
$$
is free with an explicit basis $\{\xi_{A,f}\}$ indexed by pairs $(A,f)$ as $A$ runs over the set $\Mat{\lambda}{\mu}$ and $f$ runs over a basis for
$\P^{{\mu(A)}}$.
By definition, $\xi_{A,f}$ is the unique right $\aH_r$-module homomorphism
\begin{align}\label{inbits}
\xi_{A,f}:M(\mu)\otimes_{\kk \S_r} \aH_r &\rightarrow M(\lambda)\otimes_{\kk \S_r} \aH_r,
&
\v_\mu\otimes 1 &\mapsto \sum_{y \in (\S_{\mu(A)} \backslash \S_\mu)_{\min}} \v_\lambda \otimes d_A f y.
\end{align}
\end{theorem}

\begin{proof}
By transitivity of induction,
$M(\lambda)\otimes_{\kk \S_r}\aH_r \cong \kk_\lambda \otimes_{\kk \S_\lambda} \aH_r$,
with $m_\lambda \otimes 1$ in the left hand module
corresponding to $1 \otimes 1$ on the right.
Using this description, to show that there is a well-defined such homomorphism $\xi_{A,f}:M(\mu)\otimes_{\kk \S_r}\aH_r\rightarrow M(\lambda)\otimes_{\kk \S_r}\aH_r$, it suffices to show that the vector
\begin{equation}\label{thisv}
\sum_{y \in (\S_{\mu(A)} \backslash \S_\mu)_{\min}} \v_\lambda \otimes d_A f y
\end{equation}
is invariant under right multiplication by a simple reflection $s_i \in \S_\mu$.
For $y \in (\S_{\mu(A)} \backslash \S_\mu)_{\min}$, \cite[Lem.~1.1]{DJgl} shows that 
{\em either} $y s_i \in 
(\S_{\mu(A)} \backslash \S_\mu)_{\min}$
{\em or} $y s_i = s_j y$ for a basic transposition $s_j \in \S_{\mu(A)}$.
In the latter case,
\cref{dylan}(1) implies that
$d_A s_j = s_k d_A$ for $s_k \in \S_{\lambda(A)}$,
so that 
$$
\v_\lambda \otimes d_A f y s_i
= 
\v_\lambda \otimes d_A f s_j
y
= 
\v_\lambda \otimes d_A s_j f y
= 
\v_\lambda \otimes s_k d_A f y
=  \v_\lambda s_k \otimes d_A f y
= \v_\lambda \otimes d_A f y.
$$
Since right multiplication by $s_i$ permutes left $\S_\lambda$-cosets in the double coset $\S_\lambda d_A \S_\mu$, we deduce that right multiplication by $s_i$ permutes the summands of \cref{thisv}, thereby fixing the sum itself.

We can view $\P_r$ as a right $\aH_r$-module---the {\em right} polynomial representation---by identifying it with 
$\kk \otimes_{\kk \S_r} \aH_r$.
Let $\P_\lambda$ denote 
the restriction of this to a right $\aH_\lambda$-module.
Since $\P_\lambda \cong \kk_\lambda \otimes_{\kk \S_\lambda} \aH_\lambda$, transitivity of induction implies that
$M(\lambda)\otimes_{\kk \S_r}\aH_r \cong \kk_\lambda \otimes_{\kk \S_\lambda} \aH_r
\cong \P_\lambda \otimes_{\aH_\lambda} \aH_r$,
the natural isomorphism taking $\v_\lambda\otimes 1$ to $1 \otimes 1$.
By Frobenius reciprocity, we have that
\begin{align*}
\Hom_{\dash \aH_r}\big(M(\mu)\otimes_{\kk \S_r}\aH_r, M(\lambda)\otimes_{\kk \S_r}\aH_r
\big)
&\cong
\Hom_{\dash \aH_r}\left(\P_\mu \otimes_{\aH_\mu} \aH_r, \P_\lambda \otimes_{\aH_\lambda} \aH_r\right)\\
&\cong
\Hom_{\dash \aH_\mu}\left(\P_\mu, \P_\lambda \otimes_{\aH_\lambda} \aH_r
\downarrow_{\aH_\mu}\right).
\end{align*}
Under these isomorphisms, $\xi_{A,f}$ maps to the unique right $\aH_\mu$-module homomorphism 
\begin{align*}
\xi_{A,f}':\P_\mu &\rightarrow \P_\lambda \otimes_{\aH_\lambda} \aH_r \downarrow_{\aH_\mu},
&
1 &\mapsto \sum_{y \in (\S_{\mu(A)} \backslash \S_\mu)_{\min}} 1 \otimes d_A f y.
\end{align*}

Now we recall the Mackey theorem for degenerate affine Hecke algebras; e.g., see \cite[Th.~3.5.2]{Kleshchev}.
Enumerate the elements of $\Mat{\lambda}{\mu}$
as $A_1,\dots,A_n$ so that $d_{A_i} < d_{A_j}$ in the Bruhat order implies $i < j$.
For $0 \leq m \leq n$, 
let $V_m$ be the $\aH_\mu$-submodule of
$\P_\lambda \otimes_{\aH_\lambda} \aH_r$ generated by
$1 \otimes d_{A_1},\dots,1 \otimes d_{A_m}$.
This defines a filtration
$$
\{0\} = V_0 \leq V_1 \leq \cdots \leq V_n = \P_\lambda \otimes_{\aH_\lambda} \aH_r \downarrow_{\aH_\mu}.
$$
Then the Mackey theorem implies that
$V_m / V_{m-1} \cong \P_{\mu(A_m)} \otimes_{\aH_{\mu(A_m)}}
\aH_\mu$ as a right $\aH_\mu$-module,
with an explicit isomorphism taking $1 \otimes d_{A_m} + V_{m-1}$
to $1 \otimes 1$.
We observe that $\xi_{A_m,f}'$ has image contained in
$V_m$.
We are going to show by induction on $m=0,1,\dots,n$ that
the homomorphisms 
$\xi_{A_{l},f}'$ for $1 \leq l \leq m$ and 
$f$ running over a basis of $\P^{{\mu(A_{l})}}$ 
give a basis for 
$\Hom_{\dash \aH_\mu}\left(\P_\mu, V_m\right)$ as a free $\kk$-module.
The $m=n$ case is sufficient to prove the theorem.

For the induction step, suppose that $1 \leq m \leq n$
and consider
$\Hom_{\dash \aH_\mu}\left(\P_\mu, V_{m}\right)$.
Applying $\Hom_{\dash \aH_\mu}(\P_\mu,-)$ to
$0 \rightarrow V_{m-1} \rightarrow V_m \rightarrow V_m / V_{m-1}
\rightarrow 0$ gives 
an exact sequence
$$
0 \longrightarrow 
\Hom_{\dash \aH_\mu}\left(\P_\mu, V_{m-1}\right)
\longrightarrow
\Hom_{\dash \aH_\mu}\left(\P_\mu, V_{m}\right)
\stackrel{\theta}{\longrightarrow}
\Hom_{\dash \aH_\mu}\left(\P_\mu, 
V_m / V_{m-1}\right).
$$
To check the induction step, it suffices to show that 
the homomorphisms $\xi_{A_m,f}'' := 
\theta(\xi'_{A_m,f})$ as $f$ runs over a basis
for $\P^{{\mu(A_m)}}$ give a basis for
$\Hom_{\dash \aH_\mu}(\P_\mu, V_m / V_{m-1})$ (this also shows that $\theta$ is surjective).
Using the isomorphism $V_m / V_{m-1} \cong \P_{\mu(A_m)} \otimes_{\aH_{\mu(A_m)}}
\aH_\mu$ from the Mackey theorem, 
this follows if we can show that the right $\aH_\mu$-module
homomorphisms 
\begin{align*}
\xi_{A_m,f}''' :
\P_\mu &\rightarrow 
\P_{\mu(A_m)} \otimes_{\aH_{\mu(A_m)}}
\aH_\mu,&
1 &\mapsto \sum_{y \in (\S_{\mu(A)} \backslash \S_\mu)_{\min}} 1 \otimes \end{align*}
give a basis for $\Hom_{\dash \aH_\mu}\big(\P_\mu, \P_{\mu(A_m)} \otimes_{\aH_{\mu(A_m)}} \aH_\mu\big)$ as $f$ runs over a basis for $\P^{{\mu(A_m)}}$.

Recall that $\mu(A_m)$ is obtained by reading the entries of the matrix $A_m$ down columns starting with the leftmost column.
Let $\nu(A_m)$ be the composition obtained by reading the entries of $A_m$ {\em up} columns starting with the leftmost column. 
Let $d$ be the
longest element of $(\S_{\mu(A_m)} \backslash \S_\mu / \S_{\nu(A_m)})_{\min}$. 
The double coset $\S_{\mu(A_m)} d \S_{\nu(A_m)}$ is special: it equals $\S_{\mu(A_m)} d = d \S_{\nu(A_m)}$.
By \cite[Cor.~3.7.3]{Kleshchev}, there is a unique isomorphism of right $\aH_\mu$-modules
$$
\P_{\mu(A_m)} \otimes_{\aH_{\mu(A_m)}} \aH_{\mu}
\stackrel{\sim}{\rightarrow}
\Hom_{\dash \aH_{\nu(A_m)}}\left(\aH_\mu, \P_{\nu(A_m)}\right)
$$
mapping $1 \otimes 1$ to the unique right $\aH_{\nu(A_m)}$-module homomorphism 
$\varphi:\aH_\mu \rightarrow \P_{\nu(A_m)}$
which maps 
$d' \in (\S_\mu / \S_{\nu(A_m)})_{\min}$ to $\delta_{d,d'}$.
This isomorphism induces the first of the following:
\begin{multline*}
\Hom_{\dash \aH_\mu}\big(\P_\mu, \P_{\mu(A_m)} \otimes_{\aH_{\mu(A_m)}} \aH_\mu\big)
\stackrel{\sim}{\rightarrow} \Hom_{\dash \aH_\mu}\big(\P_\mu, \Hom_{\dash \aH_{\nu(A_m)}}\big(\aH_\mu, \P_{\nu(A_m)}\big)\big)\\
\stackrel{\sim}{\rightarrow}  \Hom_{\dash \aH_{\nu(A_m)}}\big(\P_{\nu(A_m)},  \P_{\nu(A_m)}\big)
\stackrel{\sim}{\rightarrow} \P^{{\nu(A_m)}}
\stackrel{\sim}{\rightarrow} \P^{{\mu(A_m)}}.
\end{multline*}
The second of these isomorphisms is another Frobenius reciprocity, the third one is defined by evaluation at 1 using that $Z(\aH_{\nu(A_m)}) = \P^{{\nu(A_m)}}$,
and the last one is $f \mapsto d(f)$.
We claim that
the image of $\xi_{A_m,f}'''$ under this sequence of isomorphisms is simply $f$.
We are trying to show that the morphisms 
$\xi_{A_m,f}'''$ give a basis for 
$\Hom_{\dash \aH_\mu}\big(\P_\mu, \P_{\mu(A_m)} \otimes_{\aH_{\mu(A_m)}} \aH_\mu\big)$
as $f$ runs over a basis for $\P^{{\mu(A_m)}}$. 
This obviously follows from the claim.

Finally, to prove the claim, the image of $\xi_{A_m,f}'''$ under the first isomorphism is the unique right $\aH_\mu$-module homomorphism
mapping $1 \mapsto \sum_{y \in (\S_{\mu(A_m)} \backslash \S_\mu)_{\min}} \varphi f y$.
Applying the remaining three isomorphisms to this produces
$$
d\ \bigg(\sum_{y \in (\S_{\mu(A_m)} \backslash \S_\mu)_{\min}} (\varphi f y)(1)\bigg)=
\sum_{y \in (\S_{\mu(A_m)} \backslash \S_\mu)_{\min}} d(\varphi (f y)).
$$
We have that
$fy = \delta_{y,d}\  d \ d^{-1}(f)+(*)$ where $(*)$
is a sum of terms
of the form $z w g$ for $z \in (\S_\mu / \S_{\nu(A_m)})_{\min}$
with $\ell(z) < \ell(d)$,
$w \in \S_{\nu(A_m)}$ and $g \in \P_r$. The map $\varphi$ is zero on $(*)$ so this expression simplifies to give
$d(\varphi(d\  d^{-1}(f))) = d(d^{-1}(f)) = f$.
\end{proof}

We view $\P_r$ as a graded algebra so that each $x_i$
is of degree 1.
Then the smash product $\kk \S_r \circltimes \P_r$ is a graded algebra with permutations in $\S_r$ being of degree 0.
For $\lambda \in \Lambda(n,r)$, there is also a graded right $\kk S_r \circltimes \P_r$-module 
$M(\lambda) \circltimes \P_r$, which is the tensor product
$M(\lambda)\otimes \P_r$ with $\P_r$ acting by right multiplication and $w \in \S_r$ acting by
$(\v_\lambda y \otimes f) w = \v_\lambda
yw \otimes w^{-1}(f)$. We call the endomorphism algebra
\begin{equation}\label{current}
\gS(n,r) := \End_{\dash \kk \S_r \circltimes \P_r}\left(
\bigoplus_{\lambda \in \Lambda(n,r)}
M(\lambda) \circltimes \P_r
\right)
\end{equation}
the {\em current Schur algebra}. 

\begin{remark}
In \cref{current}, one can also replace the polynomial algebra $\P_r=\kk[x_1,\dots,x_r]$ with the algebra $\kk[x_1^{\pm 1},\dots,x_r^{\pm r}]$ of Laurent polynomials.
The resulting endomorphism algebra could be called the
``loop Schur algebra'' but it is just the same as the affine Schur algebra $\oneS(n,r)$ mentioned above, and we will continue to use this established terminology for it.
This coincidence is apparent from the 
exposition in \cite[Sec.~2.4]{Antor}.
\end{remark}

There is an
ascending filtration
$$
\{0\} = F_{-1}\, \aH_r \subseteq F_0\, \aH_r \subseteq F_1\, \aH_r \subseteq \cdots
$$
defined by letting $F_d\, \aH_r$ be the subspace spanned by
all $w f$ for polynomials $f \in \P_r$ of degree $\leq d$ 
and $w \in \S_r$. 
Thus, $x_1,\dots,x_r$ are in filtered degree 1, and permutations are of degree 0.
The associated graded algebra
$\gr \aH_r$ is identified with $\kk \S_r 
\circltimes
\P_r$
so that $\gr_0 s_i = s_i \otimes 1$ and 
$\gr_1 x_i = 1 \otimes x_i$.
There is an induced filtration making 
$M(\lambda)\otimes_{\kk \S_r}\aH_r$ into a filtered right $\aH_r$-module, with
$F_d\, M(\lambda)\otimes_{\kk \S_r}\aH_r := \v_\lambda \otimes (F_d \,\aH_r)$.
The associated graded right $\gr \aH_r$-module
$\gr\left( M(\lambda)\otimes_{\kk \S_r}\aH_r\right)$ is identified with
the graded right $\kk \S_r \circltimes \P_r$-module
$M(\lambda) \circltimes \P_r$, that is, the tensor product
$M(\lambda)\otimes \P_r$ with $\P_r$ acting by right multiplication and $w \in \S_r$ acting by
$(\v_\lambda y \otimes f) w = \v_\lambda
yw \otimes w^{-1}(f)$.
Finally, there is an ascending filtration 
$$
\{0\} = F_{-1}\,\aS(n,r)
\subseteq F_0\,\aS(n,r)
\subseteq F_1\,\aS(n,r)
\subseteq \cdots
$$
on $\aS(n,r)$
with $F_d (1_\lambda \aS(n,r) 1_\mu)$
being all homomorphisms $M(\mu)\otimes_{\kk \S_r}\aH_r\rightarrow M(\lambda)\otimes_{\kk \S_r}\aH_r$
which take the subspace $F_i\, M(\mu)\otimes_{\kk \S_r}\aH_r$ into $F_{d+i}\, M(\lambda)\otimes_{\kk \S_r}\aH_r$.

\begin{theorem}\label{bots}
The associated graded algebra
$\gr \aS(n,r)$ may be identified with 
$\gS(n,r)$ in such a way that
$\gr_d \xi_{A,f}$ 
(for $A \in \Mat{\lambda}{\mu}$
and $f \in \P^{{\mu(A)}}$ that is homogeneous of degree $d$)
is identified with the unique right $\kk \S_r \circltimes \P_r$-module
homomorphism 
\begin{align}
\varsigma_{A,f}:M(\mu) \circltimes \P_r &\rightarrow M(\lambda)\circltimes \P_r, 
&\v_\mu \otimes 1 &\mapsto
\sum_{y \in (\S_{\mu(A)}\backslash \S_\mu)_{\min}}
\v_\lambda d_A y \otimes y^{-1}(f).\label{cloudy}
\end{align}
The homomorphisms $\varsigma_{A,f}$ defined by \cref{cloudy}
for $A \in \Mat{\lambda}{\mu}$
and $f$ running over a basis for 
$\P^{{\mu(A)}}$
give a basis for $1_\lambda \gS(n,r) 1_\mu$ as a free $\kk$-module.
\end{theorem}

\begin{proof}
Under the identifications explained above and $f$ that is homogeneous of degree $d$, 
$\gr_d \xi_{A,f}$ is the graded right 
$\kk \S_r \circltimes \P_r$-module
homomorphism 
\begin{align*}
M(\mu) \circltimes \P_r &\rightarrow M(\lambda)\circltimes \P_r,
&
\v_\mu \otimes 1 &\mapsto
\sum_{y \in (\S_{\mu(A)}\backslash \S_\mu)_{\min}}
(\v_\lambda d_A \otimes f) y,
\end{align*}
which is the same map as in \cref{cloudy}.
In view of \cref{bits}, the other parts of the present 
theorem follow if we can show that these homomorphisms
for $A \in \Mat{\lambda}{\mu}$
and $f$ running over a basis for $\P^{{\mu(A)}}$
give a basis for $\Hom_{\kk \S_r \circltimes \P_r}(M(\mu)\circltimes \P_r,M(\lambda)\circltimes \P_r)$.
This can be proved by mimicking the proof of \cref{bits}.
In fact, several of the steps are easier in the graded setting: the analog of the Mackey theorem of \cite[Th.~3.5.2]{Kleshchev} gives a decomposition as a direct sum rather than merely being a filtration, and \cite[Cor.~3.7.3]{Kleshchev} can be simplified
because $\kk \S_\mu \circltimes \P_r < \kk \S_r \circltimes \P_r$ is a Frobenius extension. We omit the details.
\end{proof}

\begin{corollary}\label{easy}
There is an injective algebra homomorphism
$\iota:\S(n,r) \hookrightarrow \aS(n,r)$
mapping $\xi_A$ to $\xi_{A,1} = \xi_{1,A}$
for $A \in \Mat{\lambda}{\mu}$.
Its image is the subalgebra $F_0\, \aS(n,r)$.
\end{corollary}

\begin{proof}
The existence of $\iota$ 
follows by applying the functor $-\otimes_{\kk \S_r} \aH_r$
to the definition \cref{schurdef} of $\S(n,r)$ and
using the definition \cref{affdef} $\aS(n,r)$.
It is an isomorphism $\S(n,r) \stackrel{\sim}{\rightarrow}
F_0\aS(n,r)$ because it sends the basis vectors $\xi_A$ of $\S(n,r)$ to the basis vectors $\xi_{A,1}$
of $F_0 \aS(n,r)$.
\end{proof}

\begin{corollary}\label{buts}
For $\lambda,\mu \in \Lambda(n,r)$, 
$1_\lambda \aS(n,r) 1_\mu
= \Hom_{\dash \aH_r}\big(M(\mu)\otimes_{\kk \S_r}\aH_r,M(\lambda)\otimes_{\kk \S_r}\aH_r\big)$
is free as a $\kk$-module
with a basis $\{\xi_{f,A}\}$ indexed by pairs $(A,f)$ as $A$ runs over the set $\Mat{\lambda}{\mu}$ and $f$ runs over a basis for
$\P^{{\lambda(A)}}$.
By definition, $\xi_{f,A}$ is the unique right $\aH_r$-module homomorphism
\begin{equation}\label{inbuts}
M(\mu)\otimes_{\kk \S_r}\aH_r\rightarrow M(\lambda)\otimes_{\kk \S_r}\aH_r,
\qquad
\v_\mu \otimes 1\mapsto \sum_{y \in (\S_{\mu(A)} \backslash \S_\mu)_{\min}} \v_\lambda \otimes f d_A y.
\end{equation}
\end{corollary}

\begin{proof}
One first checks that there is a homomorphism $\xi_{f,A}$
as described by following the argument from the first paragraph of the proof of \cref{bits}. 
Acting with $d_A^{-1}$ defines an isomorphism
$\P^{{\lambda(A)}} \stackrel{\sim}{\rightarrow} \P^{{\mu(A)}}$.
If $f \in \P^{{\lambda(A)}}$ 
is homogeneous of degree $d$ then $\xi_{f,A}$ belongs to 
$F_d\, \aS(n,r)$ with
$\gr_d \xi_{f,A} = \gr_d \xi_{A, d_A^{-1}(f)}$.
Now \cref{bots} implies that
the homomorphisms $\xi_{f,A}$ 
give a basis for $1_\lambda \aS(n,r) 1_\mu$ 
for $A \in \Mat{\lambda}{\mu}$ and $f$ running
over a basis for $\P^{{\lambda(A)}}$.
\end{proof}

In the previous section, we identified
$\bigoplus_{\lambda \in \Lambda(n,r)} M(\lambda)$
with the tensor space $V^{\otimes r}$.
Consequently, 
$\bigoplus_{\lambda \in \Lambda(n,r)} M(\lambda) \otimes_{\kk S_r} \aH_r$ is identified with 
the {\em induced tensor space}
$V^{\otimes r} \otimes_{\kk S_r} \aH_r$,
and we have that
\begin{equation}\label{nitpick1}
\aS(n,r) \equiv
\End_{\dash\aH_r}(V^{\otimes r} \otimes_{\kk S_r} \aH_r).
\end{equation}
Similarly, $\bigoplus_{\lambda \in \Lambda(n,r)} M(\lambda)
\circltimes \P_r$ is identified
with $V^{\otimes r} \circltimes \P_r$, so 
\begin{equation}\label{nitpick2}
\gS(n,r) \equiv
\End_{\kk S_r \circltimes \P_r}(V^{\otimes r} \circltimes \P_r).
\end{equation}
Finally, we assume that $n \geq r$, when a little more can be said. Let $\omega := (1^r, 0^{n-r}) \in \Lambda(n,r)$. 
The module $M(\omega)\otimes_{\kk \S_r} \aH_r$ is obviously isomorphic to the right regular $\aH_r$-module.
So we have that
\begin{equation}\label{easybit}
1_\omega \aS(n,r) 1_\omega = 
\End_{\dash\aH_r}(M(\omega)\otimes_{\kk \S_r} \aH^r) \cong
\End_{\dash \aH_r}(\aH_r) \cong \aH_r.
\end{equation}
Also, for $\lambda \in \Lambda(n,r)$, we have that
\begin{multline}\label{easybot}
1_\lambda \aS(n,r) 1_\omega 
=\Hom_{\aS(n,r)}(M(\omega)\otimes_{\kk \S_r} \aH_r,M(\lambda)\otimes_{\kk S_r} \aH_r)\\
\cong \Hom_{\dash \aH_r}(\aH_r,M(\lambda)\otimes_{\kk \S_r}\aH_r) \cong M(\lambda)\otimes_{\kk \S_r}\aH_r.
\end{multline}
Identifying $1_\omega \aS(n,r) 1_\omega$ with $\aH_r$ via \cref{easybit}, it follows that the $(\aS(n,r), \aH_r)$-bimodule
$V^{\otimes r}\otimes_{\kk S_r} \aH_r$ 
is isomorphic to the left ideal $\aS(n,r) 1_\omega$. 
The following two results 
are well known in this sort of situation.

\begin{theorem}\label{dcp}
When $n \geq r$, the right
$\aH_r$-module $T := V^{\otimes r}\otimes_{\kk S_r} \aH_r$
satisfies the double centralizer property,
i.e., $\End_{\End(T)}(T) = T$.
\end{theorem}

\begin{proof}
This follows because
$\End_{\aS(n,r)}(\aS(n,r) 1_\omega)
\cong 1_\omega \aS(n,r) 1_\omega$.
\end{proof}

\begin{theorem}\label{equivcats}
If $\kk$ is a field of characteristic 0
and $n \geq r$ then $\aS(n,r)$ and $\aH_r$ are Morita equivalent. The functor
\begin{equation}\label{schurfunctor}
\Functor_{n,r}:\aH_r\mod\rightarrow \aS(n,r)\mod
\end{equation}
defined by
tensoring over $\aH_r$ with the
$(\aS(n,r), \aH_r)$-bimodule
$V^{\otimes r}\otimes_{\kk S_r} \aH_r$ is an equivalence of categories.
Identifying $\aH_r$ with $1_\omega \aS(n,r) 1_\omega$ as in \cref{easybit}, 
a quasi-inverse equivalence is given
by the idempotent truncation functor $1_\omega(-):\aS(n,r)\mod \rightarrow \aH_r\mod, M \mapsto 1_\omega M$.
\end{theorem}

\begin{proof}
Note that $\aS(n,r) 1_\omega \aS(n,r) = \aS(n,r)$.
This follows because every idempotent
$1_\lambda\:(\lambda \in \Lambda(n,r))$
lies in 
$\aS(n,r) 1_\omega \aS(n,r)$. Indeed, $1_\lambda = 
\lambda_1! \cdots \lambda_n! \ 
\xi_A 1_\omega \xi_B$ where $A \in \Mat{\lambda}{\omega}$ corresponds to the double coset diagram that 
merges $r$ thin strings to thick strings of thickness $\lambda_1,\dots,\lambda_n$
with no crossings,
and $B := A^\transpose$. This is a well-known identity
already in $\S(n,r)$.
Hence, by standard Morita theory, 
the idempotent truncation functor $1_\omega (-)$
is an equivalence of categories $\aS(n,r)\mod \rightarrow \aH_r \mod$.
Morever, this functor is isomorphic to the functor
$1_\omega \aS(n,r) \otimes_{\aS(n,r)}-$, which is left adjoint
to $\Hom_{\aH_r}(1_\omega \aS(n,r),-) \cong \Functor_{n,r}$,
so the latter functor is also an equivalence.
\end{proof}

%% file: s4-pres.tex
\setcounter{section}{3}

\section{The strict monoidal category \texorpdfstring{$\aSchur$}{}}\label{s4-pres}

Sometimes it is convenient to repackage the affine Schur algebras $\aS(n,r)$ for all $n$ and $r$ as follows:

\begin{definition}\label{ascdef}
The {\em degenerate affine Schur category} $\aSchur$
is the $\kk$-linear category with object set
$\coprod_{n \geq 0} \Lambda(n)$ (i.e., all compositions), and morphisms
\begin{equation}\label{howtodoit}
\Hom_{\aSchur}(\mu,\lambda) := 
\begin{cases}
\Hom_{\dash \aH_r}\big(M(\mu)\otimes_{\kk \S_r}\aH_r,M(\lambda)\otimes_{\kk \S_r}\aH_r\big)&\text{if $r := |\lambda| = |\mu|$}\\
\{0\}&\text{if $|\lambda| \neq |\mu|$.}
\end{cases}
\end{equation}
The composition law $-\circ-$ making $\aSchur$ into a $\kk$-linear category is the obvious composition of morphisms;
sometimes, we might omit the symbol $\circ$, denoting a composition $f \circ g$ of morphisms in $\aSchur$ simply by $fg$.
We write $1_\lambda$ for the identity endomorphism $\id_{M(\lambda)\otimes_{\kk \S_r}\aH_r}$.
\end{definition}

\begin{remark}\label{terminology}
The papers \cite{SongWang,SongWang2} use ``affine web category'' for our ``degenerate affine Schur category,'' using ``affine Schur category'' for an extended tensor product version with additional red strands. The terminology becomes even more variable when it comes to cyclotomic quotients.
\end{remark}

From \cref{ascdef}, it is clear that 
the path algebra of the full subcategory of $\aSchur$ 
generated by the objects $\Lambda(n,r)$ is 
the degenerate affine Schur algebra $\aS(n,r)$.

\begin{theorem}\label{basechange}
For $\lambda \in \Lambda(n,r), \mu \in \Lambda(m,r)$,
$\Hom_{\aSchur}(\mu,\lambda)$ is free as a $\kk$-module with a basis
$\xi_{A,f}$ for $A \in \Mat{\lambda}{\mu}$ and $f$ running over a basis for $\P^{{\mu(A)}}$, with the homomorphism
$\xi_{A,f}:M(\mu)\otimes_{\kk \S_r}\aH_r \rightarrow M(\lambda)\otimes_{\kk \S_r}\aH_r$
defined in exactly the same way as in \cref{inbits}.
\end{theorem}

\begin{proof}
This follows from the proof of \cref{bits}. There, $\lambda$ and $\mu$ were assumed to be of the same length, but there is no need to make this assumption.
\end{proof}

There is also the {\em current Schur category} $\gSchur$,
which is defined similarly to $\aSchur$ replacing
$\aH_r$ with $\kk \S_r \circltimes \P_r$ and $M(\lambda)\otimes_{\kk \S_r}\aH_r$ with $M(\lambda)\circltimes \P_r$.
Similar to \cref{bots},
$\aSchur$ is naturally filtered and the associated graded category $\gr \aSchur$ is identified with $\gSchur$.
Moreover, $\Hom_{\gSchur}(\mu,\lambda)$ is free as a $\kk$-module with a natural basis
$\varsigma_{A,f}$ for $A \in \Mat{\lambda}{\mu}$ and $f$ running over a basis for $\P^{{\mu(A)}}$, which is
defined as in \cref{cloudy}.
As in \cref{buts}, it follows that
$\aSchur$ has another basis
$\xi_{f,A}$ for $A \in \Mat{\lambda}{\mu}$ and $f$ running over a basis for $\P^{{\lambda(A)}}$, which is defined as in \cref{inbuts}.

More elementary, the {\em Schur category} $\Schur$ is the $\kk$-linear category defined in the same way as \cref{howtodoit}, replacing 
$\aH_r$ with $\kk \S_r$ and the $\aH_r$-modules $M(\lambda)\otimes_{\kk \S_r}\aH_r$
with the $\kk \S_r$-modules $M(\lambda)$.
The morphism space $\Hom_{\Schur}(\mu,\lambda)$ has basis
$\left\{\xi_A\:\big|\:A \in \Mat{\lambda}{\mu}\right\}$ defined by \cref{standardbasis}.
Almost the same category was defined in \cite[Def.~4.2]{BEEO}.
We are allowing compositions with some parts equal to $0$ whereas \cite{BEEO} only considered strict compositions, which requires some minor modifications to the definition. This is discussed in detail in 
\cite[Sec.~5]{BSchur}, which treats the $q$-analog.

Like in \cref{easy}, there is a faithful $\kk$-linear functor
\begin{equation}
\iota:\Schur \hookrightarrow \aSchur,\qquad
\lambda \mapsto \lambda,\quad
\xi_A \mapsto \xi_{A,1}= \xi_{1,A}.
\end{equation}
This functor identifies $\Schur$ with the wide subcategory of $\aSchur$ consisting of morphisms that are of filtered degree 0, which is also the wide subcategory of $\gSchur$ consisting of the homogeneous morphisms of degree 0.

The most interesting new feature is that all of the categories
$\aSchur, \gSchur$ and $\Schur$ have the additional structure of a tensor product bifunctor $-\star-$ making them into strict $\kk$-linear monoidal categories. We explain this just in the case of $\aSchur$, but the construction is similar for the other two categories. 
For $a,b \geq 0$,
we identify $\S_a \times \S_b$ with the parabolic subgroup $\S_{(a,b)} < \S_{a+b}$, the tensor product
$\P_a \otimes \P_b$ of polynomial algebras 
with $\P_{a+b}$,
and the tensor product
$\aH_a \otimes \aH_b$ with a subalgebra of $\aH_{a+b}$, all in the obvious way; e.g., 
$x_i \otimes 1 \in \P_a \otimes \P_b$
is identified with 
$x_i \in \P_{a+b}$, and $1 \otimes x_j \in \P_a \otimes \P_b$ is identified with $x_{a+j} \in \P_{a+b}$.
Given a right $\aH_a$-module $U$ and a right $\aH_b$-module $V$, there is a right $\aH_{a+b}$-module 
$$
U \ostar V
:= (U \otimes V) \otimes_{\aH_a \otimes \aH_b} \aH_{a+b}.
$$
In fact, this defines a bifunctor $-\ostar-$,
often called {\em induction product}.
Then the tensor product bifunctor
\begin{equation}
-\star-:\aSchur\boxtimes\aSchur \rightarrow \aSchur
\end{equation}
is defined on objects by concatenation of compositions, and on morphisms
$f:\lambda \rightarrow \mu$ and $g:\lambda' \rightarrow \mu'$
with $a := |\lambda|=|\mu|$ and $b := |\lambda'|=|\mu'|$
so that $f \star g: \lambda\star \lambda' \rightarrow \mu\star \mu'$
is the morphism obtained from
$f \ostar g: (M(\lambda)\otimes_{\kk \S_a}\aH_a) \ostar (M(\lambda')\otimes_{\kk \S_{b}} \aH_{b})
\rightarrow (M(\mu)\otimes_{\kk \S_a}\aH_a) \ostar (M(\mu')\otimes_{\kk \S_{b}}\aH_{b})$
using the canonical isomorphisms
$(M(\lambda)\otimes_{\kk \S_a}\aH_a)\ostar (M(\lambda')\otimes_{\kk \S_{b}} \aH_{b})
\cong M(\lambda\star\lambda')\otimes_{\kk \S_{a+b}} \aH_{a+b}$
and
$(M(\mu)\otimes_{\kk \S_a}\aH_a)\ostar (M(\mu')\otimes_{\kk \S_{b}} \aH_{b}) \cong M(\mu\star\mu')\otimes_{\kk \S_{a+b}} \aH_{a+b}$.
We have that
\begin{align}
\xi_{A,f} \star \xi_{B,g} &= \xi_{\diag(A,B), f \otimes g},
&\xi_{f,A} \star \xi_{g,B} &= \xi_{f \otimes g,\diag(A,B)},
\end{align}
It is straightforward to verify that the axioms of strict $\kk$-linear monoidal category are satisfied; this amounts to verifying that the Interchange Law holds.

The embedding $\iota$ of $\Schur$ into $\aSchur$ is a strict monoidal functor.
The full monoidal subcategory of $\Schur$ generated
by the objects $(1^r) \in \Lambda(r,r)$ for all $r \geq 0$
is the {\em symmetric category} $\Sym$.
This is the $\kk$-linearization of the symmetric groupoid, which is the free symmetric strict 
monoidal category generated by one object.
The full monoidal subcategory of $\aSchur$ 
generated by the objects
$(1^r)$ for all $r \geq 0$ is 
the {\em degenerate affine symmetric category}, 
which we denote by $\aSym$.

The monoidal categories $\Sym,\aSym,\Schur,\gSchur$ and $\aSchur$
have explicit monoidal presentations, which we explain next.

\vspace{2mm}
\noindent
\underline{Monoidal presentation of $\Sym$}:
The Coxeter presentation of symmetric groups
implies that $\Sym$ can be presented as the strict $\kk$-linear monoidal category generated by the object $(1)$, whose identity endomorphism we denote by a thin string, together with the morphism 
$\begin{tikzpicture}[anchorbase,scale=.7]
\draw[-] (-0.3,-.3) to (.3,.3);
\draw[-] (0.3,-.3) to (-.3,.3);
\end{tikzpicture}:(1,1) \rightarrow (1,1)$
subject just to the relations
\begin{align}
\begin{tikzpicture}[anchorbase]
\draw[-] (-0.2,-0.4) to[out=45,in=-90] (0.15,0) to[out=90,in=-45] (-0.2,0.4);
\draw[-] (0.2,-0.4) to[out=135,in=-90] (-0.15,0) to[out=90,in=225] (0.2,0.4);
\end{tikzpicture}
&=
\begin{tikzpicture}[anchorbase]
\draw[-] (-0.2,-0.4) -- (-0.2,0.4);
\draw[-] (0.1,-0.4) -- (0.1,0.4);
\end{tikzpicture}\ ,
&
\begin{tikzpicture}[anchorbase]
\draw[-] (-0.4,-0.4) -- (0.4,0.4);
\draw[-] (0,-0.4)  to[out=135,in=-90] (-0.32,0) to[out=90,in=225] (0,0.4);
\draw[-] (0.4,-0.4) -- (-0.4,0.4);
\end{tikzpicture}
&=
\begin{tikzpicture}[anchorbase]
\draw[-] (-0.4,-0.4)  -- (0.4,0.4);
\draw[-] (0,-0.4) to[out=45,in=-90] (0.32,0) to[out=90,in=-45] (0,0.4);
\draw[-] (0.4,-0.4) -- (-0.4,0.4);
\end{tikzpicture}\ .
\end{align}

\vspace{2mm}
\noindent
\underline{Monoidal presentation of $\aSym$}:
To obtain a monoidal presentation for $\aSym$ from the one for $\Sym$, one just needs to add one more generating morphism
$\begin{tikzpicture}[anchorbase,scale=.7]
\draw[-] (0,-.3) to (0,.3);
\closeddot{0,0};
\end{tikzpicture}:(1) \rightarrow (1)$
subject to one of the following equivalent relations:
\begin{align}\label{demazurerel}
\begin{tikzpicture}[anchorbase,scale=1]
\draw[-] (-0.3,-.3) to (.3,.3);
\draw[-] (0.3,-.3) to (-.3,.3);
\closeddot{-.15,-.15};
\end{tikzpicture}
&=
\begin{tikzpicture}[anchorbase,scale=1]
\draw[-] (-0.3,-.3) to (.3,.3);
\draw[-] (0.3,-.3) to (-.3,.3);
\closeddot{.15,.15};
\end{tikzpicture} +\ 
\begin{tikzpicture}[anchorbase,scale=1]
\draw[-] (0.2,-.3) to (.2,.3);
\draw[-] (-0.2,-.3) to (-.2,.3);
\end{tikzpicture}\ ,&
\begin{tikzpicture}[anchorbase,scale=1]
\draw[-] (-0.3,-.3) to (.3,.3);
\draw[-] (0.3,-.3) to (-.3,.3);
\closeddot{-.15,.15};
\end{tikzpicture}
&=
\begin{tikzpicture}[anchorbase,scale=1]
\draw[-] (-0.3,-.3) to (.3,.3);
\draw[-] (0.3,-.3) to (-.3,.3);
\closeddot{.15,-.15};
\end{tikzpicture} +\ 
\begin{tikzpicture}[anchorbase,scale=1]
\draw[-] (0.2,-.3) to (.2,.3);
\draw[-] (-0.2,-.3) to (-.2,.3);
\end{tikzpicture}\ .
\end{align}
This is clear from \cref{easybit} since we already know presentations for each $\aH_r$.

\vspace{2mm}
\noindent
\underline{Monoidal presentation of $\Schur$}:
There are a couple of known presentations for $\Schur$ as a strict $\kk$-linear monoidal category. One was described in \cite{BEEO}, in which the string diagram representing basis element $\xi_A$ is simply its double coset diagram. 
The presentation for $\Schur$
requires generating objects $(r)$ for $r \in \N$, and generating morphisms
given by
the two-fold {\em merges} and {\em splits}
\begin{align}\label{mergesandsplits}
\xi_{\left[\begin{smallmatrix}a& b\end{smallmatrix}\right]}=
\begin{tikzpicture}[anchorbase,centerzero=.1]
\draw[-,line width=1pt] (0.28,-.3) to (0.08,0.04);
\draw[-,line width=1pt] (-0.12,-.3) to (0.08,0.04);
\draw[-,line width=2pt] (0.08,.4) to (0.08,0);
\strand{-.22,-.4}{a};
\strand{0.35,-.4}{b};
\end{tikzpicture} 
&:(a,b) \rightarrow (a+b),&
\xi_{\left[\begin{smallmatrix}a\\       b\end{smallmatrix}\right]}=\begin{tikzpicture}[anchorbase,centerzero=.1]
\draw[-,line width=2pt] (0.08,-.3) to (0.08,0.04);
\draw[-,line width=1pt] (0.28,.4) to (0.08,0);
\draw[-,line width=1pt] (-0.12,.4) to (0.08,0);
\strand{-0.22,.5}{a};
\strand{0.36,.5}{b};
\end{tikzpicture}
&:(a+b)\rightarrow (a,b),
\end{align}
the {\em thick crossings}
\begin{align}\label{ms}
\xi_{\left[\begin{smallmatrix}0 & b\\ a &0\end{smallmatrix}\right]}=
\begin{tikzpicture}[baseline=-.5mm]
\draw[-,line width=1.2pt] (-0.3,-.3) to (.3,.3);
\draw[-,line width=1.2pt] (0.3,-.3) to (-.3,.3);
\strand{0.3,-.4}{b};
\strand{-0.3,-.4}{a};
\end{tikzpicture}
&:(a,b) \rightarrow (b,a),
\end{align}
and also the {\em spots}
\begin{align}\label{spots}
\begin{tikzpicture}[anchorbase,scale=1]
\draw (0,-.3)\botlabel{0} to (0,0);
\spot{0,0};
\end{tikzpicture}\ :(0) &\rightarrow \one,&
&\begin{tikzpicture}[anchorbase,scale=1]
\draw (0,.3)\toplabel{0} to (0,0);
\spot{0,0};
\end{tikzpicture}\ :\one \rightarrow (0)
\end{align}
which are $\xi_A$ for the unique matrices $A$ in $\Mat{()}{(0)}$
and $\Mat{(0)}{()}$, respectively.
When drawing more complicated string diagrams, we use the same conventions as in \cref{conventions}, so dotted lines denote strings of thickness 0, and unlabelled thin solid lines denote strings of thickness 1.
Also, spots may be contracted to the boundary but should not be removed entirely.
A full set of relations is given by 
the following
for $a,b,c,d \geq 0$ with $a+b=c+d$:
\begin{align}
\begin{tikzpicture}[anchorbase]
\draw[zeroline] (0,-.15) to (0,0.15);
\spot{0,.15};
\spot{0,-.15};
\end{tikzpicture}\ &=\ \id_\one,&
\begin{tikzpicture}[anchorbase]
\draw[zeroline] (0,-.1) to (0,-0.3);
\draw[zeroline] (0,.1) to (0,0.3);
\spot{0,.1};
\spot{0,-.1};
\end{tikzpicture}\ &=\ 
\begin{tikzpicture}[anchorbase]
\draw[zeroline] (0,-.25) to (0,0.25);
\end{tikzpicture}\ ,\label{secondone}\\
\begin{tikzpicture}[anchorbase,scale=.8]
\draw[zeroline] (0.28,-.3) to (0.08,0.04);
\draw[-,line width=1pt] (-0.12,-.3) to (0.08,0.04);
\draw[-,line width=1pt] (0.08,.4) to (0.08,0);
\strand{-0.22,-.42}{a};
\end{tikzpicture} 
=
\begin{tikzpicture}[anchorbase,scale=.8]
\draw[-,line width=1pt] (0.08,.4) to (0.08,-.3);
\strand{0.08,-.42}{a};
\spot{.5,-.3};
\end{tikzpicture}\ ,\hspace{15mm}
\begin{tikzpicture}[anchorbase,scale=.8]
\draw[-,line width=1pt] (0.28,-.3) to (0.08,0.04);
\draw[zeroline] (-0.12,-.3) to (0.08,0.04);
\draw[-,line width=1pt] (0.08,.4) to (0.08,0);
\strand{0.35,-.42}{b};
\end{tikzpicture} 
& =\begin{tikzpicture}[anchorbase,scale=.8]
\draw[-,line width=1pt] (0.5,.4) to (0.5,-.3);
\strand{0.5,-.43}{b};
\spot{.08,-.3};
\end{tikzpicture},&
\begin{tikzpicture}[anchorbase,scale=.8,centerzero=.1]
\draw[-,line width=1pt] (0.08,-.3) to (0.08,0.04);
\draw[zeroline] (0.28,.4) to (0.08,0);
\draw[-,line width=1pt] (-0.12,.4) to (0.08,0);
\strand{-0.22,.52}{a};
\end{tikzpicture}&=\begin{tikzpicture}[anchorbase,scale=.8,centerzero=0.05]
\draw[-,line width=1pt] (0.08,-.4) to (0.08,.3);
\strand{0.08,.42}{a};
\spot{.5,.3};
\end{tikzpicture}\ ,\hspace{15mm}
\begin{tikzpicture}[anchorbase,scale=.8,centerzero=.1]
\draw[-,line width=1pt] (0.08,-.3) to (0.08,0.04);
\draw[-,line width=1pt] (0.28,.4) to (0.08,0);
\draw[zeroline] (-0.12,.4) to (0.08,0);
\strand{0.36,.53}{b};
\end{tikzpicture}=\begin{tikzpicture}[anchorbase,scale=.8,centerzero=0.05]
\draw[-,line width=1pt] (0.5,-.4) to (0.5,.3);
\strand{0.5,.45}{b};
\spot{.08,.3};
\end{tikzpicture},\label{zeroforks}\\
\label{assrel}
\begin{tikzpicture}[anchorbase]
\draw[-,thick] (0.35,-.3) to (0.08,0.14);
\draw[-,thick] (0.1,-.3) to (-0.04,-0.06);
\draw[-,line width=1pt] (0.085,.14) to (-0.035,-0.06);
\draw[-,thick] (-0.2,-.3) to (0.07,0.14);
\draw[-,line width=2pt] (0.08,.45) to (0.08,.1);
\strand{0.45,-.41}{c};
\strand{0.07,-.4}{b};
\strand{-0.28,-.41}{a};
\end{tikzpicture}
&=
\begin{tikzpicture}[anchorbase]
\draw[-,thick] (0.36,-.3) to (0.09,0.14);
\draw[-,thick] (0.06,-.3) to (0.2,-.05);
\draw[-,line width=1pt] (0.07,.14) to (0.19,-.06);
\draw[-,thick] (-0.19,-.3) to (0.08,0.14);
\draw[-,line width=2pt] (0.08,.45) to (0.08,.1);
\strand{0.45,-.41}{c};
\strand{0.07,-.4}{b};
\strand{-0.28,-.41}{a};
\end{tikzpicture}\:,
&
\begin{tikzpicture}[anchorbase]
\draw[-,thick] (0.35,.3) to (0.08,-0.14);
\draw[-,thick] (0.1,.3) to (-0.04,0.06);
\draw[-,line width=1pt] (0.085,-.14) to (-0.035,0.06);
\draw[-,thick] (-0.2,.3) to (0.07,-0.14);
\draw[-,line width=2pt] (0.08,-.45) to (0.08,-.1);
\strand{0.45,.4}{c};
\strand{0.07,.42}{b};
\strand{-0.28,.4}{a};
\end{tikzpicture}
&=\begin{tikzpicture}[anchorbase]
\draw[-,thick] (0.36,.3) to (0.09,-0.14);
\draw[-,thick] (0.06,.3) to (0.2,.05);
\draw[-,line width=1pt] (0.07,-.14) to (0.19,.06);
\draw[-,thick] (-0.19,.3) to (0.08,-0.14);
\draw[-,line width=2pt] (0.08,-.45) to (0.08,-.1);
\strand{0.45,.4}{c};
\strand{0.07,.42}{b};
\strand{-0.28,.4}{a};
\end{tikzpicture}\:,\\
\label{mergesplit}
\begin{tikzpicture}[anchorbase,scale=.8]
\draw[-,line width=2pt] (0.08,-.8) to (0.08,-.5);
\draw[-,line width=2pt] (0.08,.3) to (0.08,.6);
\draw[-,thick] (0.1,-.51) to [out=45,in=-45] (0.1,.31);
\draw[-,thick] (0.06,-.51) to [out=135,in=-135] (0.06,.31);
\strand{-.33,-.05}{a};
\strand{.45,-.05}{b};
\end{tikzpicture}
&= 
\binom{a+b}{a}\ 
\begin{tikzpicture}[anchorbase,scale=.77]
\draw[-,line width=2pt] (0.08,-.8) \botlabel{a+b}to (0.08,.6);
\end{tikzpicture},&
\begin{tikzpicture}[anchorbase,scale=1,centerzero=.55]
\draw[-,line width=1.2pt] (0,0) to (.275,.3) to (.275,.7) to (0,1);
\draw[-,line width=1.2pt] (.6,0) to (.315,.3) to (.315,.7) to (.6,1);
\strand{0,1.13}{c};
\strand{0.63,1.13}{d};
\strand{0,-.1}{a};
\strand{0.63,-.1}{b};
\end{tikzpicture}
&=
\sum_{\substack{0 \leq s \leq \min(a,c)\\0 \leq t \leq \min(b,d)\\t-s=d-a=b-c}}
\begin{tikzpicture}[anchorbase,scale=1,centerzero=.55]
\draw[-,thick] (0.58,0) to (0.58,.2) to (.02,.8) to (.02,1);
\draw[-,thick] (0.02,0) to (0.02,.2) to (.58,.8) to (.58,1);
\draw[-,thin] (0,0) to (0,1);
\draw[-,thin] (0.6,0) to (0.6,1);
\strand{0,1.13}{c};
\strand{0.6,1.13}{d};
\strand{0,-.1}{a};
\strand{0.6,-.1}{b};
\strand{-0.1,.5}{s};
\strand{0.75,.5}{t};
\end{tikzpicture}.
\end{align}
The spot generators do not appear in \cite{BEEO} but are needed here since we have added the additional generating object $(0)$ which is isomorphic but not equal to the strict identity object $\one = ()$.
This is discussed further for the $q$-analog of $\Schur$
in \cite[Th.~6.1]{BSchur}.

Using the {\em associativity} and {\em coassociativity relations} of \cref{assrel}, one can introduce $n$-fold merges and splits by composing the two-fold ones in obvious ways; in fact, these are the standard basis elements $\xi_A$
for matrices $A$ that have a single row or a single column, respectively. Then, for any $A \in \Mat{\lambda}{\mu}$, the standard basis element $\xi_A$ is equal to the string diagram
that is simply equal to the double coset diagram for the matrix $A$. This is explained in \cite[Sec.~4]{BEEO}; see also \cref{notaproof} below.
Various other relations are deduced from the defining relations in \cite{BEEO} too, including the following which imply that $\Schur$ is symmetric monoidal with a symmetric braiding defined by the thick crossings:
\begin{align}
\begin{tikzpicture}[anchorbase,scale=0.7,centerzero=.4]
\draw[-,thick] (0.4,0) to (-0.6,1);
\draw[-,thick] (0.08,0) to (0.08,1);
\draw[-,thick] (0.1,0) to (0.1,.6) to (.5,1);
\strand{0.6,1.13}{c};
\strand{0.1,1.16}{b};
\strand{-0.65,1.13}{a};
\end{tikzpicture}
&\!\!=\!\!
\begin{tikzpicture}[anchorbase,scale=0.7,centerzero=.4]
\draw[-,thick] (0.7,0) to (-0.3,1);
\draw[-,thick] (0.08,0) to (0.08,1);
\draw[-,thick] (0.1,0) to (0.1,.2) to (.9,1);
\strand{0.9,1.13}{c};
\strand{0.1,1.16}{b};
\strand{-0.4,1.13}{a};
\end{tikzpicture},\:
\begin{tikzpicture}[anchorbase,scale=0.7,centerzero=.4]
\draw[-,thick] (-0.4,0) to (0.6,1);
\draw[-,thick] (-0.08,0) to (-0.08,1);
\draw[-,thick] (-0.1,0) to (-0.1,.6) to (-.5,1);
\strand{0.7,1.13}{c};
\strand{-0.1,1.16}{b};
\strand{-0.6,1.13}{a};
\end{tikzpicture}
\!\!=\!\!
\begin{tikzpicture}[anchorbase,scale=0.7,centerzero=.4]
\draw[-,thick] (-0.7,0) to (0.3,1);
\draw[-,thick] (-0.08,0) to (-0.08,1);
\draw[-,thick] (-0.1,0) to (-0.1,.2) to (-.9,1);
\strand{0.4,1.13}{c};
\strand{-0.1,1.16}{b};
\strand{-0.95,1.13}{a};
\end{tikzpicture},
\:\begin{tikzpicture}[anchorbase,scale=0.7,centerzero=-.4]
\draw[-,thick] (0.4,0) to (-0.6,-1);
\draw[-,thick] (0.08,0) to (0.08,-1);
\draw[-,thick] (0.1,0) to (0.1,-.6) to (.5,-1);
\strand{0.6,-1.13}{c};
\strand{0.07,-1.13}{b};
\strand{-0.6,-1.13}{a};
\end{tikzpicture}
\!\!=\!\!
\begin{tikzpicture}[anchorbase,scale=0.7,centerzero=-.4]
\draw[-,thick] (0.7,0) to (-0.3,-1);
\draw[-,thick] (0.08,0) to (0.08,-1);
\draw[-,thick] (0.1,0) to (0.1,-.2) to (.9,-1);
\strand{1,-1.13}{c};
\strand{0.1,-1.13}{b};
\strand{-0.4,-1.13}{a};
\end{tikzpicture},\:
\begin{tikzpicture}[anchorbase,scale=0.7,centerzero=-.4]
\draw[-,thick] (-0.4,0) to (0.6,-1);
\draw[-,thick] (-0.08,0) to (-0.08,-1);
\draw[-,thick] (-0.1,0) to (-0.1,-.6) to (-.5,-1);
\strand{0.6,-1.13}{c};
\strand{-0.1,-1.13}{b};
\strand{-0.6,-1.13}{a};
\end{tikzpicture}
\!\!=\!\!
\begin{tikzpicture}[anchorbase,scale=0.7,centerzero=-.4]
\draw[-,thick] (-0.7,0) to (0.3,-1);
\draw[-,thick] (-0.08,0) to (-0.08,-1);
\draw[-,thick] (-0.1,0) to (-0.1,-.2) to (-.9,-1);
\strand{0.34,-1.13}{c};
\strand{-0.1,-1.13}{b};
\strand{-0.95,-1.13}{a};
\end{tikzpicture},
\label{sliders}\end{align}\begin{align}
\label{symmetric}
\begin{tikzpicture}[anchorbase,scale=0.8]
\draw[-,thick] (0.28,0) to[out=90,in=-90] (-0.28,.6);
\draw[-,thick] (-0.28,0) to[out=90,in=-90] (0.28,.6);
\draw[-,thick] (0.28,-.6) to[out=90,in=-90] (-0.28,0);
\draw[-,thick] (-0.28,-.6) to[out=90,in=-90] (0.28,0);
\strand{0.3,-.75}{b};
\strand{-0.3,-.75}{a};
\end{tikzpicture}
&=
\begin{tikzpicture}[anchorbase,scale=0.8]
\draw[-,thick] (0.2,-.6) to (0.2,.6);
\draw[-,thick] (-0.2,-.6) to (-0.2,.6);
\strand{0.2,-.75}{b};
\strand{-0.2,-.75}{a};
\end{tikzpicture}
\:,&
\begin{tikzpicture}[anchorbase,scale=0.8]
\draw[-,thick] (0.45,.6) to (-0.45,-.6);
\draw[-,thick] (0.45,-.6) to (-0.45,.6);
\draw[-,thick] (0,-.6) to[out=90,in=-90] (-.45,0);
\draw[-,thick] (-0.45,0) to[out=90,in=-90] (0,0.6);
\strand{0,-.77}{b};
\strand{0.5,-.77}{c};
\strand{-0.5,-.77}{a};
\end{tikzpicture}
&=
\begin{tikzpicture}[anchorbase,scale=0.8]
\draw[-,thick] (0.45,.6) to (-0.45,-.6);
\draw[-,thick] (0.45,-.6) to (-0.45,.6);
\draw[-,thick] (0,-.6) to[out=90,in=-90] (.45,0);
\draw[-,thick] (0.45,0) to[out=90,in=-90] (0,0.6);
\strand{0,-.77}{b};
\strand{0.5,-.77}{c};
\strand{-0.5,-.77}{a};
\end{tikzpicture}\:.
\end{align}
Also useful are the {\em absorption relations}:
\begin{align}
\begin{tikzpicture}[anchorbase,scale=.7,centerzero=-.15]
	\draw[-,line width=2pt] (0.08,.3) to (0.08,.5);
\draw[-,thick] (-.2,-.8) \botlabel{a} to [out=45,in=-45] (0.1,.31);
\draw[-,thick] (.36,-.8) \botlabel{b} to [out=135,in=-135] (0.06,.31);
\end{tikzpicture}
&=
\begin{tikzpicture}[anchorbase,scale=.7,centerzero=-.15]
	\draw[-,line width=2pt] (0.08,.1) to (0.08,.5);
\draw[-,thick] (.46,-.8) \botlabel{b} to [out=100,in=-45] (0.1,.11);
\draw[-,thick] (-.3,-.8) \botlabel{a} to [out=80,in=-135] (0.06,.11);
\end{tikzpicture}\ ,
&
\begin{tikzpicture}[anchorbase,scale=.7,centerzero=.2]
	\draw[-,line width=2pt] (0.08,-.3) to (0.08,-.5);
\draw[-,thick] (-.2,.8)\toplabel{a} to [out=-45,in=45] (0.1,-.31);
\draw[-,thick] (.36,.8) \toplabel{b}to [out=-135,in=135] (0.06,-.31);
\end{tikzpicture}
&=
\begin{tikzpicture}[anchorbase,scale=.7,centerzero=.2]
	\draw[-,line width=2pt] (0.08,-.1) to (0.08,-.5);
\draw[-,thick] (.46,.8)\toplabel{b} to [out=-100,in=45] (0.1,-.11);
\draw[-,thick] (-.3,.8)\toplabel{a} to [out=-80,in=135] (0.06,-.11);
\end{tikzpicture}\ .
\label{swallows}
\end{align}

There is another more efficient
presentation for $\Schur$ which was known before \cite{BEEO}. Algebraically, the idea for this can already be seen in \cite{doty-giaquinto}, and it is closely 
related to the more sophisticated monoidal presentations in \cite{CKM}.
To explain this, we note first that the
thick crossings can be written in terms of splits and merges
since we have that
\begin{align}
\begin{tikzpicture}[anchorbase,scale=1.2]
\draw[-,line width=1.2pt] (0.3,-.3) to (-.3,.4);
\draw[-,line width=1.2pt] (-0.3,-.3) to (.3,.4);
\strand{-0.3,-.4}{a};
\strand{0.3,-.4}{b};
\end{tikzpicture}
&=\sum_{s=0}^{\min(a,b)}
(-1)^{s}
\begin{tikzpicture}[anchorbase,scale=1,centerzero=.4]
\draw[-,line width=1.2pt] (0,0) to (0,1);
\draw[-,thick] (-0.8,0) to (-0.8,.2) to (-.03,.4) to (-.03,.6) to (-.8,.8) to (-.8,1);
\draw[-,thin] (-0.82,0) to (-0.82,1);
\strand{-0.81,-.1}{a};
\strand{0,-.1}{b};
\strand{-0.4,.9}{b-s};
\strand{-0.4,.13}{a-s};
\end{tikzpicture}
=
\sum_{s=0}^{\min(a,b)}
(-1)^{s}
\begin{tikzpicture}[anchorbase,scale=1,centerzero=.4]
\draw[-,line width=1.2pt] (0,0) to (0,1);
\draw[-,thick] (0.8,0) to (0.8,.2) to (.03,.4) to (.03,.6) to (.8,.8) to (.8,1);
\draw[-,thin] (0.82,0) to (0.82,1);
\strand{0.81,-.1}{b};
\strand{0,-.1}{a};
\strand{0.4,.9}{a-s};
\strand{0.4,.13}{b-s};
\end{tikzpicture}.\label{tourists}
\end{align}
So $\Schur$ is already generated by the two-fold splits and merges and the spots. 
A full set of relations for these generators is given by \cref{secondone,zeroforks,assrel}
and one of the equivalent {\em square-switch relations}
\begin{align}\label{squareswitch1}
\begin{tikzpicture}[anchorbase,scale=1,centerzero=.5]
\draw[-,thick] (0,0) to (0,1);
\draw[-,thick] (.015,0) to (0.015,.2) to (.57,.4) to (.57,.6) to (.015,.8) to (.015,1);
\draw[-,line width=1.2pt] (0.6,0) to (0.6,1);
\strand{0.6,-.12}{b};
\strand{0,-.1}{a};
\strand{0.3,.84}{c};
\strand{0.3,.19}{d};
\end{tikzpicture}
&=
\sum_{s=\max(0,c-b)}^{\min(c,d)}
\binom{a -b+c-\!d}{s}\ 
\begin{tikzpicture}[anchorbase,scale=1,centerzero=.5]
\draw[-,line width=1.2pt] (0,0) to (0,1);
\draw[-,thick] (0.8,0) to (0.8,.2) to (.03,.4) to (.03,.6) to (.8,.8) to (.8,1);
\draw[-,thin] (0.82,0) to (0.82,1);
\strand{0.81,-.12}{b};
\strand{0,-.1}{a};
\strand{0.4,.9}{d-s};
\strand{0.4,.13}{c-s};
\end{tikzpicture},\\\label{squareswitch2}
\begin{tikzpicture}[anchorbase,scale=1,centerzero=.5]
\draw[-,thick] (0,0) to (0,1);
\draw[-,thick] (-.015,0) to (-0.015,.2) to (-.57,.4) to (-.57,.6) to (-.015,.8) to (-.015,1);
\draw[-,line width=1.2pt] (-0.6,0) to (-0.6,1);
\strand{-0.6,-.12}{b};
\strand{0,-.1}{a};
\strand{-0.3,.84}{c};
\strand{-0.3,.19}{d};
\end{tikzpicture}&=
\sum_{s=\max(0,c-b)}^{\min(c,d)}
\binom{a-b+c-d}{s}
\begin{tikzpicture}[anchorbase,scale=1,centerzero=.5]
\draw[-,line width=1.2pt] (0,0) to (0,1);
\draw[-,thick] (-0.8,0) to (-0.8,.2) to (-.03,.4) to (-.03,.6) to (-.8,.8) to (-.8,1);
\draw[-,thin] (-0.82,0) to (-0.82,1);
\strand{-0.81,-.12}{b};
\strand{0,-.1}{a};
\strand{-0.4,.9}{d-s};
\strand{-0.4,.13}{c-s};
\end{tikzpicture}
\end{align}
for $a,b,c,d \geq 0$ 
with $d \leq a$ and $c \leq b+d$.

\vspace{2mm}
\noindent
\underline{Monoidal presentation of $\gSchur$}:
We obtain a monoidal presentation of $\gSchur$ from the one for $\Schur$ by adjoining one additional family of generating morphisms, which represent 
$\varsigma_{[r],f}=\varsigma_{f,[r]} \in \End_{\gSchur}([r])_d$
for $f \in \P^{(r)}$ that is homogeneous of degree $d$. We denote them by pinning the symmetric polynomial $f$ to a string of thickness $r$:
\begin{equation}\label{needles}
\begin{tikzpicture}[anchorbase]
\draw[ultra thick] (0,-.3) \botlabel{r} to (0,.3);
\pin{0,0}{.5,0}{f};
\end{tikzpicture}: (r) \rightarrow (r).
\end{equation}
A full set of relations is given by \cref{secondone,zeroforks,assrel,mergesplit} together with four additional families of relations.
First, we need the {\em algebra relations}
\begin{align}\label{alghom}
\begin{tikzpicture}[anchorbase]
\draw[ultra thick] (0,-.4) \botlabel{r} to (0,.4);
\pin{0,0}{.5,0}{c};
\end{tikzpicture}&=
c\ \begin{tikzpicture}[anchorbase]
\draw[ultra thick] (0,-.4) \botlabel{r} to (0,.4);
\end{tikzpicture}\ ,&
\begin{tikzpicture}[anchorbase]
\draw[ultra thick] (0,-.4) \botlabel{r} to (0,.4);
\pin{0,0}{.6,0}{f+g};
\end{tikzpicture}
&=
\begin{tikzpicture}[anchorbase]
\draw[ultra thick] (0,-.3) \botlabel{r} to (0,.4);
\pin{0,0}{.4,0}{f};
\end{tikzpicture}
+
\begin{tikzpicture}[anchorbase]
\draw[ultra thick] (0,-.3) \botlabel{r} to (0,.4);
\pin{0,0}{.4,0}{g};
\end{tikzpicture}\ ,&
\begin{tikzpicture}[anchorbase]
\draw[ultra thick] (0,-.4) \botlabel{r} to (0,.4);
\pin{0,0}{.5,0}{fg};
\end{tikzpicture}&=
\begin{tikzpicture}[anchorbase]
\draw[ultra thick] (0,-.4) \botlabel{r} to (0,.4);
\pin{0,.22}{.5,.22}{f};
\pin{0,-.22}{.5,-.22}{g};
\end{tikzpicture}&
\end{align}
for $r \geq 0$, $c \in \kk$ and $f, g \in \P^{(r)}$, 
i.e., all of the maps
$\P^{(r)} \rightarrow \End_{\gSchur}\big((r)\big),\ 
f \mapsto \!\!
\begin{tikzpicture}[anchorbase]
\draw[ultra thick] (0,-.2) \botlabel{r} to (0,.2);
\pin{0,0}{.5,0}{f};
\end{tikzpicture}$
are graded algebra homomorphisms.
Next, recall that we have identified
$\P_a \otimes \P_b$ with $\P_{a+b}$.
Under this identification, any symmetric polynomial
$f \in \P_{a+b}^{\S_{a+b}}$
is equal to $\sum_{i=1}^n f_{1,i} \otimes f_{2,i}$
for some $n \geq 0$, $f_{1,i} \in \P_a^{\S_a}$ and
$f_{2,i} \in \P_b^{\S_b}$.
We use the Sweedler-type notation $f_{(1)} \otimes f_{(2)}$ as a shorthand for this summation.
The next relations are the {\em coproduct relations}
\begin{align}
\begin{tikzpicture}[baseline = -.5mm]
\draw[-,line width=1pt] (0.28,-.3) to (0.08,0.04);
\draw[-,line width=1pt] (-0.12,-.3) to (0.08,0.04);
\draw[-,line width=2pt] (0.08,.4) to (0.08,0);
\strand{-.22,-.4}{a};
\strand{0.35,-.4}{b};
\pin{.08,.2}{-.4,.2}{f};
\end{tikzpicture} 
&=\begin{tikzpicture}[baseline = -.5mm]
\draw[-,line width=1pt] (0.28,-.3) to (0.08,0.04);
\draw[-,line width=1pt] (-0.12,-.3) to (0.08,0.04);
\draw[-,line width=2pt] (0.08,.4) to (0.08,0);
\strand{-.22,-.4}{a};
\strand{0.35,-.4}{b};
\pin{-.04,-.15}{-.65,0}{f_{(1)}};
\pin{.2,-.15}{.83,0}{f_{(2)}};
\end{tikzpicture}\ ,
&
\begin{tikzpicture}[baseline = -.5mm]
\draw[-,line width=2pt] (0.08,-.3) to (0.08,0.04);
\draw[-,line width=1pt] (0.28,.4) to (0.08,0);
\draw[-,line width=1pt] (-0.12,.4) to (0.08,0);
\strand{-0.22,.5}{a};
\strand{0.36,.5}{b};
\pin{.08,-.1}{-.4,-.1}{f};
\end{tikzpicture}
&=\begin{tikzpicture}[baseline = -.5mm]
\draw[-,line width=2pt] (0.08,-.3) to (0.08,0.04);
\draw[-,line width=1pt] (0.28,.4) to (0.08,0);
\draw[-,line width=1pt] (-0.12,.4) to (0.08,0);
\strand{-0.22,.5}{a};
\strand{0.36,.5}{b};
\pin{-.03,.22}{-.65,0.1}{f_{(1)}};
\pin{.2,.22}{.83,0.1}{f_{(2)}};
\end{tikzpicture}\label{fridaynight}
\end{align}
for all $a,b \geq 0$ and $f \in \P_{a+b}^{\S_{a+b}}$.
Then there 
is the {\em shuffle relation}
\begin{equation}\label{shufflerel}
\begin{tikzpicture}[anchorbase,scale=.8]
\draw[-,line width=2pt] (0.08,-.8) to (0.08,-.5);
\draw[-,line width=2pt] (0.08,.3) to (0.08,.6);
\draw[-,thick] (0.1,-.51) to [out=45,in=-45] (0.1,.31);
\draw[-,thick] (0.06,-.51) to [out=135,in=-135] (0.06,.31);
\strand{-.26,-.29}{a};
\strand{.39,-.25}{b};
\pin{-.08,.05}{-.66,.1}{f};
\pin{.26,.05}{.8,.1}{g};
\end{tikzpicture}
= 
\sum_{w \in \left(\S_{a+b} / \S_a \times \S_b\right)_{\min}}
\begin{tikzpicture}[anchorbase,scale=.77]
\draw[-,line width=2pt] (0.08,-.8) \botlabel{a+b}to (0.08,.6);
\pin{.08,-.1}{1.3,-.1}{w (f \otimes g)};
\end{tikzpicture}
\end{equation}
for all $a,b \geq 0$, $f \in \P_a^{\S_a}$ and $g \in \P_b^{\S_b}$.
Finally, we have that
\begin{align}\label{easydotslide}
\begin{tikzpicture}[baseline=-.5mm]
\draw[-,line width=1.2pt] (-0.3,-.3) to (.3,.3);
\draw[-,line width=1.2pt] (0.3,-.3) to (-.3,.3);
\strand{0.3,-.4}{b};
\strand{-0.3,-.4}{a};
\pin{-.15,-.15}{-.6,-.1}{f};
\end{tikzpicture}
&=
\begin{tikzpicture}[baseline=-.5mm]
\draw[-,line width=1.2pt] (-0.3,-.3) to (.3,.3);
\draw[-,line width=1.2pt] (0.3,-.3) to (-.3,.3);
\strand{0.3,-.4}{b};
\strand{-0.3,-.4}{a};
\pin{.15,.15}{.6,.1}{f};
\end{tikzpicture}\ ,&
\begin{tikzpicture}[baseline=-.5mm]
\draw[-,line width=1.2pt] (-0.3,-.3) to (.3,.3);
\draw[-,line width=1.2pt] (0.3,-.3) to (-.3,.3);
\strand{0.3,-.4}{b};
\strand{-0.3,-.4}{a};
\pin{-.15,.15}{-.6,.1}{g};
\end{tikzpicture}
&=\begin{tikzpicture}[baseline=-.5mm]
\draw[-,line width=1.2pt] (-0.3,-.3) to (.3,.3);
\draw[-,line width=1.2pt] (0.3,-.3) to (-.3,.3);
\strand{0.3,-.4}{b};
\strand{-0.3,-.4}{a};
\pin{.15,-.15}{.6,-.1}{g};
\end{tikzpicture}\ ,
\end{align}
again
for all $a,b \geq 0$, $f \in \P_a^{\S_a}$ and $g \in \P_b^{\S_b}$.
The last relation implies that $\gSchur$ is symmetric monoidal with the same symmetric braiding defined by the thick crossings as on $\Schur$.

\vspace{2mm}
\noindent
\underline{Monoidal presentation of $\aSchur$}:
Finally we come to the monoidal presentation of $\aSchur$, 
which was worked out recently in \cite{SongWang}.
It has the same generating objects and morphisms as $\gSchur$---merges, splits, thick crossings, spots, and the pins \cref{needles}
labelled by symmetric polynomials
$f \in \P^{(r)}$ which denote $f 1_{(r)} =
\xi_{[r], f} =\xi_{f,[r]}\in \End_{\aSchur}\big((r)\big)$,
notation as in \cref{global}.
Then we need the $\Schur$ relations \cref{secondone,zeroforks,assrel,mergesplit}, the algebra relations \cref{alghom},
the coproduct relations \cref{fridaynight},
and two more relations which are deformed version of \cref{shufflerel,easydotslide}.
The {\em deformed shuffle relation} is
\begin{equation}\label{deformedshuffle}
\begin{tikzpicture}[anchorbase,scale=.8]
\draw[-,line width=2pt] (0.08,-.8) to (0.08,-.5);
\draw[-,line width=2pt] (0.08,.3) to (0.08,.6);
\draw[-,thick] (0.1,-.51) to [out=45,in=-45] (0.1,.31);
\draw[-,thick] (0.06,-.51) to [out=135,in=-135] (0.06,.31);
\strand{-.26,-.29}{a};
\strand{.39,-.25}{b};
\pin{-.08,.05}{-.66,.1}{f};
\pin{.27,.05}{.8,.1}{g};
\end{tikzpicture}
= 
\sum_{w \in \left(\S_{a+b} / \S_a \times \S_b\right)_{\min}}
\begin{tikzpicture}[anchorbase,scale=.77]
\draw[-,line width=2pt] (0.08,-.8) \botlabel{a+b}to (0.08,.6);
\pin{.08,-.1}{1.45,-.1}{w \bully (f \otimes g)};
\end{tikzpicture}\ ,
\end{equation}
where $\bully$ here is the deformed left action of the symmetric group on polynomials
defined by \cref{polyaction}.
There is not any obvious analog of \cref{easydotslide} in the deformed setting for general symmetric polynomials $f$ and $g$. 
However, if we require that $f$ and $g$ are elementary symmetric polynomials, there is a reasonable replacement, which 
is sufficient because elementary polynomials generate the algebra of all symmetric polynomials.
We adopt the convention that a pin with label $e_d$ (resp., $h_d$)
attached to a string of thickness $r$ refers to the elementary symmetric polynomial $e_d(x_1,\dots,x_r)$
(resp., the complete  symmetric
polynomial $h_d(x_1,\dots,x_r)$) of degree $d$.
Then, in place of \cref{easydotslide}, we have the {\em elementary dot slide relations}
\begin{align}\label{jonrel}
\begin{tikzpicture}[anchorbase,scale=1,centerzero=.4]
\draw[-,line width=1.1pt] (0.58,0)\botlabel{b} to (.02,1);
\draw[-,line width=1.1pt] (0.02,0)\botlabel{a} to (.58,1);
\pin{.14,.22}{-.3,.3}{e_d};
\end{tikzpicture}
&=
\sum_{s=0}^{\min(a,b,d)}
s!
\begin{tikzpicture}[anchorbase,scale=1,centerzero=.4]
\draw[-,thick] (0.58,0)\botlabel{b} to (0.58,.2) to (.02,.8) to (.02,1);
\draw[-,thick] (0.02,0)\botlabel{a} to (0.02,.2) to (.58,.8) to (.58,1);
\draw[-,thin] (0,0) to (0,1);
\draw[-,thin] (0.6,0) to (0.6,1);
\strand{-0.1,.5}{s};
\strand{0.75,.5}{s};
\pin{.45,.65}{.3,1.25}{e_{d-s}};
\end{tikzpicture},&
\begin{tikzpicture}[anchorbase,scale=1,centerzero=.4]
\draw[-,line width=1.1pt] (0.58,0)\botlabel{a} to (.02,1);
\draw[-,line width=1.1pt] (0.02,0)\botlabel{b} to (.58,1);
\pin{.46,.22}{.9,.3}{e_d};
\end{tikzpicture}
&=\!\!
\sum_{s=0}^{\min(a,b,d)}\!\!
(-1)^s s!
\begin{tikzpicture}[anchorbase,scale=1,centerzero=.4]
\draw[-,thick] (0.58,0) \botlabel{a} to (0.58,.2) to (.02,.8) to (.02,1);
\draw[-,thick] (0.02,0) \botlabel{b} to (0.02,.2) to (.58,.8) to (.58,1);
\draw[-,thin] (0,0) to (0,1);
\draw[-,thin] (0.6,0) to (0.6,1);
\strand{-0.1,.5}{s};
\strand{0.75,.5}{s};
\pin{.15,.65}{.3,1.25}{e_{d-s}};
\end{tikzpicture},\\
\begin{tikzpicture}[anchorbase,scale=1,centerzero=.4]
\draw[-,line width=1.1pt] (0.58,0) to (.02,1)\toplabel{a};
\draw[-,line width=1.1pt] (0.02,0) to (.58,1)\toplabel{b};
\pin{.14,.78}{-.3,.7}{e_d};
\end{tikzpicture}
&=\!\!
\sum_{s=0}^{\min(a,b,d)}\!
s!
\begin{tikzpicture}[anchorbase,scale=1,centerzero=.4]
\draw[-,thick] (0.58,0) to (0.58,.2) to (.02,.8) to (.02,1)\toplabel{a};
\draw[-,thick] (0.02,0) to (0.02,.2) to (.58,.8) to (.58,1)\toplabel{b};
\draw[-,thin] (0,0) to (0,1);
\draw[-,thin] (0.6,0) to (0.6,1);
\strand{-0.1,.5}{s};
\strand{0.75,.5}{s};
\pin{.45,.35}{.3,-.25}{e_{d-s}};
\end{tikzpicture},&
\begin{tikzpicture}[anchorbase,scale=1,centerzero=.4]
\draw[-,line width=1.1pt] (0.58,0) to (.02,1)\toplabel{b};
\draw[-,line width=1.1pt] (0.02,0) to (.58,1)\toplabel{a};
\pin{.46,.78}{.9,.7}{e_d};
\end{tikzpicture}
&=
\!\!\sum_{s=0}^{\min(a,b,d)}\!\!
(-1)^s s!
\begin{tikzpicture}[anchorbase,scale=1,centerzero=.4]
\draw[-,thick] (0.58,0) to (0.58,.2) to (.02,.8) to (.02,1)\toplabel{b};
\draw[-,thick] (0.02,0) to (0.02,.2) to (.58,.8) to (.58,1)\toplabel{a};
\draw[-,thin] (0,0) to (0,1);
\draw[-,thin] (0.6,0) to (0.6,1);
\strand{-0.1,.5}{s};
\strand{0.75,.5}{s};
\pin{.15,.35}{.3,-.25}{e_{d-s}};
\end{tikzpicture}.\label{otherjon}
\end{align}
for all $a,b,d \geq 1$.
When $a=b=d=1$ these relations are the same as \cref{demazurerel}.
In fact, the arguments below show that any one of these four relations implies the other three (in the presence of the earlier relations).

\vspace{2mm}
\noindent
\underline{Proofs}.
Now we explain how to establish 
these presentations in the cases of $\gSchur$ and $\aSchur$.
The first important step is to understand how to represent the morphisms $\xi_{A,f}$ and $\xi_{f,A}$ using string diagrams.
We do this with an example, but the general case is similar.

\begin{example}\label{notaproof}
Take $\lambda = (4,5)$, $\mu = (3,2,4)$
and $A \in \Mat{\lambda}{\mu}$ as in \cref{myeg}.
We have that $\lambda(A) = (1,0,3,2,2,1)$ and 
$\mu(A) = (1,2,0,2,3,1)$.
Let $f = f_{(1)} \otimes f_{(2)} \otimes f_{(3)} \otimes f_{(4)} \otimes f_{(5)}$ be a symmetric polynomial in
$\P_9^{\S_{\mu(A)}} \equiv \P_1 \otimes \P_2^{\S_2} \otimes \P_2^{\S_2}
\otimes \P_3^{\S_3} \otimes \P_1$
and $g = g_{(1)} \otimes g_{(2)} \otimes g_{(3)} \otimes g_{(4)} \otimes g_{(5)}$ be in 
$\P_9^{\S_{\lambda(A)}} = \P_1 \otimes \P_3^{\S_3} \otimes \P_2^{\S_2} \otimes \P_2^{\S_2} \otimes \P_1$.
Then
\begin{align*}
\xi_{A,f} &= \begin{tikzpicture}[anchorbase,scale=1.55]
\draw[-,line width=.6mm] (.212,.5) to (.212,.39);
\draw[-,line width=.75mm] (.595,.5) to (.595,.39);
\draw[-,line width=.15mm] (0.0005,-.396) to (.2,.4);
\draw[-,line width=.3mm] (0.01,-.4) to (.59,.4);
\draw[-,line width=.3mm] (.4,-.4) to (.607,.4);
\draw[-,line width=.45mm] (.79,-.4) to (.214,.4);
\draw[-,line width=.15mm] (.8035,-.398) to (.614,.4);
\draw[-,line width=.3mm] (.4006,-.5) to (.4006,-.395);
\draw[-,line width=.6mm] (.788,-.5) to (.788,-.395);
\draw[-,line width=.45mm] (0.011,-.5) to (0.011,-.395);
\strand{.35,.35}{3};
\strand{.2,-.26}{2};
\strand{.5,-.26}{2};
\pinpinpinpinpin{.74,-.12}{.59,-.12}{.47,-.12}{.21,-.12}{0.07,-.12}{-1.3,-.12}{f_{(1)}\otimes f_{(2)} \otimes f_{(3)}\otimes f_{(4)} \otimes f_{(5)}};
\end{tikzpicture},&
\xi_{g,A} &= \begin{tikzpicture}[anchorbase,scale=1.55]
\draw[-,line width=.6mm] (.212,.5) to (.212,.39);
\draw[-,line width=.75mm] (.595,.5) to (.595,.39);
\draw[-,line width=.15mm] (0.0005,-.396) to (.2,.4);
\draw[-,line width=.3mm] (0.01,-.4) to (.59,.4);
\draw[-,line width=.3mm] (.4,-.4) to (.607,.4);
\draw[-,line width=.45mm] (.79,-.4) to (.214,.4);
\draw[-,line width=.15mm] (.8035,-.398) to (.614,.4);
\draw[-,line width=.3mm] (.4006,-.5) to (.4006,-.395);
\draw[-,line width=.6mm] (.788,-.5) to (.788,-.395);
\draw[-,line width=.45mm] (0.011,-.5) to (0.011,-.395);
\strand{.35,.35}{3};
\strand{.2,-.26}{2};
\strand{.5,-.26}{2};
\pinpinpinpinpin{.66,.23}{.57,.23}{.47,.23}{.34,.23}{0.16,.23}{-1.2,.23}{g_{(1)}\otimes g_{(2)} \otimes g_{(3)}\otimes g_{(4)} \otimes g_{(5)}};
\end{tikzpicture}.
\end{align*}
We explain how to see this in more detail for $\xi_{A,f}$.
The string diagram can be split into four horizontal strips
$$
\begin{tikzpicture}[anchorbase,scale=2]
\draw[-,line width=.15mm] (0.003,-.4) to (.2,.4);
\draw[-,line width=.15mm] (.8,-.398) to (.61,.4);
\draw[-,line width=.6mm] (.212,.5) to (.212,.39);
\draw[-,line width=.75mm] (.595,.5) to (.595,.39);
\draw[-,line width=.3mm] (0.01,-.4) to (.59,.4);
\draw[-,line width=.3mm] (.4,-.4) to (.6,.4);
\draw[-,line width=.45mm] (.79,-.4) to (.214,.4);
\draw[-,line width=.3mm] (.4006,-.5) to (.4006,-.395);
\draw[-,line width=.6mm] (.788,-.5) to (.788,-.395);
\draw[-,line width=.45mm] (0.011,-.5) to (0.011,-.395);
\draw[zeroline] (0.205,.39) to (.395,-.4);
\strand{.35,.35}{3};
\strand{.2,-.26}{2};
\strand{.5,-.26}{2};
\pinpinpinpinpinpin{.74,-.12}{.59,-.12}{.47,-.12}{.34,-.12}{.21,-.12}{0.07,-.12}{-1.3,-.12}{f_{(1)}\otimes f_{(2)} \otimes 1\otimes f_{(3)}\otimes f_{(4)} \otimes f_{(5)}};
\draw[red,densely dashed] (-.2,.25) to (1,.25);
\draw[red,densely dashed] (-.2,-.06) to (1,-.06);
\draw[red,densely dashed] (-.2,-.18) to (1,-.18);
\draw[red,densely dashed] (-.2,.5) to (1,.5);
\draw[red,densely dashed] (-.2,-.5) to (1,-.5);
\node at (1.2,-.5) {$\scriptstyle \mu$};
\node at (1.2,-.18) {$\scriptstyle \mu(A)$};
\node at (1.2,-.06) {$\scriptstyle \mu(A)$};
\node at (1.2,.25) {$\scriptstyle \lambda(A)$};
\node at (1.2,.5) {$\scriptstyle \lambda$};
\node at (2,-.35) {$\color{blue}\delta$};
\node at (2,-.13) {$\color{blue}\gamma$};
\node at (2,.1) {$\color{blue}\beta$};
\node at (2,.38) {$\color{blue}\alpha$};
\end{tikzpicture}\ .
$$
We have drawn the string of thickness 0 too for clarity.
As spelled out in \cite[Sec.~4]{BEEO}, the horizontal composition 
of the three two-fold
 splits at the bottom is
\begin{align*}
\delta:=\xi_{\left[\begin{smallmatrix}1\\2\end{smallmatrix}\right]}
\star\xi_{\left[\begin{smallmatrix}0\\2\end{smallmatrix}\right]}
\star\xi_{\left[\begin{smallmatrix}3\\1\end{smallmatrix}\right]}=
\xi_{\left[\begin{smallmatrix}1&0&0\\2&0&0\\0&0&0\\0&2&0\\0&0&3\\0&0&1\end{smallmatrix}\right]}:M(\mu)\otimes_{\kk \S_9}\aH_9 &\hookrightarrow M(\mu(A))\otimes_{\kk \S_9} \aH_9,
\end{align*}
which maps
$\v_\mu\otimes 1 \mapsto \sum_{y \in (\S_{\mu(A)} \backslash \S_\mu)_{\min}}\v_{\mu(A)} \otimes y$.
Next up is
$$
\gamma:=\xi_{[1],f_{(1)}} 
\!\!\star \xi_{[2],f_{(2)}} \star \id_{(0)}
\star \xi_{[2],f_{(3)}}
\!\!\star \xi_{[3],f_{(4)}}
\!\!\star \xi_{[1],f_{(5)}}\!:
M(\mu(A))\otimes_{\kk \S_9} \aH_9
\rightarrow M(\mu(A))\otimes_{\kk \S_9} \aH_9,
$$
the right $\aH_9$-module homomorphism 
$\v_{\mu(A)} \otimes 1\mapsto \v_{\mu(A)} \otimes f$.
It takes
$\sum_{y \in (\S_{\mu(A)} \backslash \S_\mu)_{\min}}\v_{\mu(A)}\otimes y$
to
$\sum_{y \in (\S_{\mu(A)} \backslash \S_\mu)_{\min}}\v_{\mu(A)} \otimes f y$.
Then comes the generalized permutation
$$
\beta:M(\mu(A))\otimes_{\kk \S_9} \aH_9 \rightarrow M(\lambda(A))
\otimes_{\kk \S_9} \aH_9,\qquad
\v_{\mu(A)}\otimes 1 \mapsto \v_{\lambda(A)}\otimes d_A,
$$
producing the vector
$\sum_{y \in (\S_{\mu(A)} \backslash \S_\mu)_{\min}}\v_{\lambda(A)} \otimes d_A f y$.
Finally, the horizontal composition of three-fold
merges in the top portion of the diagram is
\begin{align*}
\alpha := \xi_{\left[\begin{smallmatrix}1&0&3&0&0&0\\0&0&0&2&2&1\end{smallmatrix}\right]}:
M(\lambda(A))\otimes_{\kk \S_9} \aH_9 &\twoheadrightarrow M(\lambda)\otimes_{\kk \S_9}\aH_9,&
\v_{\lambda(A)}\otimes 1 &\mapsto 
\v_\lambda\otimes 1.
\end{align*}
This takes our vector to 
$\sum_{y \in (\S_{\mu(A)} \backslash \S_\mu)_{\min}}\v_{\lambda} \otimes d_A f y$.
This is the same as the image of $\v_\mu\otimes 1$
under $\xi_{A,f}$ from \cref{inbits}, so 
the morphism defined by this string diagram is indeed equal to $\xi_{A,f}$. Thus, we have factored
$\xi_{A,f}$ as the composition
$\alpha\circ\beta\circ\gamma\circ\delta$.
\end{example}

The discussion in \cref{notaproof} makes it clear that all of the morphisms in a basis for any morphism space in $\aSchur$
can be obtained by vertical and horizontal composition of \cref{mergesandsplits},\cref{ms,spots,needles} plus appropriate identity morphisms. This proves the following lemma for $\aSchur$,
and similar considerations prove it for $\gSchur$.

\begin{lemma}\label{panty}
Morphisms in the $\kk$-linear monoidal categories 
$\gSchur$ and $\aSchur$ are generated
by \cref{mergesandsplits,ms,spots,needles}.
\end{lemma}

Next, we show that all of the $\aSchur$ relations are valid.

\begin{lemma}\label{believers}
The generating morphisms \cref{mergesandsplits},\cref{ms,spots,needles}
of $\aSchur$
satisfy all of the relations
\cref{secondone}, \cref{zeroforks}, \cref{assrel,alghom,mergesplit,fridaynight,deformedshuffle,jonrel,otherjon}.
\end{lemma}

\begin{proof}
The relations
\cref{secondone,zeroforks,assrel,mergesplit} are shown to hold in
$\Schur$ in \cite{BEEO}, hence, they also follow in $\aSchur$
since $\Schur$ is a monoidal subcategory.
The relations \cref{alghom} follow immediately since the
map $\P^{(r)} \rightarrow \End_{\dash \aH_r}(\kk \otimes_{\kk \S_r} \aH_r)$ defined by right multiplication is an algebra homomorphism.
The coproduct relation for merge follows because
the two-fold merge is the homomorphism mapping $\v_{(a,b)}\otimes 1 \mapsto \v_{(a+b)}\otimes 1$, and this commutes with right multiplication by any central element $f = f_{(1)} \otimes f_{(2)} \in \P_{a+b}^{\S_{a+b}}$.
The coproduct relation for split follows 
because the two-fold split is the homomorphism mapping
$\v_{(a+b)}\otimes 1 \mapsto \sum_{d \in (\S_a\times \S_b \backslash \S_{a+b})_{\min}} \v_{(a,b)} \otimes d$,
and again this commutes with right multiplication by $f \in \P_{a+b}^{\S_{a+b}}$.

To check the deformed shuffle relation \cref{deformedshuffle}, take $f \in \P_a^{\S_a}$ and $g \in \P_b^{\S_b}$. The morphism on the left hand side of \cref{deformedshuffle} maps
$$
\v_{(a+b)}\otimes 1 \mapsto 
\sum_{w \in (\S_{a+b} / \S_a \times \S_b)_{\min}}\v_{(a+b)} \otimes (f \otimes g) w^{-1}.
$$
By the second identity in \cref{poof}, 
we have that $\v_{(a+b)}\otimes (f \otimes g) s_i
= \v_{(a+b)} \otimes (s_i \bully (f \otimes g))$ (i.e., the module $M\big((a+b)\big)\otimes_{\kk \S_{a+b}} \aH_{a+b}$ is the {\em right} polynomial representation). So this expression equals
$$
\sum_{w \in (\S_{a+b} / \S_a \times \S_b)_{\min}}\v_{(a+b)} \otimes (w\bully (f \otimes g)),
$$
which is the image of $\v_{(a+b)}$ under the morphism on the right hand side of \cref{deformedshuffle}.

The derivations of the elementary dot slide relations
\cref{jonrel,otherjon} are more complicated and the proof will be explained in the next section. Specifically, these relations follow from \cref{bingley} below by equating coefficients in the generating functions there.
\end{proof}

Finally, we can explain the proofs of 
the main theorems establishing the presentations for $\gSchur$ and
$\aSchur$. We start with the easier $\gSchur$.

\begin{theorem}\label{grpres}
The current Schur category $\gSchur$ is the strict graded $\kk$-linear monoidal category obtained from $\Schur$ by adjoining the additional morphisms \cref{needles}
for all $r \geq 1$ and
homogeneous $f \in \P^{(r)}$,
subject to the additional relations \cref{alghom,fridaynight,shufflerel,easydotslide}.
\end{theorem}

\begin{proof}
Let $\gSchur'$ be the strict $\kk$-linear monoidal category defined by these generators and relations.
It is easy to see directly that all of the defining relations of $\gSchur'$ hold in $\gSchur$
(this can also be deduced from \cref{believers} by passing to the associated graded category).
Hence, there is a strict $\kk$-linear functor
$G:\gSchur' \rightarrow \gSchur$. 
It is bijective on objects by definition, and
\cref{panty} shows that it is full.
It just remains to show that $G$ is faithful.
To see this, we know 
bases for morphism spaces in $\gSchur$ by \cref{bots},
with basis elements represented by string diagrams.
It suffices to show that the morphisms in $\gSchur'$
defined by the same string diagrams span morphism spaces in $\gSchur'$. 
In view of \cref{panty}, this follows from the existence of a straightening algorithm which expresses the vertical composition (either way around) of a basis vector and a generator as a linear combination of basis vectors. 
This is similar to the algorithm 
for the Schur category explained
in the proof of \cite[Lem.~4.9]{BEEO}, using the local relations \cref{assrel,mergesplit}.
To modify it so that it can be applied in the current Schur category (where there are additional symmetric polynomials pinned to strings), 
one also needs to use the coproduct relations \cref{fridaynight} to slide symmetric polynomials across merges and splits from thick to thinner strings, the relations \cref{easydotslide} to slide
symmetric polynomials past crossings, and
the shuffle relation \cref{shufflerel} in place of the first relation from \cref{mergesplit}.
\end{proof}

\begin{theorem}\label{songwangpresentation}
The degenerate affine Schur category $\aSchur$ is isomorphic to the strict $\kk$-linear monoidal category obtained from $\Schur$ by adjoining the additional morphisms \cref{needles} for all $r \geq 1$ and
$f \in \P^{(r)}$,
subject to the additional relations \cref{alghom,fridaynight,deformedshuffle}
and any one of the four relations \cref{jonrel,otherjon}.
\end{theorem}

\begin{proof}
Let $\aSchur'$ be the strict $\kk$-linear monoidal category with these generators and relations.
\cref{believers} implies that there is a strict $\kk$-linear monoidal functor $F:\aSchur' \rightarrow \aSchur$.
\cref{panty} shows that this functor is full.
It is bijective on objects by definition.
It remains to show that $F$ is faithful.
There is a filtration on $\aSchur'$ defined by declaring that the generators 
$\begin{tikzpicture}[anchorbase,scale=.8]
\draw[ultra thick] (0,-.3) \botlabel{r} to (0,.3);
\pin{0,0}{.5,0}{f};
\end{tikzpicture}$ are of filtered degree equal to the usual degree of the symmetric polynomial $f$, and all other generators are of filtered degree 0. The functor $F$ is filtered, so it induces a functor
$\gr F:\gr \aSchur' \rightarrow \gr \aSchur$
between the associated graded categories. 
Using the presentation for $\gSchur$ from \cref{grpres}, it
follows that there is a full strict $\kk$-linear monoidal
functor $G:\gSchur \rightarrow \gr \aSchur'$
such that the composite $(\gr F) \circ G:\gSchur \rightarrow \aSchur$ is the isomorphism $\gSchur \stackrel{\sim}{\rightarrow}
\gr \aSchur$ discussed earlier. This implies that $\gr F$ is faithful, hence, so too is $F$.
\end{proof}

\begin{remark}\label{tobecontinued}
\cref{songwangpresentation} is a slightly modified version of the presentation for $\aSchur$ proved originally by Song and Wang in \cite{SongWang}.
The main advantage of our setup compared to \cite{SongWang} is that we allow arbitrary symmetric polynomials to be pinned to thick strings.
Also the relations \cref{jonrel,otherjon} for $d < a$ 
seem to be new.
Song and Wang observed that to present $\aSchur$, one only needs to impose these relations in the special case that $d=a$ (when all of the  elementary symmetric polynomials arising are of the same degree as the thickness of the strings that they are pinned to).
To make further comparison 
with the setup of \cite{SongWang}, we note that
the morphism denoted \begin{tikzpicture}[anchorbase]
\draw[ultra thick] (0,-.2) \botlabel{a+b} to (0,.2);
\multcloseddot{0,0}{west}{\varpi_a};
\end{tikzpicture}
in \cite{SongWang} is equal to
$$
\begin{tikzpicture}[anchorbase,scale=.8]
\draw[-,line width=2pt] (0.08,-.8) \botlabel{a+b} to (0.08,-.5);
\draw[-,line width=2pt] (0.08,.3) to (0.08,.6);
\draw[-,thick] (0.1,-.51) to [out=45,in=-45] (0.1,.31);
\draw[-,thick] (0.06,-.51) to [out=135,in=-135] (0.06,.31);
\strand{-.18,-.45}{a};
\pin{-.1,-.1}{-.7,-.1}{e_a};
\end{tikzpicture}
=
\sum_{w \in \left(\S_{a+b} / \S_a \times \S_{b}\right)_{\min}}
\begin{tikzpicture}[anchorbase,scale=.77]
\draw[-,line width=2pt] (0.08,-.8) \botlabel{a+b}to (0.08,.6);
\pin{.08,-.1}{1.45,-.1}{w \bully (x_1 \cdots x_a)};
\end{tikzpicture}
$$ 
in our notation. 
When $a,b\geq 1$, this is rather a complicated, inhomogeneous symmetric polynomial, but
it is equal to the elementary symmetric polynomial 
$e_a(x_1,\dots,x_{a+b})$ plus terms of lower degree, which is all that really matters.
To see that 
the monoidal category $\aSchur$ as we have defined it is indeed isomorphic to the monoidal category from \cite[Def.~2.1]{SongWang},
using the relations in $\aSchur$ from \cref{believers}, it follows easily that there is a strict $\kk$-linear monoidal functor
from the Song-Wang category to
$\aSchur$ mapping 
$\begin{tikzpicture}[anchorbase]
\draw[ultra thick] (0,-.2) \botlabel{r} to (0,.2);
\multcloseddot{0.01,0}{west}{\varpi_r};
\end{tikzpicture}\mapsto\begin{tikzpicture}[anchorbase]
\draw[ultra thick] (0,-.2) \botlabel{r} to (0,.2);
\pin{.01,0}{.5,0}{e_r};
\end{tikzpicture}$
and the other generators 
to the morphisms represented by the same diagrams in $\aSchur$.
This functor is an isomorphism because it maps the spanning sets
for morphism spaces from \cite[Prop.~3.6]{SongWang} to particular bases for morphism spaces in $\aSchur$ arising from \cref{basechange}.
\end{remark}

With generators and relations in hand, it follows 
that $\aSchur$ (hence, also $\Schur$ and $\gSchur$) has two natural symmetries
\begin{align}\label{symmetries}
\div:\aSchur^\op &\rightarrow \aSchur,&
\vid
:\aSchur^\rev &\rightarrow \aSchur.
\end{align}
The first of these is defined on string diagrams by reflecting in a horizontal axis; it takes $\xi_{A,f}$ to $\xi_{f,A^\transpose}$.
The second reflects in a vertical axis then multiplies by $(-1)^d$
where $d$ is the total 
degree of all of the symmetric polynomials present in the diagram (assumed homogeneous); it takes $\xi_{A,f}$ to $(-1)^{\deg(f)} \xi_{A^\dagger, f^\dagger}$ where $A^\dagger$ is obtained from $A$ by reversing the order of rows and columns and $f^\dagger$ is obtained from $f$ by replacing $x_1,\dots,x_r$ by $x_r,\dots,x_1$ (the number $r$ of variables is the sum of the entries of $A$).

%% file: s5-rels.tex
\setcounter{section}{4}

\section{Further relations}\label{s5-rels}

In this section, we prove a couple more relations in $\aSchur$ which require some more sophisticated technique.
The main point is to work systematically with generating functions, which typically will be formal Laurent series in an auxiliary variable $u$. For example,
working in $\P_r\lround u^{-1} \rround$, 
$$
(u-x_1)\cdots (u-x_r)
= u^r - e_1(x_1,\dots,x_r) u^{n-1}+\cdots +(-1)^r
e_r(x_1,\dots,x_r)
$$
is the generating function for the elementary 
symmetric polynomials, and
$$
\frac{1}{(u-x_1)\cdots (u-x_r)}
= 
u^{-r} + u^{-r-1} h_1(x_1,\dots,x_r)
+ u^{-r-2} h_2(x_1,\dots,x_r)+\cdots
$$
is the generating function for the complete symmetric polynomials.
We also use the convention that $e_0(x_1,\dots,x_r) = h_0(x_1,\dots,x_r) = 1$ for any $r \geq 0$.

We introduce the following shorthands for the pins involving these
generating functions:
\begin{align}\label{kward}
\begin{tikzpicture}[centerzero]
\draw[ultra thick] (0,-.4)\botlabel{r} to (0,.4);
\circledinverse{0,0}{u};
\end{tikzpicture}
&:=
\begin{tikzpicture}[centerzero]
\draw[ultra thick] (0,-.4)\botlabel{r} to (0,.4);
\pin{0,0}{1.4,0}{(u-x_1)\cdots (u-x_r)};
\end{tikzpicture}\ ,
&
\begin{tikzpicture}[centerzero]
\draw[ultra thick] (0,-.4)\botlabel{r} to (0,.4);
\circled{0,0}{u};
\end{tikzpicture}
&:=
\begin{tikzpicture}[centerzero]
\draw[ultra thick] (0,-.4)\botlabel{r} to (0,.4);
\pin{0,0}{1.4,0}{\frac{1}{(u-x_1)\cdots (u-x_r)}};
\end{tikzpicture}\ .
\end{align}
These are elements of $\End_{\aSchur}\big((r)\big)\lround u^{-1}\rround$.
They also make sense if $r=0$, when they are both equal to the identity endomorphism of the unit object $\one$.
By the coproduct relations 
\cref{fridaynight}, we have that
\begin{align}\label{game}
\begin{tikzpicture}[baseline = -.5mm,scale=1.4]
\draw[-,line width=1pt] (0.28,-.3)\botlabel{b} to (0.08,0.04);
\draw[-,line width=1pt] (-0.12,-.3)\botlabel{a} to (0.08,0.04);
\draw[-,line width=2pt] (0.08,.3) to (0.08,0);
\circledinverse{.2,-.15}{u};
\circledinverse{-.04,-.15}{u};
\end{tikzpicture} 
&=\begin{tikzpicture}[baseline = -.5mm,scale=1.4]
\draw[-,line width=1pt] (0.28,-.3)\botlabel{b} to (0.08,0.04);
\draw[-,line width=1pt] (-0.12,-.3)\botlabel{a} to (0.08,0.04);
\draw[-,line width=2pt] (0.08,.3) to (0.08,0);
\circledinverse{.08,.15}{u};
\end{tikzpicture}\ ,&
\begin{tikzpicture}[baseline = .5mm,scale=1.4]
\draw[-,line width=1pt] (0.28,.3)\toplabel{b} to (0.08,-0.04);
\draw[-,line width=1pt] (-0.12,.3)\toplabel{a} to (0.08,-0.04);
\draw[-,line width=2pt] (0.08,-.3) to (0.08,0);
\circledinverse{.2,.15}{u};
\circledinverse{-.04,.15}{u};
\end{tikzpicture} 
&=\begin{tikzpicture}[baseline = .5mm,scale=1.4]
\draw[-,line width=1pt] (0.28,.3)\toplabel{b} to (0.08,-0.04);
\draw[-,line width=1pt] (-0.12,.3)\toplabel{a} to (0.08,-0.04);
\draw[-,line width=2pt] (0.08,-.3) to (0.08,0);
\circledinverse{.08,-.15}{u};
\end{tikzpicture}\ ,\\
\begin{tikzpicture}[baseline = -.5mm,scale=1.4]
\draw[-,line width=1pt] (0.28,-.3)\botlabel{b} to (0.08,0.04);
\draw[-,line width=1pt] (-0.12,-.3)\botlabel{a} to (0.08,0.04);
\draw[-,line width=2pt] (0.08,.3) to (0.08,0);
\circled{.2,-.15}{u};
\circled{-.04,-.15}{u};
\end{tikzpicture} 
&=\begin{tikzpicture}[baseline = -.5mm,scale=1.4]
\draw[-,line width=1pt] (0.28,-.3)\botlabel{b} to (0.08,0.04);
\draw[-,line width=1pt] (-0.12,-.3)\botlabel{a} to (0.08,0.04);
\draw[-,line width=2pt] (0.08,.3) to (0.08,0);
\circled{.08,.15}{u};
\end{tikzpicture}\ ,&
\begin{tikzpicture}[baseline = .5mm,scale=1.4]
\draw[-,line width=1pt] (0.28,.3)\toplabel{b} to (0.08,-0.04);
\draw[-,line width=1pt] (-0.12,.3)\toplabel{a} to (0.08,-0.04);
\draw[-,line width=2pt] (0.08,-.3) to (0.08,0);
\circled{.2,.15}{u};
\circled{-.04,.15}{u};
\end{tikzpicture} 
&=\begin{tikzpicture}[baseline = .5mm,scale=1.4]
\draw[-,line width=1pt] (0.28,.3)\toplabel{b} to (0.08,-0.04);
\draw[-,line width=1pt] (-0.12,.3)\toplabel{a} to (0.08,-0.04);
\draw[-,line width=2pt] (0.08,-.3) to (0.08,0);
\circled{.08,-.15}{u};
\end{tikzpicture}\ .
\label{istanbul}
\end{align}
Note also that the symmetry $\div$
from \cref{symmetries} fixes both of the generating functions in \cref{kward}, while $\vid$ maps them to
\begin{align}
(-1)^r \begin{tikzpicture}[centerzero]
\draw[ultra thick] (0,-.4)\botlabel{r} to (0,.4);
\ovaledinverse{0,0}{-u};
\end{tikzpicture}
&=
\begin{tikzpicture}[centerzero]
\draw[ultra thick] (0,-.4)\botlabel{r} to (0,.4);
\pin{0,0}{1.4,0}{(u+x_1)\cdots (u+x_r)};
\end{tikzpicture}\ ,
&
(-1)^r \begin{tikzpicture}[centerzero]
\draw[ultra thick] (0,-.4)\botlabel{r} to (0,.4);
\ovaled{0,0}{-u};
\end{tikzpicture}
&=
\begin{tikzpicture}[centerzero]
\draw[ultra thick] (0,-.4)\botlabel{n} to (0,.4);
\pin{0,0}{1.4,0}{\frac{1}{(u+x_1)\cdots (u+x_r)}};
\end{tikzpicture}\ ,
\end{align}
respectively.

\begin{theorem}\label{bingley}
The following hold in $\aSchur[u]$ for $a,b \geq 0$:
\begin{enumerate}
\item
$\displaystyle
\begin{tikzpicture}[anchorbase,scale=1.2,centerzero=.4]
\draw[-,line width=1.1pt] (0.58,0) to (.02,1);
\draw[-,line width=1.1pt] (0.02,0) to (.58,1);
\strand{0,-.1}{a};
\strand{0.6,-.1}{b};
\circledinverse{.14,.22}{u};
\end{tikzpicture}
=
\sum_{s=0}^{\min(a,b)}
(-1)^s s!
\begin{tikzpicture}[anchorbase,scale=1.2,centerzero=.4]
\draw[-,thick] (0.58,0) to (0.58,.2) to (.02,.8) to (.02,1);
\draw[-,thick] (0.02,0) to (0.02,.2) to (.58,.8) to (.58,1);
\draw[-,thin] (0,0) to (0,1);
\draw[-,thin] (0.6,0) to (0.6,1);
\strand{0,-.1}{a};
\strand{0.6,-.1}{b};
\strand{-0.1,.5}{s};
\strand{0.75,.5}{s};
\circledinverse{.45,.65}{u};
\end{tikzpicture}$.
\item
$\displaystyle
\begin{tikzpicture}[anchorbase,scale=1.2,centerzero=.4]
\draw[-,line width=1.1pt] (0.58,0) to (.02,1);
\draw[-,line width=1.1pt] (0.02,0) to (.58,1);
\strand{0,-.1}{a};
\strand{0.6,-.1}{b};
\circledinverse{.14,.78}{u};
\end{tikzpicture}
=
\sum_{s=0}^{\min(a,b)}
(-1)^s s!
\begin{tikzpicture}[anchorbase,scale=1.2,centerzero=.4]
\draw[-,thick] (0.58,0) to (0.58,.2) to (.02,.8) to (.02,1);
\draw[-,thick] (0.02,0) to (0.02,.2) to (.58,.8) to (.58,1);
\draw[-,thin] (0,0) to (0,1);
\draw[-,thin] (0.6,0) to (0.6,1);
\strand{0,-.1}{a};
\strand{0.6,-.1}{b};
\strand{-0.1,.5}{s};
\strand{0.75,.5}{s};
\circledinverse{.45,.35}{u};
\end{tikzpicture}$.
\item
$\displaystyle
\begin{tikzpicture}[anchorbase,scale=1.2,centerzero=.4]
\draw[-,line width=1.1pt] (0.58,0) to (.02,1);
\draw[-,line width=1.1pt] (0.02,0) to (.58,1);
\strand{0,-.1}{a};
\strand{0.6,-.1}{b};
\circledinverse{.46,.22}{u};
\end{tikzpicture}
=
\sum_{s=0}^{\min(a,b)}s!
\begin{tikzpicture}[anchorbase,scale=1.2,centerzero=.4]
\draw[-,thick] (0.58,0) to (0.58,.2) to (.02,.8) to (.02,1);
\draw[-,thick] (0.02,0) to (0.02,.2) to (.58,.8) to (.58,1);
\draw[-,thin] (0,0) to (0,1);
\draw[-,thin] (0.6,0) to (0.6,1);
\strand{0,-.1}{a};
\strand{0.6,-.1}{b};
\strand{-0.1,.5}{s};
\strand{0.75,.5}{s};
\circledinverse{.15,.65}{u};
\end{tikzpicture}$.
\item
$\displaystyle
\begin{tikzpicture}[anchorbase,scale=1.2,centerzero=.4]
\draw[-,line width=1.1pt] (0.58,0) to (.02,1);
\draw[-,line width=1.1pt] (0.02,0) to (.58,1);
\strand{0,-.1}{a};
\strand{0.6,-.1}{b};
\circledinverse{.46,.78}{u};
\end{tikzpicture}
=
\sum_{s=0}^{\min(a,b)}s!
\begin{tikzpicture}[anchorbase,scale=1.2,centerzero=.4]
\draw[-,thick] (0.58,0) to (0.58,.2) to (.02,.8) to (.02,1);
\draw[-,thick] (0.02,0) to (0.02,.2) to (.58,.8) to (.58,1);
\draw[-,thin] (0,0) to (0,1);
\draw[-,thin] (0.6,0) to (0.6,1);
\strand{0,-.1}{a};
\strand{0.6,-.1}{b};
\strand{-0.1,.5}{s};
\strand{0.75,.5}{s};
\circledinverse{.15,.35}{u};
\end{tikzpicture}$.
\end{enumerate}
\end{theorem}

\begin{proof}
We prove (1). The proofs of (2)--(4) are similar, or they can be deduced from (1) by applying the symmetries $\div$ and $\vid$ from \cref{symmetries}.
It suffices to prove the relation in the case that the ground ring $\kk$ is $\Z$---the relation over any other ground ring follows from this case by basis change.
In turn, to prove it over $\Z$, we can extend scalars to $\Q$.
We assume this from now on, and proceed to prove (1) 
by induction on $a+b$. 
The relation is trivial if $a=0$ or $b=0$. The base case $a=b=1$
follows easily from \cref{demazurerel}.
The following inductive calculation 
proves the relation for $a=1$ and $b > 1$:
\begin{align*}
\begin{tikzpicture}[anchorbase,scale=1.2]
\draw (-.3,-.5) to (.3,.5);
\circledinverse{-.15,-.25}{u};
\draw[line width=2pt] (-.3,.5) to (.3,-.5)\botlabel{b};
\end{tikzpicture}&\ \substack{\cref{mergesplit}\\\scriptstyle =\\\cref{sliders}}\ 
\frac{1}{b}\ 
\begin{tikzpicture}[anchorbase,scale=1.2]
\draw (-.3,-.5) to (.3,.5);
\circledinverse{-.19,-.31}{u};
\draw (-.333,.48) to (-.23,.28) to [out=-90,in=165,looseness=1] (.16,-.3) to (.27,-.515);
\draw[line width=1.9pt] 
(-.3,.5) to (-.2,.3) to [out=-15,in=90,looseness=1] (.2,-.3) to (.3,-.5) \botlabel{b};
\end{tikzpicture}
\stackrel{\cref{demazurerel}}{=}\frac{1}{b}\ 
\begin{tikzpicture}[anchorbase,scale=1.2]
\draw (-.3,-.5) to (.3,.5);
\circledinverse{-.02,-.03}{u};
\draw (-.3333,.48) to (-.23,.28) to [out=-90,in=165,looseness=1] (.16,-.3) to (.27,-.515);
\draw[line width=1.9pt] 
(-.3,.5) to (-.2,.3) to [out=-15,in=90,looseness=1] (.2,-.3) to (.3,-.5) \botlabel{b};
\end{tikzpicture}
-
\frac{1}{b}
\begin{tikzpicture}[anchorbase,scale=1.2]
\draw (-.333,.48) to (-.23,.28) to [out=-90,in=90,looseness=1] (-.3,-.5);
\draw[line width=1.9pt] 
(-.3,.5) to (-.2,.3) to [out=-15,in=90,looseness=1] (.2,-.3) to (.3,-.5) \botlabel{b};
\draw (.17,-.3) to [out=175,in=-90,looseness=1.5] (.3,.5);
\end{tikzpicture}
\stackrel{\cref{swallows}}{=}
\frac{1}{b}
\begin{tikzpicture}[anchorbase,scale=1.2]
\draw (-.3,-.5) to (.3,.5);
\circledinverse{.2,.34}{u};
\draw (-.333,.48) to (-.23,.28) to [out=-90,in=165,looseness=1] (.16,-.3) to (.27,-.515);
\draw[line width=1.9pt] 
(-.3,.5) to (-.2,.3) to [out=-15,in=90,looseness=1] (.2,-.3) to (.3,-.5) \botlabel{b};
\end{tikzpicture}-
\frac{1}{b}
\begin{tikzpicture}[anchorbase,scale=1.2]
\draw (-.325,-.5) to (-.325,.5);
\draw (.325,-.5) to (.325,.5);
\draw[line width=1.8] (.3,-.5)\botlabel{b} to (.3,-.2) to (-.3,.2) to (-.3,.5);
\draw[thin] (.28,-.5) to (.28,-.35) to (-.3,0) to [out=150,in=-170,looseness=2] (-.32,.35) to (-.32,.5);
\end{tikzpicture}
-
\frac{1}{b}
\begin{tikzpicture}[anchorbase,scale=1.2]
\draw (-.325,-.5) to (-.325,.5);
\draw (.325,-.5) to (.325,.5);
\draw[line width=1.8] (.3,-.5)\botlabel{b} to (.3,-.25) to (-.3,.25) to (-.3,.5);
\end{tikzpicture}\\
&\ \substack{\cref{assrel}\\\scriptstyle =\\ \cref{mergesplit}}\ 
\begin{tikzpicture}[anchorbase,scale=1.2]
\draw (-.3,-.5) to (.3,.5);
\circledinverse{.17,.29}{u};
\draw[line width=1.9pt] 
(-.3,.5) to (.3,-.5) \botlabel{b};
\end{tikzpicture}-
\frac{1}{b}
\begin{tikzpicture}[anchorbase,scale=1.2]
\draw (-.32,-.5) to (-.32,.5);
\draw (.325,-.5) to (.325,.5);
\draw[line width=1.8] (.3,-.5)\botlabel{b} to (.3,0) to (-.3,.4) to (-.3,.5);
\draw[thin] (.28,-.5) to (.28,-.35) to (-.3,-.1) to [out=160,in=-170,looseness=2] (-.325,.25) to (-.325,.5);
\end{tikzpicture}
-
\frac{1}{b}
\begin{tikzpicture}[anchorbase,scale=1.2]
\draw (-.325,-.5) to (-.325,.5);
\draw (.325,-.5) to (.325,.5);
\draw[line width=1.8] (.3,-.5)\botlabel{b} to (.3,-.25) to (-.3,.25) to (-.3,.5);
\end{tikzpicture}
\stackrel{\cref{swallows}}{=}
\begin{tikzpicture}[anchorbase,scale=1.2]
\draw (-.3,-.5) to (.3,.5);
\circledinverse{.17,.29}{u};
\draw[line width=1.9pt] 
(-.3,.5) to (.3,-.5) \botlabel{b};
\end{tikzpicture}-
\frac{1}{b}
\begin{tikzpicture}[anchorbase,scale=1.2]
\draw (-.325,-.5) to (-.325,.5);
\draw (.325,-.5) to (.325,.5);
\draw[line width=1.8] (.3,-.5)\botlabel{b} to (.3,0) to (-.3,.35) to (-.3,.5)\toplabel{b};
\draw[thin] (.28,-.5) to (.28,-.35) to (-.32,0) to (-.32,.5);
\end{tikzpicture}
-
\frac{1}{b}
\begin{tikzpicture}[anchorbase,scale=1.2]
\draw (-.325,-.5) to (-.325,.5);
\draw (.325,-.5) to (.325,.5);
\draw[line width=1.8] (.3,-.5)\botlabel{b} to (.3,-.25) to (-.3,.25) to (-.3,.5)\toplabel{b};
\end{tikzpicture}
\ \substack{\cref{assrel}\\\scriptstyle =\\\cref{mergesplit}}\ 
\begin{tikzpicture}[anchorbase,scale=1.2]
\draw (-.3,-.5) to (.3,.5);
\circledinverse{.17,.29}{u};
\draw[line width=1.9pt] 
(-.3,.5) to (.3,-.5) \botlabel{b};
\end{tikzpicture}-
\begin{tikzpicture}[anchorbase,scale=1.2]
\draw (-.325,-.5) to (-.325,.5);
\draw (.325,-.5) to (.325,.5);
\draw[line width=1.8] (.3,-.5)\botlabel{b} to (.3,-.25) to (-.3,.25) to (-.3,.5)\toplabel{b};
\end{tikzpicture}\ .
\end{align*}
Then the following inductive calculation proves
the relation for $a > 1$ and $b \geq 1$:
\begin{align*}
\begin{tikzpicture}[anchorbase,centerzero=.5,scale=1.2]
\draw[-,line width=1.5pt] (0.58,0) to (.02,1);
\draw[-,line width=1.5pt] (0.02,0) to (.58,1);
\strand{0,-.1}{a};
\strand{0.6,-.1}{b};
\circledinverse{.14,.22}{u};
\end{tikzpicture}
&\ \substack{\cref{deformedshuffle}\\\scriptstyle =\\\cref{sliders}}\ 
\frac{1}{a}\ 
\begin{tikzpicture}[anchorbase,scale=1.2]
\draw[line width=1.4pt] (.3,-.5)\botlabel{b} to (-.3,.5);
\draw[line width=1.4pt] (.32,.5) to (.23,.3) to [out=-90,in=15,looseness=1] (-.17,-.3) to (-.28,-.5)\botlabel{a};
\draw 
(.3,.511) to (.205,.3) to [out=195,in=90,looseness=1] (-.19,-.29) to (-.3,-.49);
\circledinverse{-.156,-.04}{u};
\circledinverse{-.03,-.24}{u};
\end{tikzpicture}
=
\frac{1}{a}
\begin{tikzpicture}[anchorbase,scale=1.2]
\draw[line width=1.4pt] (.3,-.5)\botlabel{b} to (-.3,.5);
\draw[line width=1.4pt] (.32,.5) to (.23,.3) to [out=-90,in=15,looseness=1] (-.17,-.3) to (-.28,-.5)\botlabel{a};
\draw 
(.3,.511) to (.205,.3) to [out=195,in=90,looseness=1] (-.19,-.29) to (-.3,-.49);
\circledinverse{.06,.24}{u};
\circledinverse{-.03,-.24}{u};
\end{tikzpicture}
-\frac{1}{a}\ 
\begin{tikzpicture}[anchorbase,scale=.8]
\draw (-.325,-1) to (-.325,.5);
\draw (.325,-1) to (.325,.5);
\draw[line width=1.8] (.3,-1)\botlabel{b} to (.3,-.25) to (-.3,.25) to (-.3,.5);
\draw[line width=1.6] (-.285,-1) \botlabel{a\,}to (-.285,-.8) to [out=30,in=-30,looseness=2] (.355,.2) to (.355,.5);
\circledinverse{-.03,-.66}{u};
\end{tikzpicture}\!\!
\ \substack{\cref{sliders}\\\scriptstyle = \\\cref{swallows}}\ 
\frac{1}{a}
\begin{tikzpicture}[anchorbase,scale=1.2]
\draw[line width=1.4pt] (.3,-.5)\botlabel{b} to (-.3,.5);
\draw[line width=1.4pt] (.32,.5) to (.23,.3) to [out=-90,in=15,looseness=1] (-.17,-.3) to (-.28,-.5)\botlabel{a};
\draw 
(.3,.511) to (.205,.3) to [out=195,in=90,looseness=1] (-.19,-.29) to (-.3,-.49);
\circledinverse{.06,.24}{u};
\circledinverse{-.03,-.24}{u};
\end{tikzpicture}
-\frac{1}{a}\begin{tikzpicture}[anchorbase,scale=1.2,centerzero=.6]
\draw[-,line width=1.1] (0.58,0) to (0.58,.2) to (.02,.8) to (.02,1);
\draw[-,line width=1.1] (0.02,0) to (0.02,.2) to (.58,.8) to (.58,1);
\draw[-,thin] (0,0) to (0,1);
\draw[-,thin] (0.6,0) to (0.6,1);
\strand{0,-.1}{a};
\strand{0.6,-.1}{b};
\circledinverse{.18,.37}{u};
\end{tikzpicture}\\
&=
\frac{1}{a}
\sum_{s=0}^{\min(a-1,b)}
(-1)^s s!
\begin{tikzpicture}[anchorbase,scale=1.2,centerzero=.6]
\draw[-,line width=1.1] (0.58,0) to (0.58,.25) to (.02,.6) to (.02,1);
\draw[-,line width=1.1] (0.02,0) to (0.02,.25) to (.58,.6) to (.58,1);
\draw[-,thin] (0,0) to (0,1);
\draw[-,thin] (0.6,0) to (0.6,1);
\strand{0,-.1}{a};
\strand{0.6,-.1}{b};
\strand{-0.08,.45}{s};
\strand{0.7,.45}{s};
\draw (-.01,0) to (-.01,.14) to [out=170,in=200,looseness=1.5] (-.01,.7) to (.56,.9) to (.56,1);
\circledinverse{.45,.52}{u};
\circledinverse{.4,.84}{u};
\end{tikzpicture}
-\frac{1}{a}
\sum_{s=0}^{\min(a-1,b-1)}
(-1)^s s!
\begin{tikzpicture}[anchorbase,scale=1.2,centerzero=.6]
\draw[-,line width=1.1] (0.58,0) to (0.58,.3) to (.02,.7) to (.02,1);
\draw[-,line width=1.1] (0.02,0) to (0.02,.3) to (.58,.7) to (.58,1);
\draw[-,line width=.8] (0,0) to (0,1);
\draw[-,line width=.8] (0.6,0) to (0.6,1);
\strand{0,-.1}{a};
\strand{0.6,-.1}{b};
\strand{-0.08,.5}{s};
\strand{0.7,.5}{s};
\draw (-.01,0) to (-.01,.15) to [out=160,in=200,looseness=1.5] (-.01,.85) to (-.01,1);
\draw (.61,0) to (.61,.15) to [out=20,in=-20,looseness=1.5] (.61,.85) to (.61,1);
\circledinverse{.45,.6}{u};
\end{tikzpicture}
\\
&\ \substack{\cref{assrel}\\\scriptstyle =\\\cref{swallows}}\ 
\frac{1}{a}
\sum_{s=0}^{\min(a-1,b)}
(-1)^s s!
\begin{tikzpicture}[anchorbase,scale=1.2,centerzero=.6]
\draw[-,line width=1.1] (0.58,0) to (0.58,.1) to (.01,.6) to (.01,1);
\draw[-,line width=1.1] (0.014,0) to (0.014,.1) to (.58,.45) to (.58,1);
\draw[-,thin] (-0.01,0) to (-0.01,1);
\draw[-,thin] (0.6,0) to (0.6,1);
\strand{0,-.1}{a};
\strand{0.6,-.1}{b};
\strand{-0.08,.5}{s};
\strand{0.7,.3}{s};
\draw (0,0) to (0,.4) to (.565,.75) to (.565,1);
\circledinverse{.46,.39}{u};
\circledinverse{.35,.62}{u};
\end{tikzpicture}
-\frac{1}{a}
\sum_{s=0}^{\min(a-1,b-1)}
(-1)^s s!
\begin{tikzpicture}[anchorbase,scale=1.2,centerzero=.6]
\draw[-,line width=1.1] (0.58,0) to (0.58,.15) to (.02,.85) to (.02,1);
\draw[-,line width=1.1] (0.02,0) to (0.02,.15) to (.58,.85) to (.58,1);
\draw[-,line width=.8] (0,0) to (0,1);
\draw[-,line width=.8] (0.6,0) to (0.6,1);
\strand{0,-.1}{a};
\strand{0.6,-.1}{b};
\strand{0.1,.5}{s};
\strand{0.5,.5}{s};
\draw (-.01,0) to (-.01,.3) to [out=160,in=200,looseness=1.5] (-.01,.7) to (-.01,1);
\draw (.61,0) to (.61,.3) to [out=20,in=-20,looseness=1.5] (.61,.7) to (.61,1);
\circledinverse{.45,.67}{u};
\end{tikzpicture}\\
&\ \substack{\cref{assrel}\\\scriptstyle =\\\cref{deformedshuffle}}\ 
\frac{1}{a}
\sum_{s=0}^{\min(a-1,b)}
(-1)^s s!(a-s)
\begin{tikzpicture}[anchorbase,scale=1.2,centerzero=.6]
\draw[-,line width=1.1] (0.58,0) to (0.58,.25) to (.01,.75) to (.01,1);
\draw[-,line width=1.1] (0.014,0) to (0.014,.25) to (.58,.75) to (.58,1);
\draw[-,line width=.8] (-0.01,0) to (-0.01,1);
\draw[-,line width=.8] (0.6,0) to (0.6,1);
\strand{0,-.1}{a};
\strand{0.6,-.1}{b};
\strand{-0.08,.5}{s};
\strand{0.7,.5}{s};
\circledinverse{.45,.63}{u};
\end{tikzpicture}
+\frac{1}{a}
\sum_{s=0}^{\min(a-1,b-1)}
(-1)^{s+1} (s+1)!(s+1)\begin{tikzpicture}[anchorbase,scale=1.2,centerzero=.6]
\draw[-,line width=1.1] (0.58,0) to (0.58,.25) to (.01,.75) to (.01,1);
\draw[-,line width=1.1] (0.014,0) to (0.014,.25) to (.58,.75) to (.58,1);
\draw[-,line width=.8] (-0.01,0) to (-0.01,1);
\draw[-,line width=.8] (0.6,0) to (0.6,1);
\strand{0,-.1}{a};
\strand{0.6,-.1}{b};
\strand{-0.16,.5}{s\!+\!1};
\strand{0.84,.5}{s\!+\!1};
\circledinverse{.45,.63}{u};
\end{tikzpicture}\\
&=\frac{1}{a}
\sum_{s=0}^{\min(a,b)}
(-1)^s s!(a-s)
\begin{tikzpicture}[anchorbase,scale=1.2,centerzero=.6]
\draw[-,line width=1.1] (0.58,0) to (0.58,.25) to (.01,.75) to (.01,1);
\draw[-,line width=1.1] (0.014,0) to (0.014,.25) to (.58,.75) to (.58,1);
\draw[-,line width=.8] (-0.01,0) to (-0.01,1);
\draw[-,line width=.8] (0.6,0) to (0.6,1);
\strand{0,-.1}{a};
\strand{0.6,-.1}{b};
\strand{-0.09,.5}{s};
\strand{0.7,.5}{s};
\circledinverse{.45,.63}{u};
\end{tikzpicture}
+\frac{1}{a}
\sum_{s=0}^{\min(a,b)}
(-1)^s s!s\begin{tikzpicture}[anchorbase,scale=1.2,centerzero=.6]
\draw[-,line width=1.1] (0.58,0) to (0.58,.25) to (.01,.75) to (.01,1);
\draw[-,line width=1.1] (0.014,0) to (0.014,.25) to (.58,.75) to (.58,1);
\draw[-,line width=.8] (-0.01,0) to (-0.01,1);
\draw[-,line width=.8] (0.6,0) to (0.6,1);
\strand{0,-.1}{a};
\strand{0.6,-.1}{b};
\strand{-0.09,.5}{s};
\strand{0.7,.5}{s};
\circledinverse{.45,.63}{u};
\end{tikzpicture}
=
\sum_{s=0}^{\min(a,b)}
(-1)^s s!\begin{tikzpicture}[anchorbase,scale=1.2,centerzero=.6]
\draw[-,line width=1.1] (0.58,0) to (0.58,.25) to (.01,.75) to (.01,1);
\draw[-,line width=1.1] (0.014,0) to (0.014,.25) to (.58,.75) to (.58,1); =
\draw[-,line width=.8] (-0.01,0) to (-0.01,1);
\draw[-,line width=.8] (0.6,0) to (0.6,1);
\strand{0,-.1}{a};
\strand{0.6,-.1}{b};
\strand{-0.09,.5}{s};
\strand{0.7,.5}{s};
\circledinverse{.45,.63}{u};
\end{tikzpicture}\ .
\end{align*}
\end{proof}

\begin{corollary}\label{rome}
$\begin{tikzpicture}[anchorbase,scale=1,centerzero=.4]
\draw[-,line width=1.1pt] (0.58,0) to (.02,1);
\draw[-,line width=1.1pt] (0.02,0) to (.58,1);
\strand{0,-.1}{a};
\strand{0.6,-.1}{b};
\circledinverse{.14,.22}{u};
\circledinverse{.46,.22}{u};
\end{tikzpicture}
=
\begin{tikzpicture}[anchorbase,scale=1,centerzero=.4]
\draw[-,line width=1.1pt] (0.58,0) to (.02,1);
\draw[-,line width=1.1pt] (0.02,0) to (.58,1);
\strand{0,-.1}{a};
\strand{0.6,-.1}{b};
\circledinverse{.46,.78}{u};
\circledinverse{.14,.78}{u};
\end{tikzpicture}$
and 
$\begin{tikzpicture}[anchorbase,scale=1,centerzero=.4]
\draw[-,line width=1.1pt] (0.58,0) to (.02,1);
\draw[-,line width=1.1pt] (0.02,0) to (.58,1);
\strand{0,-.1}{a};
\strand{0.6,-.1}{b};
\circled{.14,.22}{u};
\circled{.46,.22}{u};
\end{tikzpicture}
=
\begin{tikzpicture}[anchorbase,scale=1,centerzero=.4]
\draw[-,line width=1.1pt] (0.58,0) to (.02,1);
\draw[-,line width=1.1pt] (0.02,0) to (.58,1);
\strand{0,-.1}{a};
\strand{0.6,-.1}{b};
\circled{.46,.78}{u};
\circled{.14,.78}{u};
\end{tikzpicture}$.
\end{corollary}

\begin{proof}
\cref{bingley}(1),(3) and \cref{game} give that
\begin{align*}
\begin{tikzpicture}[anchorbase,scale=1.2,centerzero=.4]
\draw[-,line width=1.1pt] (0.58,0) to (.02,1);
\draw[-,line width=1.1pt] (0.02,0) to (.58,1);
\strand{0,-.1}{a};
\strand{0.6,-.1}{b};
\circledinverse{.14,.22}{u};
\circledinverse{.46,.22}{u};
\end{tikzpicture}&=
\sum_{s=0}^{\min(a,b)}
(-1)^s s!
\begin{tikzpicture}[anchorbase,scale=1.2,centerzero=.4]
\draw[-,thick] (0.58,-.1)\botlabel{b} to (0.58,.2) to (.02,.8) to (.02,1.1);
\draw[-,thick] (0.02,-.1)\botlabel{a} to (0.02,.2) to (.58,.8) to (.58,1.1);
\draw[-,thin] (-0.0,-.1) to (-0.0,1.1);
\draw[-,thin] (0.6,-.1) to (0.6,1.1);
\strand{-0.1,.5}{s};
\strand{0.7,.5}{s};
\circledinverse{.45,.65}{u};
\circledinverse{.58,.05}{u};
\end{tikzpicture}=
\sum_{s=0}^{\min(a,b)}
(-1)^s s!
\begin{tikzpicture}[anchorbase,scale=1.2,centerzero=.4]
\draw[-,thick] (0.58,-.1)\botlabel{b} to (0.58,.2) to (.02,.8) to (.02,1.1);
\draw[-,thick] (0.02,-.1)\botlabel{a} to (0.02,.2) to (.58,.8) to (.58,1.1);
\draw[-,thin] (-0.0,-.1) to (-0.0,1.1);
\draw[-,thin] (0.6,-.1) to (0.6,1.1);
\strand{-0.1,.5}{s};
\strand{0.7,.7}{s};
\circledinverse{.42,.35}{u};
\circledinverse{.42,.65}{u};
\circledinverse{.61,.5}{u};
\end{tikzpicture}=
\sum_{s=0}^{\min(a,b)}
(-1)^s s!
\begin{tikzpicture}[anchorbase,scale=1.2,centerzero=.4]
\draw[-,thick] (0.58,-.1)\botlabel{b} to (0.58,.2) to (.02,.8) to (.02,1.1);
\draw[-,thick] (0.02,-.1)\botlabel{a} to (0.02,.2) to (.58,.8) to (.58,1.1);
\draw[-,thin] (-0.0,-.1) to (-0.0,1.1);
\draw[-,thin] (0.6,-.1) to (0.6,1.1);
\strand{-0.1,.5}{s};
\strand{0.7,.5}{s};
\circledinverse{.42,.35}{u};
\circledinverse{.6,.95}{u};
\end{tikzpicture}\\
&=\!\!\!\sum_{s=0}^{\min(a,b)} \sum_{t=0}^{\min(a,b)-s}
\!\!\!(-1)^s s! t!
\begin{tikzpicture}[anchorbase,scale=1.2,centerzero=.4]
\draw[-,thick] (1.08,-.1)\botlabel{b} to (1.08,.2) to (-.48,.8) to (-.48,1.1);
\draw[-,thick] (-0.48,-.1)\botlabel{a} to (-0.48,.2) to (1.08,.8) to (1.08,1.1);
\draw[-,thin] (1.1,-.1) to (1.1,1.1);
\draw[-,thin] (-0.5,-.1) to (-0.5,1.1);
\draw[-,thin] (-.49,-.1) to (-.49,.21) to (-.1,.36) to (-.1,.64) to (-.49,.79) to (-.49,1.1);
\draw[-,thin] (1.09,-.1) to (1.09,.21) to (.7,.36) to (.7,.64) to (1.09,.79) to (1.09,1.1);
\strand{-0.6,.5}{s};
\strand{1.2,.5}{s};
\strand{-0.2,.5}{t};
\strand{.8,.5}{t};
\circledinverse{1.1,.95}{u};
\circledinverse{0.08,.58}{u};
\end{tikzpicture}
\!=\!\!\!\sum_{s=0}^{\min(a,b)} \sum_{t=0}^{\min(a,b)-s}
\!\!\!(-1)^s (s+t)! \binom{s+t}{s}\!\!\!
\begin{tikzpicture}[anchorbase,scale=1.2,centerzero=.4]
\draw[-,thick] (0.58,-.1)\botlabel{b} to (0.58,.2) to (.02,.8) to (.02,1.1);
\draw[-,thick] (0.02,-.1)\botlabel{a} to (0.02,.2) to (.58,.8) to (.58,1.1);
\draw[-,thin] (-0.0,-.1) to (-0.0,1.1);
\draw[-,thin] (0.6,-.1) to (0.6,1.1);
\strand{-0.25,.5}{s+t};
\strand{0.85,.5}{s+t};
\circledinverse{.15,.65}{u};
\circledinverse{.58,.95}{u};
\end{tikzpicture}
\\&=
\sum_{n=0}^{\min(a,b)} n! \left(\sum_{s=0}^{n}
(-1)^s \binom{n}{s}\right)
\begin{tikzpicture}[anchorbase,scale=1.2,centerzero=.4]
\draw[-,thick] (0.58,-.1)\botlabel{b} to (0.58,.2) to (.02,.8) to (.02,1.1);
\draw[-,thick] (0.02,-.1)\botlabel{a} to (0.02,.2) to (.58,.8) to (.58,1.1);
\draw[-,thin] (-0.0,-.1) to (-0.0,1.1);
\draw[-,thin] (0.6,-.1) to (0.6,1.1);
\strand{-0.15,.5}{n};
\strand{0.75,.5}{n};
\circledinverse{.15,.65}{u};
\circledinverse{.58,.95}{u};
\end{tikzpicture}
=
\sum_{n=0}^{\min(a,b)} n! (1-1)^n
\begin{tikzpicture}[anchorbase,scale=1.2,centerzero=.4]
\draw[-,thick] (0.58,-.1)\botlabel{b} to (0.58,.2) to (.02,.8) to (.02,1.1);
\draw[-,thick] (0.02,-.1)\botlabel{a} to (0.02,.2) to (.58,.8) to (.58,1.1);
\draw[-,thin] (-0.0,-.1) to (-0.0,1.1);
\draw[-,thin] (0.6,-.1) to (0.6,1.1);
\strand{-0.15,.5}{n};
\strand{0.75,.5}{n};
\circledinverse{.15,.65}{u};
\circledinverse{.58,.95}{u};
\end{tikzpicture}
=\begin{tikzpicture}[anchorbase,scale=1.2,centerzero=.4]
\draw[-,line width=1.1pt] (0.58,0) to (.02,1);
\draw[-,line width=1.1pt] (0.02,0) to (.58,1);
\strand{0,-.1}{a};
\strand{0.6,-.1}{b};
\circledinverse{.14,.78}{u};
\circledinverse{.46,.78}{u};
\end{tikzpicture}.
\end{align*}
This proves the first identity.
The second follows from the first on composing on the bottom 
with 
$\begin{tikzpicture}[scale=1.1,centerzero]
\draw[line width=1.2pt] (0.1,-.23)\botlabel{b} to (0.1,.23);
\draw[line width=1.2pt] (-.3,-.23)\botlabel{a} to (-0.3,.23);
\circled{-.3,0}{u};
\circled{0.1,0}{u};
\end{tikzpicture}$ and on the top with
with 
$\begin{tikzpicture}[scale=1.1,centerzero]
\draw[line width=1.2pt] (0.1,-.23)\botlabel{a} to (0.1,.23);
\draw[line width=1.2pt] (-.3,-.23)\botlabel{b} to (-0.3,.23);
\circled{-.3,0}{u};
\circled{0.1,0}{u};
\end{tikzpicture}$.
\end{proof}

\begin{corollary}\label{averagedotslide}
The following hold in $\aSchur\llbracket u^{-1}\rrbracket$ for $r \geq 1$:
\begin{enumerate}
\item
$\displaystyle
\begin{tikzpicture}[anchorbase,scale=1.4,centerzero=.6]
\draw[-,thin] (0.58,0.2) to (.02,.8);
\draw[-,line width=1.3pt] (0.3,0.2)\botlabel{r} to (.3,.8);
\circled{0.14,.68}{u};
\end{tikzpicture}
=
\begin{tikzpicture}[anchorbase,scale=1.4,centerzero=.6]
\draw[-,thin] (0.58,0.2) to (.02,.8);
\draw[-,line width=1.3pt] (0.3,0.2)\botlabel{r} to (.3,.8);
\circled{0.46,.32}{u};
\end{tikzpicture}+
\begin{tikzpicture}[anchorbase,scale=1.4,centerzero=.6]
\draw[-,thin] (0.78,0.2) to (.315,.7) to (.315,.8);
\draw[-,thin] (0.285,.2) to (0.285,.3) to (-.18,.8);
\draw[-,line width=1.1pt] (0.3,.2)\botlabel{r} to (.3,.8);
\circled{0.49,.5}{u};
\circled{0.11,.5}{u};
\end{tikzpicture}\ .
$
\item
$\displaystyle
\begin{tikzpicture}[anchorbase,scale=1.4,centerzero=.6]
\draw[-,thin] (0.58,.8) to (.02,0.2);
\draw[-,line width=1.3pt] (0.3,.2)\botlabel{r} to (.3,.8);
\circled{0.14,.32}{u};
\end{tikzpicture}
=
\begin{tikzpicture}[anchorbase,scale=1.4,centerzero=.6]
\draw[-,thin] (0.58,.8) to (.02,0.2);
\draw[-,line width=1.3pt] (0.3,0.2)\botlabel{r} to (.3,.8);
\circled{0.46,.67}{u};
\end{tikzpicture}+
\begin{tikzpicture}[anchorbase,scale=1.4,centerzero=.6]
\draw[-,thin] (0.78,.8) to (.315,.3) to (.315,.2);
\draw[-,thin] (0.285,.8) to (0.285,.7) to (-.18,.2);
\draw[-,line width=1.1pt] (0.3,.2)\botlabel{r} to (.3,.8);
\circled{0.49,.5}{u};
\circled{0.11,.5}{u};
\end{tikzpicture}\ .
$
\end{enumerate}
\end{corollary}

\begin{proof}
\cref{bingley}(1) with $a=r$ and $b=1$ gives
$\displaystyle
\begin{tikzpicture}[anchorbase,scale=1.4,centerzero=.6]
\draw[-,thin] (0.58,0.2) to (.02,.8);
\draw[-,line width=1.3pt] (0.3,.2)\botlabel{r} to (.3,.8);
\circledinverse{0.3,.32}{u};
\end{tikzpicture}
=
\begin{tikzpicture}[anchorbase,scale=1.4,centerzero=.6]
\draw[-,thin] (0.58,0.2) to (.02,.8);
\draw[-,line width=1.3pt] (0.3,0.2)\botlabel{r} to (.3,.8);
\circledinverse{0.3,.68}{u};
\end{tikzpicture}
-
\begin{tikzpicture}[anchorbase,scale=1.4,centerzero=.6]
\draw[-,thin] (0.78,.2) to (.315,.7) to (.315,.8);
\draw[-,thin] (.285,.2) to (0.285,.3) to (-.18,.8);
\draw[-,line width=1.1pt] (0.3,.2)\botlabel{r} to (.3,.8);
\circledinverse{0.3,.5}{u};
\end{tikzpicture}\ .
$
The identity (1) follows from this by composing on the bottom with $\begin{tikzpicture}[scale=1.1,centerzero]
\draw[thin] (0.1,-.23) to (0.1,.23);
\draw[line width=1.2pt] (-.3,-.23)\botlabel{r} to (-0.3,.23);
\circled{-.3,0}{u};
\circled{0.1,0}{u};
\end{tikzpicture}$ and using the previous corollary.
Then (2) follows by applying $\div$.
\end{proof}

\begin{lemma}\label{inductionex}
$\displaystyle 
1+(1 + s_1 + s_2 s_1 + \cdots + s_{r-1} \cdots s_2 s_1)
\bully \frac{1}{u-x_1} = 
\frac{(u+1-x_1)\cdots (u+1-x_r)}{(u-x_1)\cdots (u-x_r)}.
$
\end{lemma}

\begin{proof}
This is an induction exercise!
It is easily checked in the case $r=1$,
using that 
\begin{equation}\label{itallstartshere}
s_1 \bully \frac{1}{u-x_1} = 
\frac{1}{u-x_2}+\frac{1}{(u-x_1)(u-x_2)}
=
\frac{u+1-x_1}{(u-x_1)(u-x_2)}
\end{equation}
by \cref{polyaction,demazureoperator}. 
For $r>1$, using induction for the equality $(*)$, we have that
\begin{align*}
1+(1 + s_1 + s_2 s_1 + &\cdots + s_{r-1} \cdots s_2 s_1)
\bully \frac{1}{u-x_1} \\
&=
1+\frac{1}{u-x_1} + (1+s_2 + \cdots + s_{r-1} \cdots s_2) \bully \frac{u+1-x_1}{(u-x_1)(u-x_2)}\\
&=
\frac{u+1-x_1}{u-x_1} + \frac{u+1-x_1}{u-x_1}\left((1+s_2 + \cdots + s_{r-1} \cdots s_2) \bully 
\frac{1}{u-x_2}\right)\\
&\stackrel{(*)}{=}
\frac{u+1-x_1}{u-x_1} + \frac{u+1-x_1}{u-x_1}\left(\frac{(u+1-x_2)\cdots (u+1-x_r)}{(u-x_2)\cdots(u-x_r)}-1\right)\\
&=
\frac{(u+1-x_1)(u+1-x_2)\cdots (u+1-x_r)}{(u-x_1)(u-x_2)\cdots(u-x_r)}.
\end{align*}
\end{proof}

We introduce symmetric polynomials $\tilde{p}_d(x_1,\dots,x_r)$
defined from the expansion of the expression appearing in \cref{inductionex}, setting
\begin{align}\label{personofinterest}
1 + \sum_{d \geq 0} \tilde{p}_d(x_1,\dots,x_r)
u^{-d-1}
&:=
\frac{(u+1-x_1)\cdots (u+1-x_r)}{(u-x_1)\cdots (u-x_r)}.
\end{align}
This definition 
makes sense even if $r=0$, in which case 
$\tilde p_d(x_1,\dots,x_r) = 0$ for all $d$.
The right hand side of \cref{personofinterest} can also be written as
\begin{align}
\left(1+\frac{1}{u-x_1}\right)
\left(1+\frac{1}{u-x_2}\right)\cdots
\left(1+\frac{1}{u-x_r}\right)
&=
\sum_{s \geq 0} \sum_{1 \leq p_1 < \cdots < p_s \leq r}
\frac{1}{(u-x_{p_1})\cdots (u-x_{p_s})}.\label{cheap}
\end{align}
From this, it is easy to see that
$\tilde{p}_0(x_1,\dots,x_r) = r$, and
$\tilde{p}_d(x_1,\dots,x_r)$ is equal to 
the usual power sum
$p_d(x_1,\dots,x_r) = x_1^d+\cdots+x_r^d$ 
plus an inhomogeneous symmetric polynomial of strictly smaller degree.
We call
$\tilde p_d(x_1,\dots,x_r)$ the {\em deformed power sum}.

\begin{lemma}\label{jonsid}
For $d \geq 1$, we have that
$\tilde p_d(x_1,\dots,x_r)
=
p_d(x_1,\dots,x_r)+
p_{<d}(x_1,\dots,x_r)$
for inhomogeneous symmetric polynomials
$p_{<d}(x_1,\dots,x_r)$ 
of degree $(d-1)$ defined recursively by
\begin{align}\label{jonsidentity}
p_{<d}(x_1,\dots,x_r)
=
{\textstyle\binom{r}{d+1}}+
\sum_{i=1}^{d-1}(-1)^{i-1}
\left[p_{<(d-i)}(x_1,\dots,x_r)
-{\textstyle\binom{r-i}{d+1-i}}
\right]
e_i(x_1,\dots,x_r).
\end{align}
\end{lemma}

\begin{proof}
Fix the number $r\geq 1$ of variables and write simply
$e_d$ for $e_d(x_1,\dots,x_r)$; in particular, $e_0 = 1$
and $e_d=0$ for $d > r$.
Let $e(u) := u^r - u^{r-1} e_1 + \cdots+(-1)^r e_r$,
then define $\bar p(u) = \bar p_0 + \bar p_1 u^{-1} + \bar p_2 u^{-2}\cdots \in \P^{(r)}[h]\llbracket u^{-1}\rrbracket$ 
by
\begin{equation}\label{water}
\bar p(u) e(u)  = u \frac{e(u+h)-e(u)}{h}.
\end{equation}
Here, $h$ is a gratuitous 
new variable (we will only be interested in the cases $h=0$ and $h=1$).

Equating coefficients of $u^{r-d}$ on both sides of \cref{water} gives the identity
\begin{equation}\label{earth}
\sum_{i=0}^d (-1)^i e_i \bar p_{d-i}
=
\sum_{i=0}^{d} (-1)^i
{\textstyle\binom{r-i}{d+1-i}} h^{d-i} e_i
\end{equation}
for $d \geq 0$.
The $d=0$ case implies that $\bar p_0 = r$.
Using this, the identity can be rearranged to obtain
\begin{equation}\label{rareearth}
\bar p_d =
(-1)^{d-1} d e_d 
+\sum_{i=1}^{d-1} (-1)^{i-1} e_i \bar p_{d-i}+
\sum_{i=0}^{d-1}
(-1)^{i} {\textstyle\binom{r-i}{d+1-i}} h^{d-i} e_i
\end{equation}
for $d > 0$.

Now let $p(u) = p_0 + p_1 u^{-1} + \cdots$
be obtained from $\bar p(u)$ by setting $h=0$;
by \cref{water}, we have that 
$p(u) e(u) = u e'(u)$. 
We still have that $p_0 = r$.
When we set $h=0$ in \cref{rareearth}, 
we obtain
\begin{equation}\label{newtonsidentity}
p_d =
(-1)^{d-1} d e_d 
+\sum_{i=1}^{d-1} (-1)^{i-1} e_i p_{d-i}
\end{equation}
for $d \geq 1$.
This is exactly Newton's identity relating power sums to elementary symmetric polynomials, so we have that
$p_d = p_d(x_1,\dots,x_r)$.

Finally, let $\tilde p(u) = \tilde p_0 + \tilde p_1 u^{-1}+\cdots$ 
be obtained from $\bar p(u)$ by setting $h=1$. The identity
\cref{water} implies that
$1+ u^{-1} \tilde p(u) = \frac{e(u+1)}{e(u)}$.
Comparing with \cref{personofinterest}, it follows that 
$\tilde p_d = \tilde p_d(x_1,\dots,x_r)$.
So $p_{<d} := \tilde p_d - p_d$
is the symmetric polynomial $p_{< d}(x_1,\dots,x_r)$ in the  statement of the lemma.
Subtracting
\cref{newtonsidentity} from
the identity \cref{rareearth} at $h=1$ gives
$$
p_{<d}=
\sum_{i=1}^{d-1} (-1)^{i-1} e_i p_{<(d-i)}
+
\sum_{i=0}^{d-1}
(-1)^{i} {\textstyle\binom{r-i}{d+1-i}} e_i.
$$
This is easily rearranged to produce \cref{jonsidentity}.
\end{proof}

The generating function \cref{personofinterest}
pinned to a string of thickness $r$
can be represented diagrammatically by
$\begin{tikzpicture}[centerzero]
\draw[ultra thick] (0,-.4)\botlabel{r} to (0,.4);
\ovaledinverse{0,.15}{u+1};
\circled{0,-.15}{u};
\end{tikzpicture}$. 
In view of \cref{deformedshuffle}, 
\cref{inductionex} implies the relation
\begin{equation}\label{skiving}
\begin{tikzpicture}[anchorbase,scale=.8]
\draw[-,line width=2pt] (0.08,-.7) \botlabel{r}to (0.08,.5);
\end{tikzpicture}
\ +\ \begin{tikzpicture}[anchorbase,scale=.8]
\draw[-,line width=2pt] (-0.08,-.7)\botlabel{r} to (-0.08,-.49);
\draw[-,line width=2pt] (-0.08,.29) to (-0.08,.5);
\draw[-] (-0.09,-.51) to [out=135,in=-135] (-0.09,.31);
\draw[-,line width=1.5pt] (-0.095,-.52) to [out=45,in=-45] (-0.095,.32);
\circled{-.26,-.1}{u};
\end{tikzpicture}
= 
\begin{tikzpicture}[anchorbase,scale=.8]
\draw[-,line width=2pt] (0.08,-.7) \botlabel{r}to (0.08,.5);
\ovaledinverse{.08,.1}{u+1};
\circled{.08,-.3}{u};
\end{tikzpicture}
\end{equation}
for any $r \geq 1$.
Applying $\vid$ gives also the relation
\begin{equation}\label{skiving2}
\begin{tikzpicture}[anchorbase,scale=.8]
\draw[-,line width=2pt] (0.08,-.7) \botlabel{r}to (0.08,.5);
\end{tikzpicture}
\ -\ \begin{tikzpicture}[anchorbase,scale=.8]
\draw[-,line width=2pt] (-0.08,-.7)\botlabel{r} to (-0.08,-.49);
\draw[-,line width=2pt] (-0.08,.29) to (-0.08,.5);
\draw[-] (-0.07,-.51) to [out=45,in=-45] (-0.055,.31);
\draw[-,line width=1.5pt] (-0.055,-.52) to [out=135,in=-135] (-0.07,.32);
\circled{.1,-.1}{u};
\end{tikzpicture}
= 
\begin{tikzpicture}[anchorbase,scale=.8]
\draw[-,line width=2pt] (0.08,-.7) \botlabel{r}to (0.08,.5);
\ovaledinverse{.08,.1}{u-1};
\circled{.08,-.3}{u};
\end{tikzpicture}
\end{equation}
for $r \geq 1$.
The following theorem gives some generalizations.

\begin{theorem}\label{crazy}
The following hold in $\aSchur\lround u^{-1}\rround$ for $a,b \geq 0$:
\begin{enumerate}
\item
$\displaystyle
\begin{tikzpicture}[anchorbase,scale=.8]
\draw[-,line width=2pt] (-0.08,-.8)\botlabel{\phantom{a+b}} to (-0.08,-.49);
\draw[-,line width=2pt] (-0.08,.29) to (-0.08,.6);
\draw[-,line width=1.2pt] (-0.08,-.51) to [out=135,in=-135] (-0.08,.31);
\draw[-,line width=1.2pt] (-0.08,-.52) to [out=45,in=-45] (-0.08,.32);
\circled{-.26,.05}{u};
\strand{-.4,-.3}{a};
\strand{.2,-.32}{b};
\end{tikzpicture}
= 
\frac{1}{a!}
\sum_{i=0}^a (-1)^{a-i} \binom{a}{i}
\begin{tikzpicture}[anchorbase,scale=.8]
\draw[-,line width=2pt] (0.08,-.8) \botlabel{a+b}to (0.08,.6);
\ovaledinverse{.08,.1}{u+i};
\circled{.08,-.3}{u};
\end{tikzpicture}$.
\item
$\displaystyle
\begin{tikzpicture}[anchorbase,scale=.8]
\draw[-,line width=2pt] (-0.08,-.8)\botlabel{\phantom{a+b}} to (-0.08,-.49);
\draw[-,line width=2pt] (-0.08,.29) to (-0.08,.6);
\draw[-,line width=1.2pt] (-0.08,-.51) to [out=135,in=-135] (-0.08,.31);
\draw[-,line width=1.2pt] (-0.08,-.52) to [out=45,in=-45] (-0.08,.32);
\circled{.08,.1}{u};
\strand{-.4,-.3}{a};
\strand{.2,-.32}{b};
\end{tikzpicture}
= 
\frac{1}{b!}
\sum_{i=0}^b (-1)^{i} \binom{b}{i}
\begin{tikzpicture}[anchorbase,scale=.8]
\draw[-,line width=2pt] (0.08,-.8) \botlabel{a+b}to (0.08,.6);
\ovaledinverse{.08,.1}{u-i};
\circled{.08,-.3}{u};
\end{tikzpicture}$.
\end{enumerate}
(The right hand sides here involve division by some factorials. This should be interpreted by working first over $\Z$, when the right hand side can be rewritten as a linear combination involving only integer coefficients. Then one can base change to obtain a valid formula for any $\kk$.)
\end{theorem}

\begin{proof}
We prove (1). Then (2) follows by applying $\vid$.
We proceed by induction on $a$. The case $a=0$ is trivial, while
the $a=1$ case follows from \cref{skiving}.
For the induction step, for $a > 1$, we have that
\begin{align*}
\begin{tikzpicture}[anchorbase,scale=1.2]
\draw[-,line width=2pt] (-0.08,-.8)\botlabel{\phantom{a+b}} to (-0.08,-.49);
\draw[-,line width=2pt] (-0.08,.29) to (-0.08,.6);
\draw[-,line width=1.2pt] (-0.08,-.51) to [out=135,in=-135] (-0.08,.31);
\draw[-,line width=1.2pt] (-0.08,-.52) to [out=45,in=-45] (-0.08,.32);
\circled{-.26,0}{u};
\strand{-.35,-.34}{a};
\strand{.18,-.32}{b};
\end{tikzpicture}\!
&=
\frac{1}{a}
\begin{tikzpicture}[anchorbase,scale=1.3]
\draw[-,line width=2pt] (-0.08,-.8)\botlabel{\phantom{a+b}} to (-0.08,-.49);
\draw[-,line width=2pt] (-0.08,.29) to (-0.08,.6);
\draw[-,line width=1.2pt] (-0.08,-.51) to [out=135,in=-135] (-0.08,.31);
\draw[-,line width=1.2pt] (-0.08,-.52) to [out=45,in=-45] (-0.08,.32);
\draw (-.055,-.51) to [out=135,in=-60] (-.145,-.4) to[out=45,in=-45] (-.145,.2) to [out=60,in=-135] (-.055,.31);
\circled{-.26,0}{u};
\circled{-.05,-.18}{u};
\strand{-.48,-.24}{a-1};
\strand{.16,-.32}{b};
\end{tikzpicture}
=
\frac{1}{a}
\begin{tikzpicture}[anchorbase,scale=1.3]
\draw (-.1,-.51) to [out=45,in=-120] (-0.01,-.4) to[out=135,in=-135] (-0.01,.2) to [out=120,in=-45] (-.1,.31);
\draw[-,line width=2pt] (-0.08,-.8)\botlabel{\phantom{a+b}} to (-0.08,-.49);
\draw[-,line width=2pt] (-0.08,.29) to (-0.08,.6);
\draw[-,line width=1.2pt] (-0.08,-.51) to [out=135,in=-135] (-0.08,.31);
\draw[-,line width=1.2pt] (-0.08,-.52) to [out=45,in=-45] (-0.08,.32);
\circled{-.26,0}{u};
\circled{-.1,-.2}{u};
\strand{-.45,-.32}{a-1};
\strand{.32,-.15}{b+1};
\end{tikzpicture}\!\stackrel{\cref{skiving}}{=}
\frac{1}{a}
\begin{tikzpicture}[anchorbase,scale=1.3]
\draw[-,line width=2pt] (-0.08,-.8)\botlabel{\phantom{a+b}} to (-0.08,-.49);
\draw[-,line width=2pt] (-0.08,.29) to (-0.08,.6);
\draw[-,line width=1.2pt] (-0.08,-.51) to [out=135,in=-135] (-0.08,.31);
\draw[-,line width=1.2pt] (-0.08,-.52) to [out=45,in=-45] (-0.08,.32);
\circled{-.26,-0.05}{u};
\circled{.1,-.15}{u};
\ovaledinverse{.08,0.1}{u+1};
\strand{-.45,-.4}{a-1};
\strand{.3,-.4}{b+1};
\end{tikzpicture}\!\!\!
-
\frac{1}{a}
\begin{tikzpicture}[anchorbase,scale=1.3]
\draw[-,line width=2pt] (-0.08,-.8)\botlabel{\phantom{a+b}} to (-0.08,-.49);
\draw[-,line width=2pt] (-0.08,.29) to (-0.08,.6);
\draw[-,line width=1.2pt] (-0.08,-.51) to [out=135,in=-135] (-0.08,.31);
\draw[-,line width=1.2pt] (-0.08,-.52) to [out=45,in=-45] (-0.08,.32);
\circled{-.26,-0.05}{u};
\strand{-.45,-.38}{a-1};
\strand{.3,-.38}{b+1};
\end{tikzpicture}=
\frac{1}{a}
\begin{tikzpicture}[anchorbase,scale=1.3]
\draw[-,line width=2pt] (-0.08,-.8)\botlabel{\phantom{a+b}} to (-0.08,-.49);
\draw[-,line width=2pt] (-0.08,.29) to (-0.08,.6);
\draw[-,line width=1.2pt] (-0.08,-.51) to [out=135,in=-135] (-0.08,.31);
\draw[-,line width=1.2pt] (-0.08,-.52) to [out=45,in=-45] (-0.08,.32);
\ovaled{-.26,-0.05}{u+1};
\circled{-.08,-.65}{u};
\ovaledinverse{-.08,0.45}{u+1};
\strand{-.45,-.4}{a-1};
\strand{.3,-.4}{b+1};
\end{tikzpicture}\!\!\!
-
\frac{1}{a}
\begin{tikzpicture}[anchorbase,scale=1.3]
\draw[-,line width=2pt] (-0.08,-.8)\botlabel{\phantom{a+b}} to (-0.08,-.49);
\draw[-,line width=2pt] (-0.08,.29) to (-0.08,.6);
\draw[-,line width=1.2pt] (-0.08,-.51) to [out=135,in=-135] (-0.08,.31);
\draw[-,line width=1.2pt] (-0.08,-.52) to [out=45,in=-45] (-0.08,.32);
\circled{-.26,-0.05}{u};
\strand{-.45,-.38}{a-1};
\strand{.3,-.38}{b+1};
\end{tikzpicture}\!.
\end{align*}
The two terms at the end here can now be rewritten using the induction hypothesis (with $u$ replaced by $u+1$ for the first one). The result can then be simplified using Pascal's identity.
\end{proof}

%% file: s6-center.tex
\setcounter{section}{5}

\section{The center of the degenerate affine Schur algebra}\label{s6-center}

Let $Z(\aS(n,r))$ be the center of the degenerate affine Schur algebra.
In this section, we prove that 
$Z(\aS(n,r))\cong \P^{(r)}$, the algebra of symmetric polynomials
in $x_1,\dots,x_r$, for all $n \geq 1$.
We will also determine the center $Z(\aSchur)$ of the category $\aSchur$, that is, the algebra of endomorphisms of the identity functor $\id_{\aSchur}$, and the centers of each of the endomorphism algebras $\End_{\aSchur}(\lambda)$.

For $\lambda \in \Lambda(n,r)$ and $f \in \P^\lambda$, we start now to use the shorthand
\begin{equation}\label{global}
f 1_\lambda = 1_\lambda f := \xi_{\diag(\lambda_1,\dots,\lambda_n),f} = \xi_{f,\diag(\lambda_1,\dots,\lambda_n)} \in \End_{\aSchur}(\lambda)
\end{equation}
The diagram for this is just $f$ pinned to the diagram for $1_\lambda$, that is, $n$ parallel vertical strings of thicknesses $\lambda_1,\dots,\lambda_n$.
We use similar shorthand for $\varsigma_{\diag(\lambda_1,\dots,\lambda_n),f}$ in $\End_{\gSchur}(\lambda)$.
The identity element of $\aS(n,r)$ is $1_{n,r} := \sum_{\lambda 
\in\Lambda(n,r)} 1_\lambda$.
For $f \in \P^{(r)}$, we let 
\begin{equation}\label{glueball}
f 1_{n,r} := 
\sum_{\lambda \in \Lambda(n,r)} f 1_\lambda.
\end{equation}

\begin{lemma}\label{interesting}
For any $r\geq 0$ and $f \in \P^{(r)}$, there is a natural transformation
$(\delta_{r,|\lambda|}f 1_\lambda)_{\lambda \in \Lambda}$
in $Z(\aSchur)$.
\end{lemma}

\begin{proof}
Since $\P^{(r)}$ is generated by elementary symmetric polynomials,
it suffices to show that
$\big(\delta_{r,|\lambda|} e_d(x_1,\dots,x_r) 1_\lambda\big)_{\lambda \in \Lambda}$
is a natural transformation in $Z(\aSchur)$ for each $d \geq 1$.
This follows from \cref{game,rome}.
\end{proof}

\begin{corollary}\label{interestinger}
For $f \in \P^{(r)}$ and $\lambda \in \Lambda(n,r)$,
$f 1_\lambda$ is central in $\End_{\aSchur}(\lambda)$,
and $f 1_{n,r}$ is central in $\aS(n,r)$.
\end{corollary}

The following lemma proves \cite[Conj.~3.14]{SongWang}.

\begin{lemma}\label{weiqiang}
For any $n,r\geq 0$ and $\lambda \in \Lambda(n,r)$,
the map 
\begin{align}
\P^{(r)} &\rightarrow Z(\End_{\aSchur}(\lambda)), &f &\mapsto f 1_\lambda
\end{align}
is an algebra isomorphism.
\end{lemma}

\begin{proof}
\cref{interestinger} implies that $f 1_\lambda\:(f \in \P^{(r)})$ 
belongs to $Z(\End_{\aSchur}(\lambda))$, so the map makes sense.
It is also clear that it is an algebra homomorphism, and its injectivity follows from \cref{bits}.
To show that it is surjective,
we pass to the associated graded algebra
$\gr \End_{\aSchur}(\lambda)$, which 
is identified with $\End_{\gSchur}(\lambda)$. 
Since $\gr Z(\End_{\aSchur}(\lambda)) \subseteq Z(\End_{\gSchur}(\lambda)))$, it suffices to show that
\begin{equation*}
Z(\End_{\gSchur}(\lambda)) \subseteq
\{f 1_\lambda\:|\:f \in \P^{(r)}\}.
\end{equation*}
Take a central element of $\End_{\gSchur}(\lambda)$. It can be expressed in terms of the basis from \cref{bots} as
$$
z= \sum_{A \in \Mat{\lambda}{\lambda}}
\varsigma_{A,f_A}
$$
Consider $A \in \Mat{\lambda}{\lambda}$ which is {\em not} a diagonal matrix. Let $i$ be minimal such that $a_{i,j} \neq 0$ for some $j \neq i$.
We then have that $a_{i,k} \neq 0$ and $a_{k',i} \neq 0$ for 
some $i < k,k' \leq n$.
Let \begin{align*}
I &:= \left\{\lambda_{<i}+1,
\dots,\lambda_{\leq i}\right\},\\
J &:= \coprod_{j=1}^n
\left\{\lambda_{<j}+a_{1,j}+\cdots+a_{i-1,j}+1,\dots,\lambda_{<j}+a_{1,j}+\cdots+a_{i-1,j}+a_{i,j}\right\}.
\end{align*}
Let $x_I := \prod_{i \in I} x_i \in \P^\lambda$ and $x_J := \prod_{j \in J} x_j \in \P^{{\mu(A)}}$.
We have that 
$(x_I 1_\lambda) \varsigma_{A,f_A}
= \varsigma_{A, f_A x_J}$ and $\varsigma_{A,f_A} (x_I 1_\lambda)
= \varsigma_{A,f_A x_I}$.
Using $(x_I 1_\lambda) z = z (x_I 1_\lambda)$,
we deduce that $f_A x_I = f_A x_J$.
The choice of $i$ ensures that $x_I \neq x_J$, so
it follows that $f_A = 0$.

We have now proved that $z = f 1_\lambda$
for $f \in \P^\lambda$.
It remains to show that in fact $f$ belongs to $\P^{(r)} \subseteq \P^\lambda$.
This follows if we can show that $s_k(f) = f$ for each $1 \leq k < n$ of the form
$k=\lambda_{\leq i}$.
Given such a $k$, we can choose $1 \leq i < j \leq n$ so that
$k = \lambda_{\leq i}$,
$\lambda_i \neq 0$, $\lambda_{i+1}=\cdots=\lambda_{j-1}=0$
and $\lambda_j \neq 0$.
Let $A \in \Mat{\lambda}{\lambda}$ be the matrix $\diag(\lambda_1,\dots,\lambda_n) + e_{i,j}+e_{j,i}-e_{i,i}-e_{j,j}$.
The corresponding double coset diagram has a thin crossing 
between its $i$th and $j$th vertical strings.
We have that $(f 1_\lambda) \xi_A = \xi_{A, s_k(f)}$ and $\xi_A (f 1_\lambda) = \xi_{A,f}$.
The centrality of $f 1_\lambda$ implies that these are equal,
hence, $s_k(f)= f$.
\end{proof}

\begin{theorem}\label{whatIexpect}
For $n \geq 1$, 
the map $\P^{(r)} \rightarrow Z(\aS(n,r)), f \mapsto f 1_{n,r}$
is an algebra isomorphism.
\end{theorem}

\begin{proof}
\cref{interestinger} implies that $f 1_{n,r}$ is central for each $f \in \P^{(r)}$, so the map is well defined. It is clearly an algebra homomorphism, and it is injective by \cref{bits}.
To show that it is surjective, take a central element
$z \in Z(\aS(n,r))$. 
Using that $1_\lambda z = z 1_\lambda$, it follows easily that
$z = \sum_{\lambda \in \Lambda(n,r)} z_\lambda$ with
$z_\lambda \in 1_\lambda \aS(n,r) 1_\lambda$.
The centrality of $z$ in $\aS(n,r)$
implies that each $z_\lambda$ is central in 
$\End_{\aSchur}(\lambda)$.
Hence, by \cref{weiqiang}, we have that $z_\lambda = f_\lambda 1_\lambda$
for $f_\lambda \in \P^{(r)}$.
Finally, we let $\mu := (r,0,\dots,0) \in \Lambda(n,r)$ 
and $f := f_\mu \in \P^{(r)}$.
Then take any $\lambda\in \Lambda(n,r)$
and let 
$A$ be the unique element of $\Mat{\lambda}{\mu}$; the basis vector
$\xi_A$ splits a single string of thickness $r$ into strings of thicknesses $\lambda_1,\dots,\lambda_n$.
Using $\xi_A z = z \xi_A$, we get that $\xi_{A,f} = \xi_{A, f_\lambda}$. Hence, $f_\lambda = f$ for all $\lambda \in \Lambda(n,r)$.
This shows that $z = f 1_{n,r}$, completing the proof.
\end{proof}

\begin{corollary}\label{overcenter}
$\aS(n,r)$ is free of finite rank as a module over its center.
\end{corollary}

\begin{proof}
For $\lambda \in \Lambda(n,r)$, $P^\lambda$ is a free $P^{(r)}$-module of rank $r! / \lambda_1!\cdots \lambda_n!$, so
this follows from \cref{whatIexpect,bits}.
\end{proof}

\begin{corollary}\label{impossible}
The center $Z(\aSchur)$ of the degenerate affine Schur category is 
isomorphic to $\prod_{r \geq 0} \P^{(r)}$
via the map sending $(f_r)_{r \geq 0}\in \prod_{r \geq 0} \P^{(r)}$ to 
the natural transformation $(f_{|\lambda|} 1_\lambda)_{\lambda \in \Lambda}$.
\end{corollary}

\begin{proof}
This follows from \cref{interesting} and \cref{whatIexpect}.
\end{proof}

\begin{remark}
When $n \geq r$, \cref{whatIexpect} also follows from the double centralizer property \cref{dcp} and the already known description of $Z(\aH_r)$.
\end{remark}

%% file: s7-drinfeld.tex
\setcounter{section}{6}

\section{Drinfeld's homomorphism from Yangians to degenerate affine Schur algebras}\label{s7-drinfeld}

The exposition in this section is based on the beautiful paper \cite{Arakawa}. 
An important point is that Arakawa works with a different definition of the degenerate affine Hecke algebra $\aH_r$ to us.
An isomorphism from our
version to his is given by mapping $x_i$ to $\epsilon_i$ and $w \in \S_r$ to $(-1)^{\ell(w)} w$.
We have systematically translated the results in \cite{Arakawa} taking this additional sign twist into account, but also repeat the proof of \cref{reality} in order to be self-contained.

As at the end of \cref{s2-reminders}, let $V$ be the natural representation of $\g := \gl_n(\kk)$ 
with standard basis $v_1,\dots,v_n$.
Let
\begin{align}
P &:= \sum_{i,j=1}^n e_{i,j} \otimes e_{j,i} \in \End_\kk(V)^{\otimes 2},
&
Q &:= \sum_{i,j=1}^n e_{i,j} \otimes e_{i,j} \in \End_\kk(V)^{\otimes 2}.
\end{align}
Note that $P$ acts on $V \otimes V$ as the tensor flip: $P(v_i \otimes v_j) = v_j \otimes v_i$.
In $\End_\kk(V)^{\otimes 3}$,
we have that
\begin{align}
P^{[2,3]} Q^{[1,3]} &= Q^{[1,2]} P^{[2,3]} = Q^{[1,2]}Q^{[1,3]},&
P^{[2,3]} Q^{[1,2]} &= Q^{[1,3]} P^{[2,3]} = Q^{[1,3]}Q^{[1,2]}.
\end{align}
Here, $P^{[2,3]}$ denotes $1 \otimes P$, that is, $P$ 
in the tensor positions 2 and 3 with the identity in the first position,
$Q^{[1,3]}$ denotes $Q$ in tensor positions 1 and 3 with the identity
in the second position, etc..
Using these identities and \cref{daharels}, 
one can check the following:

\begin{lemma}\label{learnt}
In the algebra $\End_\kk(V)^{\otimes 3} \otimes \aH_2[u]$,
we have that
\begin{multline*}
\big(u-x_1+Q^{[1,2]}\big)\big(u-x_2+Q^{[1,3]}\big)\big(P^{[2,3]}-s_1\big)
=\\ \big(P^{[2,3]}-s_1\big)\left(
\big(u-x_1+Q^{[1,3]}\big)\big(u-x_2+Q^{[1,2]}\big) + s_1 \big(Q^{[1,2]}-Q^{[1,3]}\big)
\right).
\end{multline*}
(In this equation, $x_1, x_2$ and $s_1$ denote these 
elements of $\aH_2$
identified with the subalgebra $1 \otimes 1 \otimes 1 \otimes \aH_2$
of $\End_\kk(V)^{\otimes 3} \otimes \aH_2[u]$
in the obvious way.)
\end{lemma}

Let $\Yn$ be the Yangian associated to $\g$.
The quickest way to define this algebra is via the RTT presentation: it has generators $T_{i,j}^{(d)}\:(1 \leq i,j \leq n, d \geq 1)$ subject to the relations 
\begin{equation}\label{mr}
[T_{i,j}^{(a)}, T_{k,l}^{(b)}]
= \sum_{c=0}^{\min(a,b)-1}
\left( 
T_{i,l}^{(a+b-1-c)}T_{k,j}^{(b)} -
T_{i,l}^{(c)}T_{k,j}^{(a+b-1-c)}
\right)
\end{equation}
for every $1 \leq i,j,k,l \leq n$ and $a,b \geq 1$,
where $T_{i,j}^{(0)} := \delta_{i,j}$.
These relations can be written equivalently in terms of generating functions as
\begin{equation}\label{RTT}
R^{[1,2]}(u-v) T^{[1,3]}(u) T^{[2,3]}(v)
= 
T^{[2,3]}(v) T^{[1,3]}(u) R^{[1,2]}(u-v),
\end{equation}
equality in $\End_\kk(V)^{\otimes 2} \otimes \Yn\lround u^{-1}, v^{-1}\rround$.
This needs a little more explanation (see also \cite{MNO} for a fuller account):
the superscript notation indicates tensor positions like in the opening paragraph;
the variables $u$ and $v$ are indeterminates; and
$$
R(u):=
u + P
\in \End_\kk(V)^{\otimes 2}[u],
$$
\begin{align*}
T_{i,j}(u) &:= \sum_{t \geq 0} T_{i,j}^{(t)} u^{-t} \in \Yn\llbracket u^{-1}\rrbracket,&
T(u) := \sum_{i,j=1}^n e_{i,j} \otimes T_{i,j}(u)
\in \End_\kk(V) \otimes \Yn \llbracket u^{-1} \rrbracket.
\end{align*}

\begin{lemma}\label{easygen}
The Yangian $\Yn$ is generated as an algebra
by the elements $T_{1,1}^{(d)}\:(d \geq 1)$ and $T_{i,j}^{(1)}\:(1 \leq i,j \leq n)$.
\end{lemma}

\begin{proof}
By \cref{mr}, $[T_{1,1}^{(d)}, T_{1,j}^{(1)}]$
and $[T_{i,1}^{(1)}, T_{1,1}^{(d)}]$ yield
$T_{1,j}^{(d)}\:(j > 1)$ and $T_{i,1}^{(d)}\:(i> 1)$, respectively.
Then $[T_{i,1}^{(d)}, T_{1,j}^{(1)}]+ \delta_{i,j} T_{1,1}^{(d)}$
yields $T_{i,j}^{(d)}\:(i,j > 1)$.
\end{proof}

There are some useful symmetries 
(see \cite[Prop.~1.12]{MNO}):
\begin{itemize}
\item ({\em Translation}) For $c \in \kk$, let $\eta_c:\Yn 
\rightarrow \Yn$
be the automorphism defined by $\eta_c(T_{i,j}(u)) = T_{i,j}(u+c)$, i.e.,
$\eta_c(T_{i,j}^{(d)}) = \sum_{s=0}^{d-1}
\binom{t-1}{s} (-c)^s T_{i,j}^{(d-s)}$.
\item ({\em Multiplication by a
power series}) For $f(u) \in 1 + u^{-1} \kk[[u^{-1}]]$, let 
$\mu_f:\Yn \rightarrow \Yn$ be the automorphism defined by
$\mu_f(T_{i,j}(u)) = f(u) T_{i,j}(u)$, i.e.,
$\mu_f(T_{i,j}^{(d)}) = \sum_{r=0}^d a_r T_{i,j}^{(d-r)}$
if $f(u) = \sum_{r \geq 0} a_r u^{-r}$.
\item ({\em Transposition})
Let $\tau:\Yn\rightarrow \Yn$
be the antiautomorphism of order 2
defined from
$\tau(T_{i,j}(u)) = T_{j,i}(u)$, i.e.,
$\tau(T_{i,j}^{(t)}) = T_{j,i}^{(t)}$.
\item ({\em Inversion})
Let $\omega_n:\Yn \rightarrow \Yn$ be the
automorphism of order $2$ defined from the equation
$(\id_{\End_\kk(V)} \otimes \omega_n)(T(u)) = T(-u)^{-1}$.
\end{itemize}
The Yangian is a Hopf algebra with comultiplication
$\Delta$ defined by
\begin{equation}\label{comultiplication}
(\id_{\End_\kk(V)} \otimes \Delta)(T(u)) := T^{[1,2]}(u) T^{[1,3]}(u) \in \End_\kk(V) \otimes \Yn^{\otimes 2}\llbracket u^{-1}\rrbracket.
\end{equation}
Also, for any scalar $c$, 
there is the {\em evaluation homomorphism} $\ev_c:\Yn \rightarrow \End_\kk(V)$ defined by
\begin{equation}\label{evaluation}
(\id_{\End_\kk(V)} \otimes \ev_c)(T(u)) :=
1 + \frac{Q}{u-c} \in \End_\kk(V)^{\otimes 2}\llbracket u^{-1}\rrbracket.
\end{equation}
Evaluation homomorphisms can be defined more generally: for any algebra $A$ and any $c \in A$ there is an algebra homomorphism
$\ev_c:\Yn \rightarrow \End_\kk(V) \otimes A$ defined by the same formula \cref{evaluation}. 
Identifying $\End_\kk(V) \otimes A$ with
$\End_{\dash A}(V \otimes A)$, this makes $V \otimes A$ into a $(\Yn,A)$-bimodule. 

Now let $\Delta^{(r)}:\Yn \rightarrow \Yn^{\otimes r}$ be the $r$th iterated comultiplication. Similar to the previous paragraph, 
using also that $x_1,\dots,x_r \in \aH_r$ commute,
there is an algebra homomorphism
\begin{equation}
\DRINFELD_{n,r}
:= (\ev_{x_1} \barotimes \cdots \barotimes \ev_{x_r}) \circ \Delta^{(r)}:\Yn \rightarrow \End_\kk(V)^{\otimes r} \otimes \aH_r
\end{equation}
defined in terms of generating functions by
\begin{align}\label{formalism}
(\id_{\End_\kk(V)}\otimes \DRINFELD_{n,r})(T(u))
&:=
\left(1+\frac{Q^{[1,2]}}{u-x_1}\right)\cdots \left(1+\frac{Q^{[1,r+1]}}{u-x_r}\right),
\end{align}
equality in $\End_\kk(V)^{\otimes (r+1)} \otimes \aH_r\llbracket u^{-1}\rrbracket$.
The definition implies that
\begin{align}
\DRINFELD_{n,r}(T_{i,j}(u)) &= \sum_{\substack{\bi \in I(n,r)\\i_r = j}}
\left(\delta_{i,i_1} + \frac{e_{i,i_1}^{[1]}}{u-x_1}
\right) 
\left(\delta_{i_1,i_2} + \frac{e_{i_1,i_2}^{[2]}}{u-x_2}
\right) 
\cdots
\left(\delta_{i_{r-1},i_r} + \frac{e_{i_{r-1},i_r}^{[r]}}{u-x_r}
\right)\label{hard}\\
&=
\delta_{i,j}+\!\!\!\!
\sum_{\substack{s \geq 1\\1 \leq p_1 < \cdots < p_s \leq r}}
\sum_{\substack{\bi \in I(n,s)\\ i_s = j}}
\frac{e_{i,i_1}^{[p_1]} e_{i_1,i_2}^{[p_2]}
\cdots e_{i_{s-1},i_s}^{[p_s]}}
{(u-x_{p_1})\cdots (u-x_{p_s})}\in \End_\kk(V)^{\otimes r}\otimes \aH_r\llbracket u^{-1}\rrbracket.\label{harder}
\end{align}
Identifying $\End_\kk(V)^{\otimes r} \otimes \aH_r$
with $\End_{\dash \aH_r}(V^{\otimes r} \otimes \aH_r)$,
the homomorphism $\DRINFELD_{n,r}$
makes $V^{\otimes r} \otimes \aH_r$ into a $(\Yn,\aH_r)$-bimodule.

\begin{lemma}\label{reality}
The action of $\Yn$ on $V^{\otimes r} \otimes \aH_r$
induces an action on the quotient
$V^{\otimes r} \otimes_{\kk S_r} \aH_r$.
\end{lemma}

\begin{proof}
In this proof, which follows \cite[Prop.~2]{Arakawa},
we use $s_i$ and $x_j$ to denote the endomorphisms of $V^{\otimes r} \otimes \aH_r$
defined by left multiplication by these elements on the last tensor factor $\aH_r$.
The endomorphism of $V^{\otimes r} \otimes \aH_r$ defined by the right action of 
$s_i$ on $V^{\otimes r}$ is the operator $P^{[i,i+1]}$.
We need to show that the subspace
$\sum_{i=1}^{r-1} \Im (P^{[i,i+1]}-s_i)$ is a $\Yn$-submodule
of $V^{\otimes r} \otimes \aH_r$.
Using \cref{formalism}, this follows if we show
for $i=1,\dots,r-1$ that
$$
\frac{
\big(u-x_1+Q^{[1,2]}\big)\cdots \big(u-x_r+Q^{[1,r+1]}\big)}{(u-x_1) \cdots (u-x_r)}:V^{\otimes (r+1)} \otimes \aH_r\llbracket u^{-1}
\rrbracket \rightarrow V^{\otimes (r+1)} \otimes \aH_r\llbracket u^{-1}
\rrbracket
$$
maps $\Im (P^{[i+1,i+2]}-s_i)$ into $\Im (P^{[i+1,i+2]}-s_i)$.
Since the coefficients of $(u-x_1)\cdots(u-x_r)$ are in the center of $\aH_r$, we can ignore the denominator. Then the conclusion follows since 
$$
\big(u-x_1+Q^{[1,2]}\big)\cdots \big(u-x_r+Q^{[1,r+1]}\big)
\big(P^{[i+1,i+2]}-s_i\big)
= \big(P^{[i+1,i+2]}-s_i\big) X
$$
for some $X \in \End_\kk(V)^{\otimes (r+1)} \otimes \aH_r[u]$
by \cref{learnt}.
\end{proof}

To avoid potential confusion, 
we will use the notation $v \barotimes h$
to denote the image of $v \otimes h \in V^{\otimes r} \otimes \aH_r$
under the quotient map
$\pi:V^{\otimes r} \otimes \aH_r \twoheadrightarrow
V^{\otimes r}\otimes_{\kk S_r} \aH_r$.
We have now made both of these spaces into
$(\Yn, \aH_r)$-bimodules in such a way that $\pi$ is a bimodule homomorphism.
Now recall from \cref{nitpick1} that
$\aS(n,r) = \End_{\dash \aH_r}\big(V^{\otimes r} \otimes_{\kk \S_r} \aH_r\big)$.
So the action of $\Yn$ on $V^{\otimes r}\otimes_{\kk \S_r} \aH_r$
induces an algebra homomorphism
\begin{equation}
\Drinfeld_{n,r}:\Yn \rightarrow \aS(n,r),
\end{equation}
which we call the {\em Drinfeld homomorphism}.
Recall also the homomorphism $\drinfeld_{n,r}:\Un\rightarrow \S(n,r)$ from \cref{notdrinfeld}.

\begin{lemma}\label{silly}
The following diagram commutes:
$$
\begin{tikzcd}
\arrow[d,hookrightarrow]
\Un\arrow[r,"\drinfeld_{n,r}" above]&\S(n,r)\arrow[d,hookrightarrow]\\
\Yn\arrow[r,"\Drinfeld_{n,r}" below]&\aS(n,r).
\end{tikzcd}
$$
Here,
the left hand vertical map is the natural embedding
$e_{i,j} \mapsto T_{i,j}^{(1)}$, and the
right hand vertical map is the inclusion from \cref{easy}.
\end{lemma}

\begin{proof}
Taking $u^{-1}$-coefficients in \cref{harder} gives that
$\Drinfeld_{n,r}(T_{i,j}^{(1)}) = \sum_{p=1}^r 
e_{i,j}^{[p]}$. This acts
in the same way as $\drinfeld_{n,r}(e_{i,j})$
on $V^{\otimes r} \otimes_{\kk \S_r} \aH_r$
by \cref{fooly,folly}.
\end{proof}

%% file: s8-diags.tex
\setcounter{section}{7}

\section{Diagrams for the Drinfeld homomorphism}\label{s8-diags}

We would like to find a formula expressing the image of $T_{i,j}(u) \in \Yn\llbracket u^{-1}\rrbracket$ under the homomorphism $\Drinfeld_{n,r}:\Yn \rightarrow \aS(n,r)$ in terms of the basis vectors $\xi_{A,f}$ of the degenerate
affine Schur algebra. To do this, 
since the vectors $v_{\bi^\mu} \barotimes 1$ generate $V^{\otimes r} \otimes_{\kk \S_r} \aH_r$ as a right $\aH_r$-module, we should think about how $T_{i,j}(u)$ acts on $v_{\bi^\mu} \barotimes 1 \in V^{\otimes r} \otimes_{\kk \S_r} \aH_r$ for 
$\mu \in \Lambda(n,r)$.
By the definition of the action, this is the image of $T_{i,j}(u) (v_{\bi^\mu} \otimes 1) \in V^{\otimes r} \otimes \aH_r$ 
under the quotient map
$\pi: V^{\otimes r} \otimes \aH_r
\twoheadrightarrow  V^{\otimes r} \otimes_{\kk \S_r} \aH_r,
v \otimes h \mapsto v \barotimes h$.

Recall that $e_{i,j}^{[p]}\:(1 \leq i,j \leq n, 1 \leq p \leq r)$ denotes the endomorphism of $V^{\otimes r}\otimes \aH_r$
that is $e_{i,j}$ acting on the $p$th tensor position,
and elements of $\aH_r$ are viewed as endomorphisms of this $\kk$-module acting by left multiplication on $\aH_r$ in 
the last tensor factor.
For $1 \leq a, b \leq r$, let
\begin{align}\label{ready}
e_{i,j}^{[a,b]}(u) := 
\begin{dcases}
\sum_{p=a}^b 
\frac{e_{i,j}^{[p]}}{u-x_p}\left(1+\frac{1}{u-x_{p+1}}\right)
\cdots \left(1+\frac{1}{u-x_b}\right)
&\text{if $i \neq j$}\\
\left(1+\frac{1}{u-x_{a}}\right)
\left(1+\frac{1}{u-x_{a+1}}\right)
\cdots \left(1+\frac{1}{u-x_b}\right)
&\text{if $i=j$.}
\end{dcases}
\end{align}
This is $\delta_{i,j}$ if $a > b$.

\begin{lemma}\label{claimylaimy}
Let $\bi = \bi^\mu$ for $\mu \in \Lambda(n,r)$.
Let $a_j := \mu_{<j}+1$
and $b_j := \mu_{\leq j}$; when $\mu_j > 0$, these index the first and last entries that equal $j$ in the increasing multi-index $\bi$, respectively.
For $1 \leq i,j \leq n$,
$T_{i,j}(u)$ acts on $v_{\bi} \otimes 1$ in the same way as
\begin{equation}\label{claimyl}
\sum_{\substack{t \geq 1\\1 \leq j_1<\dots<j_t =j\\ i\geq j_1}}
e_{i,j_1}^{[a_{j_1},b_{j_1}]}(u) e_{j_1,j_2}^{[a_{j_2},b_{j_2}]}(u) \cdots e_{j_{t-1},j_t}^{[a_{j_t},b_{j_t}]}(u).
\end{equation}
\end{lemma}

\begin{proof}
We first show that the expression \cref{claimyl} equals
\begin{equation}\label{claimyr}
\delta_{i,j}+\!\!\!\!
\sum_{\substack{t \geq 1\\1 \leq j_1<\dots<j_t =j}}\!\!\!\!\!
\left(e_{i,j_1}^{[a_{j_1},b_{j_1}]}(u)-\delta_{i,j_1}\right)
\prod_{k=2}^t e_{j_{k-1},j_k}^{[a_{j_k},b_{j_k}]}(u).
\end{equation}
To see this, \cref{claimyr}
is equal to $\delta_{i,j}+A+B+C$ where
$A$, $B$ and $C$ are the sums of the terms of the summation with 
$i=j_1$, $i>j_1$ and $i<j_1$, respectively.
We have that
$\delta_{i,j}+A = A_1-A_2$ where
\begin{align*}
A_1&=\!\!\!\!\!
\sum_{\substack{t \geq 1\\1 \leq j_1<\dots<j_t =j\\i=j_1}}\!\!\!\!
e_{i,j_1}^{[a_{j_1},b_{j_1}]}(u) \prod_{k=2}^t e_{j_{k-1},j_k}^{[a_{j_k},b_{j_k}]}(u),&
A_2 &=\!\!\!\!\!
\sum_{\substack{t \geq 2\\1 \leq j_1<\dots<j_t =j\\  i=j_1}}\!\!\!\!
e_{i,j_2}^{[a_{j_2},b_{j_2}]}(u)
\prod_{k=3}^t e_{j_{k-1},j_k}^{[a_{j_k},b_{j_k}]}(u).\\\intertext{Also}
B&=\!\!\!\!\!
\sum_{\substack{t \geq 1\\1 \leq j_1<\dots<j_t =j\\ i>j_1}}\!\!\!\!
e_{i,j_1}^{[a_{j_1},b_{j_1}]}(u) 
\prod_{k=2}^t e_{j_{k-1},j_k}^{[a_{j_k},b_{j_k}]}(u),&
C&=\!\!\!\!\!
\sum_{\substack{t \geq 1\\1 \leq j_1<\dots<j_t =j\\i<j_1}}\!\!\!\!
e_{i,j_1}^{[a_{j_1},b_{j_1}]}(u)
\prod_{k=2}^t e_{j_{k-1},j_k}^{[a_{j_k},b_{j_k}]}(u).
\end{align*}
It remains to observe that $A_1+B$ is equal to \cref{claimyl}
and $A_2=C$.

By \cref{harder}, $T_{i,j}(u)$
acts on $v_\bi \otimes 1$ in the same way as
\begin{equation}\label{claimy}
\delta_{i,j}+\!\!\!\!\!\!
\sum_{\substack{s\geq 1\\1 \leq p_1 < \cdots < p_s \leq r\\
i_{p_s} = j}}\!\!\!\!\!\!\!\!\!\!
\frac{e_{i,i_{p_1}}^{[p_1]} e_{i_{p_1},i_{p_2}}^{[p_2]}
\cdots e_{i_{p_{s-1}},i_{p_s}}^{[p_s]}}
{(u-x_{p_1})\cdots (u-x_{p_s})}.
\end{equation}
Using the claim established in the previous paragraph, to complete the proof, it suffices 
to show that the expressions \cref{claimy} and \cref{claimyr}
act on $v_\bi \otimes 1$ in the same way.
For $s \geq 1$ and 
$1 \leq p_1<\cdots<p_s\leq r$ with $i_{p_s} = j$, we have that
$\{i_{p_1},\dots,i_{p_s}\} = \{j_1,\dots,j_t\}$
for unique $t \geq 1$ and $1 \leq j_1 < \cdots < j_t = j$.
Consequently, the proof reduces further to showing for any
$t \geq 1$ and $1 \leq j_1 < \cdots < j_t = j$ that
\begin{equation}\label{done}
\sum_{\substack{s\geq 1 \\1 \leq p_1 < \cdots < p_s \leq r\\
\{i_{p_1},\dots,i_{p_s}\} = \{j_1,\dots,j_t\}}}\!\!\!\!\!\!\!\!\!\!\!\!\!\!
\frac{e_{i,i_{p_1}}^{[p_1]} e_{i_{p_1},i_{p_2}}^{[p_2]}
\cdots e_{i_{p_{s-1}},i_{p_s}}^{[p_s]}}
{(u-x_{p_1})\cdots (u-x_{p_s})}(v_\bi \otimes 1)
=
\left(e_{i,j_1}^{[a_{j_1},b_{j_1}]}(u)-\delta_{i,j_1}\right) \prod_{k=2}^t e_{j_{k-1},j_k}^{[a_{j_k},b_{j_k}]}(u) (v_\bi \otimes 1).
\:\:\end{equation}

It remains to prove \cref{done}.
Each $e_{i_{p_{q-1}},i_{p_q}}^{[p_q]}$ on the left hand side
with $i_{p_{q-1}}=i_{p_q}$ 
acts as the identity on $v_\bi \otimes 1$,
so these terms can be omitted.
Then we factor to see that the left hand side equals
$$
\left(\sum_{\substack{s \geq 1 \\ 1 \leq p_1 < \cdots < p_s \leq r\\
i_{p_1}=\cdots=i_{p_s}=j_1}}\!\!\!\!\!\!\!\!
\frac{e_{i,j_1}^{[p_1]}}
{(u-x_{p_1})\cdots (u-x_{p_s})}\right)\times
\prod_{k=2}^t
\left(\sum_{\substack{s \geq 1 \\ 1 \leq p_1 < \cdots < p_s \leq r\\
i_{p_1}=\cdots=i_{p_s}=j_k}}\!\!\!\!\!\!\!\!
\frac{e_{j_{k-1},j_k}^{[p_k]}}{(u-x_{p_1})\cdots (u-x_{p_s})}\right) (v_\bi \otimes 1).
$$
The expression in the first big bracket here is equal to 
$$
\begin{dcases}
\sum_{\substack{s\geq 1 \\ a_{j_1} \leq p_1 < \cdots < p_s \leq b_{j_1}}}\!\!\!\!\!\!\!
\frac{e_{i,j_1}^{[p_1]}}{u-x_{p_1}}
\frac{1}{(u-x_{p_2})\cdots (u-x_{p_s})}&\text{if $i \neq j_1$}\\
\sum_{\substack{s \geq 1 \\ a_{j_1} \leq p_1 < \cdots < p_s \leq b_{j_1}}}\!\!\!\!\!\!\!
\frac{1}
{(u-x_{p_1})\cdots (u-x_{p_s})}
&\text{if $i=j_1$},
\end{dcases}
$$
which is equal to $e_{i,j_1}^{[a_{j_1},b_{j_1}]}(u)-\delta_{i,j_1}$ by \cref{cheap}.
Similarly, the $k$th big bracket in the product is equal to
$e_{j_{k-1},j_k}^{[a_{j_k},b_{j_k}]}(u)$.
This shows that the left hand side of \cref{done} equals the right hand side.
\end{proof}

\begin{lemma}\label{gettingclose}
Suppose that $1 \leq i, j \leq n$ and
$\mu \in \Lambda(n,r)$ with $\mu_j > 0$.
Let $a := \mu_{<j}+1$
and $b:= \mu_{\leq j}$.
The following hold 
for any $\bi \in I(n,r)$ with $i_a = i_{a+1}=\cdots = i_b = j$:
\begin{enumerate}
\item
If $i < j$ then 
$$
\pi\left(e_{i,j}^{[a,b]}(u) (v_\bi \otimes 1)\right)
=
\sum_{d \in (\S_\nu \backslash \S_\mu)_{\min}}
v_\bh \barotimes 
\frac{1}{u-x_a} d
$$
where $v_{\bh} := e_{i,j}^{[a]} v_\bi$ and
$\nu := (\mu_1,\dots,\mu_{j-1},1,\mu_j-1,\mu_{j+1},\dots,\mu_r)$.
\item
If $i = j$ then 
$$
\pi\left(e_{i,j}^{[a,b]}(u) (v_\bi \otimes 1)\right)
=
v_\bi \barotimes 
\prod_{p=a}^b \left(1+\frac{1}{u-x_p}\right).
$$
\item
If $i > j$ then
$$
\pi\left(e_{i,j}^{[a,b]}(u) (v_\bi \otimes 1)\right)
=
\sum_{d \in (\S_\nu \backslash \S_\mu)_{\min}}
v_\bh \barotimes 
\left[\prod_{p=a}^{b-1} \left(1+\frac{1}{u-x_p}\right)\right] \frac{1}{u-x_b}d
$$
where $v_\bh := e_{i,j}^{[b]} v_\bi$ and $\nu := (\mu_1,\dots,\mu_{j-1},\mu_j-1,1,\mu_{j+1},\dots,\mu_r)$.
\end{enumerate}
\end{lemma}

\begin{proof}
(1) Note that the sum $\sum_{d \in (\S_\nu \backslash \S_\mu)_{\min}} d$ is
$\sum_{p=a}^b s_a s_{a+1} \cdots s_{p-1}$.
From the definition \cref{ready}, it is easy to see that
\begin{equation}\label{have}
\pi\left(e_{i,j}^{[a,b]}(u) (v_\bi \otimes 1)\right)
=
v_\bh \barotimes
\sum_{p=a}^b
 s_a s_{a+1} \cdots s_{p-1}
\frac{1}{u-x_p} \left(1+\frac{1}{u-x_{p+1}}\right)
\cdots
\left(1+\frac{1}{u-x_b}\right)
\end{equation}
where $v_\bh := e_{i,j}^{[a]} v_\bi$.
To complete the proof, 
we use induction on $(b-a)$ to show for any $\bh$ such that
$h_{a+1}=\cdots=h_b$ that the right hand side of \cref{have} equals
\begin{equation}\label{want}
v_{\bh} \barotimes \frac{1}{u-x_a}\sum_{p=a}^b s_a s_{a+1} \cdots s_{p-1}.
\end{equation}
The case $a=b$ is trivial. For the induction step, we have by induction that
\begin{multline*}
v_{\bh \cdot s_a}  \barotimes 
\sum_{p=a+1}^b s_{a+1}s_{a+2} \cdots s_{p-1}
\frac{1}{u-x_p} \left(1+\frac{1}{u-x_{p+1}}\right)
\cdots \left(1+\frac{1}{u-x_{b}}\right)
=\\
v_{\bh\cdot s_a} \barotimes
\frac{1}{u-x_{a+1}} 
\sum_{p=a+1}^b s_{a+1}s_{a+2} \cdots s_{p-1}.
\end{multline*}
Using this, we deduce that \cref{have} equals
$$
v_\bh \barotimes
\frac{1}{u-x_a} \left(1+\frac{1}{u-x_{a+1}}\right)
\cdots
\left(1+\frac{1}{u-x_b}\right)
+
v_{\bh} \barotimes
s_a \frac{1}{u-x_{a+1}} 
\sum_{p=a+1}^b s_{a+1}s_{a+2} \cdots s_{p-1}.
$$
Then we use the commutation relation
$s_a \frac{1}{u-x_{a+1}} 
=
\frac{1}{u-x_{a}} s_a 
- \frac{1}{u-x_a}\frac{1}{u-x_{a+1}}$,
which may be proved using \cref{poof},
to deduce that this equals
\begin{multline*}
v_\bh \barotimes
\frac{1}{u-x_a} \left(1+\frac{1}{u-x_{a+1}}\right)
\cdots
\left(1+\frac{1}{u-x_b}\right)
+
v_{\bh} \barotimes
 \frac{1}{u-x_{a}} 
\sum_{p=a+1}^b s_a s_{a+1} \cdots s_{p-1}\\
-
v_{\bh} \barotimes
 \frac{1}{u-x_{a}}  \frac{1}{u-x_{a+1}} 
\sum_{p=a+1}^b s_{a+1} \cdots s_{p-1}.
\end{multline*}
We apply the induction hypothesis to rewrite the third term in this expression to obtain
\begin{multline}\label{new}
v_\bh \barotimes
\frac{1}{u-x_a} \left(1+\frac{1}{u-x_{a+1}}\right)
\cdots
\left(1+\frac{1}{u-x_b}\right)
+
v_{\bh} \barotimes
 \frac{1}{u-x_{a}} 
\sum_{p=a+1}^b s_a s_{a+1} \cdots s_{p-1}\\
-
v_{\bh} \barotimes
\sum_{p=a+1}^b s_{a+1} \cdots s_{p-1}  \frac{1}{u-x_{a}} \frac{1}{u-x_{p}}\left(1+\frac{1}{u-x_{p+1}}\right)\cdots \left(1+\frac{1}{u-x_b}\right) 
.
\end{multline}
In the third term of \cref{new}, the word $s_{a+1} \cdots s_{p-1}$ fixes $v_\bh$,
so it can be removed. Then the first and third terms of \cref{new} 
together give
$v_\bh \barotimes
\frac{1}{u-x_{a}}$ multiplied on the right by
$$
\left(1+\frac{1}{u-x_{a+1}}\right)
\cdots
\left(1+\frac{1}{u-x_b}\right)
-\sum_{p=a+1}^b  \frac{1}{u-x_{p}}\left(1+\frac{1}{u-x_{p+1}}\right)\cdots \left(1+\frac{1}{u-x_b}\right),
$$
which is simply equal to $1$ by \cref{cheap}.
Thus, we have $v_\bh \barotimes
\frac{1}{u-x_{a}}$ plus the second term of \cref{new}, which is the required \cref{want}.

\vspace{2mm}
\noindent
(2) This is obvious from \cref{ready}.

\vspace{2mm}
\noindent
(3) This is proved in a similar way to (1).
The counterpart of \cref{have} when $i > j$ is
\begin{equation}\label{havenot}
\pi\left(e_{i,j}^{[a,b]}(u) (v_\bi \otimes 1)\right)
=
v_\bh \barotimes
\sum_{p=a}^b
 s_{b-1} s_{b-2} \cdots s_{p}
\frac{1}{u-x_p}
\left(1+\frac{1}{u-x_{p+1}}\right)
\cdots
\left(1+\frac{1}{u-x_{b}}\right)
\end{equation}
where $v_\bh := e_{i,j}^{[b]} v_\bi$.
We need to show that the right hand side of this equation equals
\begin{equation}\label{wantnot}
v_{\bh} \barotimes \left(1+\frac{1}{u-x_{a}}\right)\cdots \left(1+\frac{1}{u-x_{b-1}}\right)
\frac{1}{u-x_b}\sum_{p=a}^b s_{b-1} s_{b-2} \cdots s_{p}.
\end{equation}
This follows by an induction argument like in the proof of (1).
\end{proof}

\begin{theorem}\label{SlavaTheorem}
For $1 \leq i,j \leq n$, we have that
\begin{equation}\label{crap}
\Drinfeld_{n,r}\big(T_{i,j}(u)\big)
= 
\sum_{\substack{t \geq 1 \\ 1 \leq j_1 < \cdots < j_t = j\\ i \geq j_1}}
\sum_{\substack{\mu \in \Lambda(n,r)\\
\mu_{j_1} > 0\text{ if }i \neq j_1\\
\mu_{j_2},\dots,\mu_{j_t}>0}}
\xi_{A_\mu[i\geq j_1<\cdots<j_{t}], f_\mu[i\geq j_1<\cdots<j_{t}]}
\end{equation}
where $A_\mu(i\geq j_1 < \dots < j_{t}) \in \Mat{\mu+\eps_{i}-\eps_{j}}{\mu}$
is the $n\times n$ matrix $$
\diag(\mu_1,\dots,\mu_n)+
(e_{i,j_1}-e_{j_1,j_1})+
\sum_{k=2}^t \left(e_{j_{k-1},j_k}-e_{j_k,j_k}\right),
$$
and $f_\mu(i\geq j_1<\cdots<j_t) \in \P_r$ is the polynomial 
$$
\begin{dcases}
\left[\prod_{p=a_{j_1}}^{b_{j_1}-1}\left(1+\frac{1}{u-x_p}\right)\right]
\frac{1}{u-x_{b_{j_1}}}
\left[\prod_{k=2}^t\frac{1}{u-x_{a_{j_k}}}\right]&\text{if $i > j_1$}\\
\left[\prod_{p=a_{j_1}}^{b_{j_1}}\left(1+\frac{1}{u-x_p}\right)\right]
\left[\prod_{k=2}^t\frac{1}{u-x_{a_{j_k}}}\right]
&\text{if $i = j_1$}
\end{dcases}
$$
for $a_j := \mu_{<j}+1$ and $b_j := \mu_{\leq j}$.
\end{theorem}

Before we give the proof, we explain how to work with the formula in \cref{SlavaTheorem} diagrammatically,
using the notation from \cref{kward}.
Let $\mu$ and $[i\geq j_1<\cdots<j_t]$ be as in \cref{crap}.
The double coset diagram 
for the matrix 
$A_\mu[i\geq j_1<\cdots<j_t]$ in \cref{SlavaTheorem} has vertical strings of thickness
$\mu_1,\dots,\mu_n$ at the bottom and thickness
$\mu_1,\dots,\mu_i+1,\dots,\mu_j-1,\dots,\mu_n$ at the top.
When $i>j_1$, there is 
a propagating string
$\begin{tikzpicture}[centerzero]\draw(-.2,-.2) to (.2,.2);
\end{tikzpicture}$ of thickness 1 from the $i$th vertical string at the top to the $j_1$th one at the bottom, 
and there are 
propagating strings 
$\begin{tikzpicture}[centerzero]\draw(-.2,.2) to (.2,-.2);
\end{tikzpicture}$
of thickness 1 from the $j_{k}$th vertical string at the top to the $j_{k+1}$th one at the bottom
for $k=1,\dots,t-1$.
Then the string diagram for the morphism 
$\xi_{A_\mu[i\geq j_1<\dots<j_{t}], f_\mu[i\geq j_1<\dots<j_{t}]}$
is obtained from this by
adding a label
$\begin{tikzpicture}[centerzero]
\circled{0,0}{u};
\end{tikzpicture}$
to each of the non-vertical propagating strings,
and the labels
$\begin{tikzpicture}[centerzero]
\ovaledinverse{0,.15}{u+1};
\circled{0,-.15}{u};
\end{tikzpicture}$ on the $i$th vertical string. These labels should be placed below all of the merges and crossings and above all of the splits.

\begin{example}\label{SlavaExample}
When $n=1$, the Drinfeld homomorphism $\Drinfeld_{1,r}:\Y(\gl_1) \rightarrow \aS(1,r)$ maps
$$
T_{1,1}(u) \mapsto
\begin{tikzpicture}[centerzero]
\node at (0,.8) {$\color{blue}\scriptstyle[1\geq 1]$};
\draw[ultra thick] (0,-.6)\botlabel{r} to (0,.6);
\ovaledinverse{0,.16}{u+1};
\circled{0,-.16}{u};
\end{tikzpicture}.
$$
When $n=2$, the homomorphism $\Drinfeld_{2,r}:\Y(\gl_2) \rightarrow \aS(2,r)$
maps
\begin{align*}
T_{1,1}(u) &\mapsto
\sum_{\mu \in \Lambda(2,r)}
\begin{tikzpicture}[centerzero]
\node at (.35,.8) {$\color{blue}\scriptstyle[1\geq 1]$};
\draw[ultra thick] (0,-.6)\botlabel{\mu_1} to (0,.6);
\ovaledinverse{0,.15}{u+1};
\circled{0,-.15}{u};
\draw[ultra thick] (.7,-.6)\botlabel{\mu_2} to (.7,.6);
\end{tikzpicture}\ ,&
T_{1,2}(u) &\mapsto
\sum_{\substack{\mu \in \Lambda(2,r)\\\mu_2 > 0}}
\begin{tikzpicture}[centerzero]
\node at (.5,.8) {$\color{blue}\scriptstyle[1\geq 1<2]$};
\draw[ultra thick] (0,-.6)\botlabel{\mu_1} to (0,.6);
\draw (.97,-.6) to (.97,-.35) to (.03,.35) to (.03,.6);
\ovaledinverse{0,-.05}{u+1};
\circled{0,-.35}{u};
\circled{0.5,0}{u};
\draw[ultra thick] (1,.6) to (1,-.6)\botlabel{\mu_2};
\end{tikzpicture}\ ,\\
T_{2,1}(u) &\mapsto
\sum_{\substack{\mu \in \Lambda(2,r)\\\mu_1 > 0}}
\begin{tikzpicture}[centerzero]
\node at (.5,.8) {$\color{blue}\scriptstyle(2\geq 1)$};
\draw[ultra thick] (0,-.6)\botlabel{\mu_1} to (0,.6);
\draw (.97,.6) to (.97,.35) to (.03,-.35) to (.03,-.6);
\ovaledinverse{0,.35}{u+1};
\circled{0,.05}{u};
\circled{0.5,0}{u};
\draw[ultra thick] (1,.6) to (1,-.6)\botlabel{\mu_2};
\end{tikzpicture}\ ,&
T_{2,2}(u) &\mapsto
\sum_{\mu \in \Lambda(2,r)}
\begin{tikzpicture}[centerzero]
\node at (.35,.8) {$\color{blue}\scriptstyle[2\geq 2]$};
\draw[ultra thick] (.7,-.6)\botlabel{\mu_2} to (.7,.6);
\draw[ultra thick] (0,-.6)\botlabel{\mu_1} to (0,.6);
\ovaledinverse{.7,.15}{u+1};
\circled{.7,-.15}{u};
\end{tikzpicture}
\ +\ 
\sum_{\substack{\mu \in \Lambda(2,r)\\\mu_1, \mu_2 > 0}}
\begin{tikzpicture}[centerzero]
\node at (.65,.8) {$\color{blue}\scriptstyle[2\geq 1< 2]$};
\draw[ultra thick] (0,-.6)\botlabel{\mu_1} to (0,.6);
\draw (1.27,.6) to (1.27,.5) to (.03,-.5) to (.03,-.6);
\draw (1.27,-.6) to (1.27,-.5) to (.03,.5) to (.03,.6);
\ovaledinverse{0,.15}{u+1};
\circled{0,-.15}{u};
\circled{.4,-.2}{u};
\circled{.9,-.2}{u};
\draw[ultra thick] (1.3,.6) to (1.3,-.6)\botlabel{\mu_2};
\end{tikzpicture}\ .\end{align*}
Also, when $n=3$, the homomorphism $\Drinfeld_{3,r}:\Y(\gl_3) \rightarrow \aS(3,r)$ maps
\begin{align*}
T_{3,1}(u)&\mapsto\sum_{\substack{\mu \in \Lambda(3,r)\\\mu_1 > 0}}
\begin{tikzpicture}[centerzero]
\node at (0,.8) {$\color{blue}\scriptstyle[3\geq 1]$};
\draw[ultra thick] (-.65,-.6) \botlabel{\mu_1} to (-.65,.6);
\draw[ultra thick] (0,-.6) \botlabel{\mu_2} to (0,.6);
\draw[ultra thick] (.65,-.6) \botlabel{\mu_3} to (.65,.6);
\draw (-0.62,-.6) to (-0.62,-.25) to (0.62,.25) to (0.62,.6);
\ovaledinverse{-.65,.35}{u+1};
\circled{-.65,0.04}{u};
\circled{-0.325,-0.125}{u};
\end{tikzpicture},\qquad\qquad
T_{1,3}(u)\mapsto
\sum_{\substack{\mu \in \Lambda(3,r)\\\mu_3 > 0}}
\begin{tikzpicture}[centerzero]
\node at (0,.8) {$\color{blue}\scriptstyle[1\geq 1<3]$};
\draw[ultra thick] (-.65,-.6) \botlabel{\mu_1} to (-.65,.6);
\draw[ultra thick] (0,-.6) \botlabel{\mu_2} to (0,.6);
\draw[ultra thick] (.65,-.6) \botlabel{\mu_3} to (.65,.6);
\draw (0.62,-.6) to (0.62,-.25) to (-.62,.25) to (-.62,.6);
\ovaledinverse{-.65,-.08}{u+1};
\circled{-.65,-0.4}{u};
\circled{0.325,-0.125}{u};
\end{tikzpicture}+\sum_{\substack{\mu \in \Lambda(3,r)\\\mu_2,\mu_3>0}}
\begin{tikzpicture}[centerzero]
\node at (0,.8) {$\color{blue}\scriptstyle[1\geq 1< 2< 3]$};
\draw[ultra thick] (-.65,-.6) \botlabel{\mu_1} to (-.65,.6);
\draw[ultra thick] (0,-.6) \botlabel{\mu_2} to (0,.6);
\draw[ultra thick] (.65,-.6) \botlabel{\mu_3} to (.65,.6);
\draw (-0.03,-.6) to (-0.03,-.05) to (-.62,.45) to (-.62,.6);
\draw (0.62,-.6) to (0.62,-.05) to (0.03,.45) to (0.03,.6);
\ovaledinverse{-.65,-.07}{u+1};
\circled{-.65,-0.4}{u};
\circled{-0.325,0.2}{u};
\circled{0.325,0.2}{u};
\end{tikzpicture}.
\end{align*}
(We have also written the label $[i\geq j_1<\cdots<j_t]$ from \cref{SlavaTheorem} above each diagram.)
\end{example}

\begin{proof}[Proof of \cref{SlavaTheorem}]
It suffices to show that the left and right hand sides of \cref{crap} act in the same way on $v_{\bi^\mu}i\barotimes 1$
for each $\mu \in \Lambda(n,r)$. We fix such a choice of $\mu$ and $\bi$ from now on and let $\bi := \bi^\mu$,
$a_j := \mu_{<j}+1$ and $b_j := \mu_{\leq j}$
as in \cref{claimylaimy}.
Note that both \cref{claimyl,crap}
involve the same summation over
$t \geq 1$ and $1 \leq j_1<\cdots<j_t=j$ with $i \geq j_1$.
Also
$e_{i,j_1}^{[a_{j_1},b_{j_1}]}
e_{j_1,j_2}^{[a_{j_2},b_{j_2}]}
\cdots e_{j_{t-1},j_t}^{[a_{j_t},b_{j_t}]}(v_\bi \otimes 1) = 0$
if
$\mu_{j_1}=0$ and $i \neq j_1$ or if any of
$\mu_{j_2},\dots,\mu_{j_t}$ are equal to 0.
Consequently, applying \cref{claimylaimy}, the theorem follows if we can show that
\begin{equation}\label{weneedit}
\pi\big(
e_{i,j_1}^{[a_{j_1},b_{j_1}]}
e_{j_1,j_2}^{[a_{j_2},b_{j_2}]}
\cdots 
e_{j_{t-1},j_t}^{[a_{j_{t}},b_{j_t}]}
(v_\bi \otimes 1)\big)
= 
\xi_{A_\mu[i\geq j_1 < \cdots < j_t],f_\mu[i\geq j_1<\cdots<j_t]}
(v_\bi \barotimes 1),
\end{equation}
for
$t \geq 1$ and $1 \leq j_1<\cdots<j_t=j$ such that $i \geq j_1$,
$\mu_{j_1}>0$ if $i \neq  j_1$, and $\mu_{j_1},\dots,\mu_{j_t}>0$.

To prove \cref{weneedit},
let $\lambda := \mu + \eps_i - \eps_j$,
$A := A_\mu[i\geq j_1 < \cdots < j_t] \in \Mat{\lambda}{\mu}$ 
and $f := f_\mu[i\geq j_1<\cdots<j_t]$.
As in \cref{notaproof}, we have that $\xi_{A,f} = \alpha\circ\beta\circ\gamma\circ\delta$
where $\delta \in 1_{\mu(A)} \aS(n,r) 1_\mu$ is defined by the bottom horizontal strip of the string diagram of $\xi_{A,f}$, i.e., the splits,
$\gamma \in 1_{\mu(A)} \aS(n,r)1_{\mu(A)} $ is the defined by the next horizontal strip up, i.e., the pin labelled by $f$,
$\beta\in 1_{\lambda(A)} \aS(n,r)1_{\mu(A)}$ comes from the strip above that, i.e., the crossings
of propagating strings, and $\alpha\in 1_\lambda \aS(n,r) 1_{\lambda(A)}$ is 
defined by the top horizontal strip, i.e., the merges.
Similarly to \cref{notaproof}, we have that
$$
\delta \circ \gamma (v_\bi \barotimes 1)
= \sum_{d \in (\S_\mu / \S_{\mu(A)})_{\min}}
v_{\bi^{\mu(A)}} \barotimes f d.
$$
By the definitions, we have that 
$$
\alpha\circ \beta (v_{\bi^{\mu(A)}} \barotimes 1) = v_\bh \barotimes 1
$$
for $\bh \in I(n,r)$ defined so that 
$v_\bh = e_{i,j_1}^{[a_{j_1}]}
e_{j_1,j_2}^{[a_{j_2}]} \cdots e_{j_{t-1},j_t}^{[a_{j_t}]} v_{\bi}$.
Hence, the right hand side of \cref{weneedit} is equal to
$$
\sum_{d \in (\S_\mu / \S_{\mu(A)})_{\min}}
v_\bh \barotimes f d.
$$
This is equal to the left hand side of \cref{weneedit}
by \cref{gettingclose}.
\end{proof}

\begin{corollary}\label{trans}
The following diagram commutes
$$
\begin{tikzcd}
\arrow[d,"\Drinfeld_{n,r}" left]\Yn\arrow[r,"\tau" above]&\Yn\arrow[d,"\Drinfeld_{n,r}" right]\\
\aS(n,r)\arrow[r,"\div" below]&\aS(n,r).
\end{tikzcd}
$$
\end{corollary}

\begin{proof}
By \cref{easygen}, it suffices to check that
$\Drinfeld_{n,r}(\tau(X))
= \Drinfeld_{n,r}(X)^\div$
just for $X= T_{1,1}(u)$ and for $X = T_{i,j}^{(1)}$.
When $X = T_{1,1}(u)$, this follows from the form of \cref{crap}.
When $X = T_{i,j}^{(1)}$, it follows using \cref{folly,silly}.
\end{proof}

\begin{remark}
Using \cref{trans,SlavaTheorem}, one obtains another formula describing $\Drinfeld_{n,r}$ on the RTT generators: we have that
\begin{equation}\label{crapper}
\Drinfeld_{n,r}\big(T_{i,j}(u)\big)
= 
\sum_{\substack{t \geq 1 \\ i=  i_1 > \cdots > i_t \geq 1\\ i_t \leq j}}
\sum_{\substack{\mu \in \Lambda(n,r)\\ \mu_{i_1},\dots,\mu_{i_{t-1}}>0\\\mu_t > 0\text{ if }i \neq j_t}}
\xi_{f_\mu[i_1>\dots>i_{t}\leq j],A_\mu[i_1> \cdots>i_{t}\leq j]}
\end{equation}
where $A_\mu[i_1>\cdots>i_t\leq j] \in \Mat{\mu+\eps_{i}-\eps_{j}}{\mu}$
is the $n\times n$ matrix $$
\diag(\mu_1,\dots,\mu_n)+
\sum_{k=1}^{t-1} \left(e_{i_{k},i_{k+1}}-e_{i_k,i_k}\right)+(e_{i_t,j}-e_{i_t,i_t}),
$$
and $f_\mu[i_1>\cdots>i_t\leq j] \in \P_r$ is the polynomial 
$$
\begin{dcases}
\left[\prod_{k=1}^{t-1}\frac{1}{u-x_{a_{i_k}}}\right]
\left[\prod_{p=a_{i_t}}^{b_{i_t}-1}\left(1+\frac{1}{u-x_p}\right)\right]
\frac{1}{u-x_{b_{i_t}}}&\text{if $i_t < j$}\\
\left[\prod_{k=1}^{t-1}\frac{1}{u-x_{a_{i_1}}}\right]
\left[\prod_{p=a_{i_t}}^{b_{i_t}}\left(1+\frac{1}{u-x_p}\right)\right]
&\text{if $i_t = j$}
\end{dcases}
$$
for $a_i := \mu_{<i}+1$ and $b_i := \mu_{\leq i}$.
For example,
$\Drinfeld_{3,r}:\Y(\gl_3) \rightarrow \aS(3,r)$ maps
\begin{align*}
T_{3,1}(u)&\mapsto\sum_{\substack{\mu \in \Lambda(3,r)\\\mu_1 > 0}}
\begin{tikzpicture}[centerzero]
\node at (0,.8) {$\color{blue}\scriptstyle[3> 1 \leq 1]$};
\draw[ultra thick] (-.65,-.6) \botlabel{\mu_1} to (-.65,.6);
\draw[ultra thick] (0,-.6) \botlabel{\mu_2} to (0,.6);
\draw[ultra thick] (.65,-.6) \botlabel{\mu_3} to (.65,.6);
\draw (-0.62,-.5) to (-0.62,-.25) to (0.62,.25) to (0.62,.6);
\ovaledinverse{-.65,.35}{u+1};
\circled{-.65,0.04}{u};
\circled{0.325,0.125}{u};
\end{tikzpicture}+\sum_{\substack{\mu \in \Lambda(3,r)\\\mu_1>0,\mu_2>0}}
\begin{tikzpicture}[centerzero]
\node at (0,.8) {$\color{blue}\scriptstyle[3 > 2> 1\leq 1]$};
\draw[ultra thick] (-.65,-.6) \botlabel{\mu_1} to (-.65,.6);
\draw[ultra thick] (0,-.6) \botlabel{\mu_2} to (0,.6);
\draw[ultra thick] (.65,-.6) \botlabel{\mu_3} to (.65,.6);
\draw (0.03,-.6) to (0.03,-.45) to (.62,.05) to (.62,.6);
\draw (-0.62,-.6) to (-0.62,-.45) to (-0.03,.05) to (-0.03,.6);
\ovaledinverse{-.65,.35}{u+1};
\circled{-.65,0.04}{u};
\circled{-0.325,-0.2}{u};
\circled{0.325,-0.2}{u};
\end{tikzpicture},\qquad\qquad
T_{1,3}(u)\mapsto
\sum_{\substack{\mu \in \Lambda(3,r)\\\mu_3 > 0}}
\begin{tikzpicture}[centerzero]
\node at (0,.8) {$\color{blue}\scriptstyle[1 \leq 3]$};
\draw[ultra thick] (-.65,-.6) \botlabel{\mu_1} to (-.65,.6);
\draw[ultra thick] (0,-.6) \botlabel{\mu_2} to (0,.6);
\draw[ultra thick] (.65,-.6) \botlabel{\mu_3} to (.65,.6);
\draw (0.62,-.6) to (0.62,-.15) to (-.62,.35) to (-.62,.6);
\ovaledinverse{-.65,-.08}{u+1};
\circled{-.65,-0.4}{u};
\circled{-0.325,0.225}{u};
\end{tikzpicture}.
\end{align*}
(The expressions in this example 
can also be derived directly from the ones for
$T_{3,1}(u)$ and $T_{1,3}(u)$ from \cref{SlavaExample} using \cref{averagedotslide}.)
\end{remark}

Take $m \geq 0$.
Tensoring with the object $(m)$ either on the right or the left defines
$\kk$-linear functors $-\star(m):\aSchur \rightarrow \aSchur$
and $(m)\star-:\aSchur\rightarrow \aSchur$.
On string diagrams, $-\star(m)$ adds a vertical string of thickness $m$ on the right hand side, and $(m)\star -$ adds such a string on the left hand side.
Recalling that $\aS(n,r)$ is the path algebra of the full subcategory of $\aSchur$ with object set $\Lambda(n,r)$, these functors
induce a pair of algebra homomorphisms
\begin{align}
\varphi_m:\aS(n,r) &\rightarrow \aS(n+1,r+m),&
\xi_{A,f} &\mapsto \xi_{\diag(A,(m)),f \otimes 1},\\
\psi_m:\aS(n,r) &\rightarrow \aS(n+1,m+r),&
\xi_{A,f} &\mapsto \xi_{\diag((m),A),1\otimes f}.
\end{align}
There are also homomorphisms 
\begin{align}
\varphi:\Yn &\rightarrow \Y(\gl_{n+1}),&
\psi:\Yn &\rightarrow \Y(\gl_{n+1}).
\end{align}
The first is the natural embedding
taking $T_{i,j}(u)$ to $T_{i,j}(u)$ for $1 \leq i,j \leq n$.
The second is defined in \cite{NT} by the formula
\begin{equation}
\psi := \omega_{n+1} \circ \varphi \circ \omega_n.
\end{equation}
Because of this formula, 
we sometimes call $\psi$ the {\em unnatural} embedding.
By \cite[Lem.~4.2]{BKdrinfeld}, we have that
\begin{equation}\label{hurt}
\psi(T_{i,j}(u)) = T_{i+1,j+1}(u)- 
T_{i+1,1}(u) T_{1,1}(u)^{-1} T_{1,j+1}(u).
\end{equation}

The following theorem explains a sense in which $\varphi$ corresponds to $\varphi_m$ and $\psi$ corresponds to $\psi_m$.
For the statement, recall from the previous section 
that $\eta_{-1}$ is the shift automorphism mapping $T_{i,j}(u)$ to $T_{i,j}(u-1)$.

\begin{lemma}\label{shifty}
For any $m,n,r\geq 1$, the following diagrams commute
$$
\begin{tikzcd}
\Yn\arrow[r,"\varphi" above]\arrow[d,"\Drinfeld_{n,r}" left]
&\Y(\gl_{n+1})\arrow[d,"\Drinfeld_{n+1,r+m}" right]\\
\aS(n,r)\arrow[r,"\varphi_m" below]&\aS(n+1,r+m)
\end{tikzcd},
\qquad
\begin{tikzcd}
\Yn\arrow[r,"\psi\circ \eta_{-1}" above]\arrow[d,"\Drinfeld_{n,r}" left]
&\Y(\gl_{n+1})\arrow[d,"\Drinfeld_{n+1,m+r}" right]\\
\aS(n,r)\arrow[r,"\psi_m " below]&\aS(n+1,m+r)
\end{tikzcd}.
$$
\end{lemma}

\begin{proof}
The commutativity of the first diagram follows because the formula for $\Drinfeld_{n,r}(T_{i,j}(u))$ from \cref{crap} only involves $j_1,\dots,j_t \leq \min(i,j)$, i.e., it
is the same for any $n \geq \max(i,j)$.

To prove that the second diagram commutes, in view of \cref{easygen},
it suffices to show that \begin{equation}\label{sufficient}
\Drinfeld_{n+1,m+r}\circ \psi\circ \eta_{-1}(X)
= \psi_m \circ \Drinfeld_{n,r}(X)
\end{equation}
for $X = T_{1,1}^{(d)}\:(d \geq 1)$ and for 
$X = T_{i,j}^{(1)}\:(1 \leq i,j \leq n)$.
To check the equality \cref{sufficient} 
for $X = T_{i,j}^{(1)}$, we have that
$\Drinfeld_{n+1,m+r}\circ \psi\circ \eta_{-1}(T_{i,j}^{(1)})
= \Drinfeld_{n+1,m+r}(T_{i+1,j+1}^{(1)})
= \drinfeld_{n+1,m+r}(e_{i+1,j+1})$
and
$\psi_m \circ \Drinfeld_{n,r}(T_{i,j}^{(1)}) 
= \psi_m(\drinfeld_{n+1,m+r}(e_{i,j})$,
using \cref{silly}. These are easily seen to be equal by \cref{folly}.
It remains to check the equality \cref{sufficient}
when $X = T_{1,1}^{(d)}\:(d \geq 1)$. Using the commutativity of the first diagram,
this follows if we can show that
$$
\Drinfeld_{2,m+r}\circ \psi\circ \eta_{-1}(T_{1,1}(u))
= \psi_m \circ \Drinfeld_{1,r}(T_{1,1}(u)).
$$
By \cref{hurt}, the left hand side is $\Drinfeld_{2,m+r}\left(T_{2,2}(u-1)-
T_{2,1}(u-1)T_{1,1}(u-1)^{-1}T_{1,2}(u-1)\right)$.
Using the information in \cref{SlavaExample}, 
we are reduced to showing that
\begin{equation}\label{cheesecake}
\sum_{\substack{\mu \in \Lambda(2,r)\\\mu_1, \mu_2 > 0}}
\begin{tikzpicture}[anchorbase]
\draw[ultra thick] (0,-1.2)\botlabel{\mu_1} to (0,1.2);
\draw (1.27,1.2) to (1.27,.9) to (.03,-.9) to (.03,-1.2);
\draw (1.27,-1.2) to (1.27,-.9) to (.03,.9) to (.03,1.2);
\circledinverse{0,.37}{u};
\ovaled{0,.05}{u-1};
\ovaled{.44,-.24}{u-1};
\ovaled{.93,-.49}{u-1};
\draw[ultra thick] (1.3,1.2) to (1.3,-1.2)\botlabel{\mu_2};
\end{tikzpicture}\ 
-
\sum_{\substack{\mu \in \Lambda(2,r)\\\mu_2 > 0}}
\begin{tikzpicture}[anchorbase]
\draw[ultra thick] (-.2,-1.6)\botlabel{\mu_1} to (-0.2,.8);
\draw (.97,.8) to (.97,.4) to (-.17,-.15) to (-.17,-.9) to (.97,-1.45) to (.97,-1.6);
\circledinverse{-0.2,.6}{u};
\ovaled{-0.2,.3}{u-1};
\circled{-.2,-.37}{u};
\ovaledinverse{-0.2,-.68}{u-1};
\circledinverse{-.2,-1.1}{u};
\ovaled{-0.2,-1.4}{u-1};
\ovaled{0.4,-1.2}{u-1};
\ovaled{0.4,.1}{u-1};
\draw[ultra thick] (1,.8) to (1,-1.6)\botlabel{\mu_2};
\end{tikzpicture}
=\sum_{\mu \in \Lambda(2,r)}
\begin{tikzpicture}[anchorbase]
\draw[ultra thick] (0,-1.6)\botlabel{\mu_1} to (0,.8);
\draw[ultra thick] (.7,-1.6)\botlabel{\mu_2} to (.7,.8);
\ovaledinverse{0.7,-.24}{u+1};
\circled{0.7,-.56}{u};
\end{tikzpicture}\ -
\sum_{\mu \in \Lambda(2,r)}
\begin{tikzpicture}[anchorbase]
\draw[ultra thick] (.7,-1.2)\botlabel{\mu_2} to (.7,1.2);
\draw[ultra thick] (0,-1.2)\botlabel{\mu_1} to (0,1.2);
\circledinverse{.7,.16}{u};
\ovaled{.7,-.16}{u-1};
\end{tikzpicture}\ .
\end{equation}
Using the coproduct relations \cref{game,istanbul},
then the $r=1$ case of \cref{averagedotslide}(1), then \cref{skiving2}, the first term on the left hand side 
of \cref{cheesecake} equals
\begin{align*}
\sum_{\substack{\mu \in \Lambda(2,r)\\\mu_1, \mu_2 > 0}}
\begin{tikzpicture}[anchorbase]
\draw[ultra thick] (0,-.9)\botlabel{\mu_1} to (0,.9);
\draw (1.27,.9) to (1.27,.4) to (.03,-.4) to (.03,-.9);
\draw (1.27,-.9) to (1.27,-.4) to (.03,.4) to (.03,.9);
\circledinverse{0,.63}{u};
\ovaled{0,-.63}{u-1};
\circled{.38,.18}{u};
\ovaled{.9,-.18}{u-1};
\draw[ultra thick] (1.3,.9) to (1.3,-.9)\botlabel{\mu_2};
\end{tikzpicture}\ 
&=
\sum_{\substack{\mu \in \Lambda(2,r)\\\mu_1, \mu_2 > 0}}
\begin{tikzpicture}[anchorbase]
\draw[ultra thick] (0,-.9)\botlabel{\mu_1} to (0,.9);
\draw (1.27,.9) to (1.27,.7) to (.03,0) to (.03,-.9);
\draw (1.27,-.9) to (1.27,-.4) to (.03,.3) to (.03,.9);
\circledinverse{0,.63}{u};
\ovaled{0,-.63}{u-1};
\circled{.55,0.03}{u};
\ovaled{.94,-.2}{u-1};
\draw[ultra thick] (1.3,.9) to (1.3,-.9)\botlabel{\mu_2};
\end{tikzpicture}
+
\sum_{\substack{\mu \in \Lambda(2,r)\\\mu_1, \mu_2 > 0}}
\begin{tikzpicture}[anchorbase]
\draw[ultra thick] (0,-.9)\botlabel{\mu_1} to (0,.9);
\draw (1.27,.9) to (1.27,.4) to[out=-150,in=150,looseness=2] (1.27,-.4) to(1.27,-.9);
\draw (.03,-.9) to (.03,-.4) to[out=30,in=-30,looseness=2] (.03,.4) to (.03,.9);
\circledinverse{0,.63}{u};
\ovaled{0,-.63}{u-1};
\circled{.34,.18}{u};
\circled{.96,.18}{u};
\ovaled{.92,-.18}{u-1};
\draw[ultra thick] (1.3,.9) to (1.3,-.9)\botlabel{\mu_2};
\end{tikzpicture}\\
&=
\sum_{\substack{\mu \in \Lambda(2,r)\\\mu_1, \mu_2 > 0}}
\begin{tikzpicture}[anchorbase]
\draw[ultra thick] (0,-.9)\botlabel{\mu_1} to (0,.9);
\draw (1.27,.9) to (1.27,.7) to (.03,0) to (.03,-.9);
\draw (1.27,-.9) to (1.27,-.4) to (.03,.3) to (.03,.9);
\circledinverse{0,.63}{u};
\ovaled{0,-.63}{u-1};
\circled{.55,0.03}{u};
\ovaled{.94,-.2}{u-1};
\draw[ultra thick] (1.3,.9) to (1.3,-.9)\botlabel{\mu_2};
\end{tikzpicture}
+
\sum_{\substack{\mu \in \Lambda(2,r)\\\mu_2 > 0}}
\begin{tikzpicture}[anchorbase]
\draw[ultra thick] (0,-.9)\botlabel{\mu_1} to (0,.9);
\draw (1.27,.9) to (1.27,.4) to[out=-150,in=150,looseness=2] (1.27,-.4) to(1.27,-.9);
\circledinverse{0,.16}{u};
\ovaled{0,-.16}{u-1};\circled{.96,.18}{u};
\ovaled{.92,-.18}{u-1};
\draw[ultra thick] (1.3,.9) to (1.3,-.9)\botlabel{\mu_2};
\end{tikzpicture}
-
\sum_{\substack{\mu \in \Lambda(2,r)\\\mu_2 > 0}}
\begin{tikzpicture}[anchorbase]
\draw[ultra thick] (0,-.9)\botlabel{\mu_1} to (0,.9);
\draw (1.27,.9) to (1.27,.4) to[out=-150,in=150,looseness=2] (1.27,-.4) to(1.27,-.9);
\circled{.96,.18}{u};
\ovaled{.92,-.18}{u-1};
\draw[ultra thick] (1.3,.9) to (1.3,-.9)\botlabel{\mu_2};
\end{tikzpicture}\ .\\\intertext{Using the coproduct relations then the merge-split relation \cref{mergesplit}, the second term on the left hand side of \cref{cheesecake} equals}
\sum_{\substack{\mu \in \Lambda(2,r)\\\mu_2 > 0}}
\begin{tikzpicture}[anchorbase]
\draw[ultra thick] (-.2,-1.6)\botlabel{\mu_1} to (-0.2,.2);
\draw (.97,.2) to (.97,0) to (-.17,-.55) to (-.17,-.9) to (.97,-1.45) to (.97,-1.6);
\circledinverse{-0.2,-.2}{u};
\ovaled{-0.2,-1.3}{u-1};
\circled{.1,-1.03}{u};
\ovaled{0.55,-1.25}{u-1};
\draw[ultra thick] (1,.2) to (1,-1.6)\botlabel{\mu_2};
\end{tikzpicture}&=
\sum_{\substack{\mu \in \Lambda(2,r)\\\mu_1, \mu_2 > 0}}
\begin{tikzpicture}[anchorbase]
\draw[ultra thick] (0,-.9)\botlabel{\mu_1} to (0,.9);
\draw (1.27,.9) to (1.27,.7) to (.03,0) to (.03,-.9);
\draw (1.27,-.9) to (1.27,-.4) to (.03,.4) to (.03,.9);
\circledinverse{0,.63}{u};
\ovaled{0,-.63}{u-1};
\circled{.55,0.03}{u};
\ovaled{.94,-.2}{u-1};
\draw[ultra thick] (1.3,.9) to (1.3,-.9)\botlabel{\mu_2};
\end{tikzpicture}
+
\sum_{\substack{\mu \in \Lambda(2,r)\\\mu_2 > 0}}
\begin{tikzpicture}[anchorbase]
\draw[ultra thick] (0,-.9)\botlabel{\mu_1} to (0,.9);
\draw (1.27,.9) to (1.27,.4) to[out=-150,in=150,looseness=2] (1.27,-.4) to(1.27,-.9);
\circledinverse{0,.16}{u};
\ovaled{0,-.16}{u-1};
\circled{.96,.18}{u};
\ovaled{.92,-.18}{u-1};
\draw[ultra thick] (1.3,.9) to (1.3,-.9)\botlabel{\mu_2};
\end{tikzpicture}\ .
\end{align*}
Subtracting 
gives that the left hand side of \cref{cheesecake} equals
$$
-\sum_{\substack{\mu \in \Lambda(2,r)\\\mu_2 > 0}}
\begin{tikzpicture}[anchorbase]
\draw[ultra thick] (0.3,-.6)\botlabel{\mu_1} to (0.3,.6);
\draw (1.27,.6) to (1.27,.4) to[out=-150,in=150,looseness=2] (1.27,-.4) to(1.27,-.6);
\circled{.96,.18}{u};
\ovaled{.92,-.18}{u-1};
\draw[ultra thick] (1.3,.6) to (1.3,-.6)\botlabel{\mu_2};
\end{tikzpicture}
=\sum_{\substack{\mu \in \Lambda(2,r)\\\mu_2 > 0}}
\begin{tikzpicture}[anchorbase]
\draw[ultra thick] (0.3,-.6)\botlabel{\mu_1} to (0.3,.6);
\draw (1.27,.6) to (1.27,.4) to[out=-150,in=150,looseness=2] (1.27,-.4) to(1.27,-.6);
\circled{.9,0}{u};
\draw[ultra thick] (1.3,.6) to (1.3,-.6)\botlabel{\mu_2};
\end{tikzpicture}-
\sum_{\substack{\mu \in \Lambda(2,r)\\\mu_2 > 0}}
\begin{tikzpicture}[anchorbase]
\draw[ultra thick] (0.3,-.6)\botlabel{\mu_1} to (0.3,.6);
\draw (1.27,.6) to (1.27,.4) to[out=-150,in=150,looseness=2] (1.27,-.4) to(1.27,-.6);
\ovaled{.9,0}{u-1};
\draw[ultra thick] (1.3,.6) to (1.3,-.6)\botlabel{\mu_2};
\end{tikzpicture}\ .
$$
For the equality here, we used that
$-\frac{1}{(u-x)(u-1-x)} = \frac{1}{u-x}-\frac{1}{u-1-x}$.
This is equal to the right hand side of \cref{cheesecake} 
thanks to \cref{skiving}.
\end{proof}

Now we switch from the RTT generators $T_{i,j}^{(d)}$ for the Yangian to the Drinfeld
generators $D_i^{(d)} (d \geq 0, i=1,\dots,n)$ and
$E_i^{(d)}, F_i^{(d)} (d \geq 1, i=1,\dots,n-1)$. 
These generate $\Yn$ subject to relations which are recorded\footnote{We cite this relatively recent paper because the version of the relations recorded there are valid even if $2=0$ in the ground ring $\kk$.} 
in \cite[Th.~4.3]{BT}.
We briefly recall their definition
following \cite[Sec.~5]{BKdrinfeld}:
since the leading minors of the matrix $T(u)$ are invertible, 
it possesses a Gauss factorization
\begin{equation}\label{gfact}
T(u) = F(u) D(u) E(u)
\end{equation}
for unique matrices
$$
D(u) = \left(
\begin{array}{cccc}
D_{1}(u) & 0&\cdots&0\\
0 & D_{2}(u) &\cdots&0\\
\vdots&\vdots&\ddots&\vdots\\
0&0 &\cdots&D_{n}(u)
\end{array}
\right),
$$$$
E(u) = 
\left(
\begin{array}{cccc}
1 & E_{1,2}(u) &\cdots&E_{1,n}(u)\\
0 & 1 &\cdots&E_{2,n}(u)\\
\vdots&\vdots&\ddots&\vdots\\
0&0 &\cdots&1
\end{array}
\right),\qquad
F(u) = \left(
\begin{array}{cccc}
1 & 0 &\cdots&0\\
F_{1,2}(u) & 1 &\cdots&0\\
\vdots&\vdots&\ddots&\vdots\\
F_{1,n}(u)&F_{2,n}(u) &\cdots&1
\end{array}
\right).
$$
This defines the formal power series
$D_i(u) = \sum_{d \geq 0} D_i^{(d)} u^{-d}$, 
$E_{i,j}(u) = \sum_{d \geq 1} E_{i,j}^{(d)} u^{-d}$ and 
$F_{i,j}(u) = \sum_{d \geq 1} F_{i,j}^{(d)} u^{-d}$.
In particular, we have that $D_i^{(0)} = 1$.
Finally, we let $E_i(u) = \sum_{d \geq 1} E_i^{(d)} u^{-d}
:= E_{i,i+1}(u)$ and $F_i(u) = \sum_{d \geq 1} F_i^{(d)} u^{-d}:= 
F_{i,i+1}(u)$ for short,
and have constructed the Drinfeld generators
from the RTT generators.
It is obvious from the definition that
\begin{align}\label{easyembed}
\varphi(D_i(u)) &= D_i(u),&
\varphi(E_i(u))&=E_i(u)),&
\varphi(F_i(u))&=F_i(u).\intertext{Less obvious is that}
\psi(D_i(u))&=D_{i+1}(u),&
\psi(E_i(u))&=E_{i+1}(u)),&
\psi(F_i(u))&=F_{i+1}(u);\label{hardembed}
\end{align}
e.g., see \cite[Lem.~5.1]{BKdrinfeld}.

\begin{theorem}\label{def}
The Drinfeld homomorphism $\Drinfeld_{n,r}:\Yn \rightarrow \aS(n,r)$ maps
\begin{align}\label{Dform}
D_i(u)&\mapsto
\sum_{\lambda \in \Lambda(n,r)}
\begin{tikzpicture}[centerzero,scale=.9]
\draw[ultra thick] (-1.5,-.6)\botlabel{\lambda_1} to (-1.5,.6);
\node at (-1.08,0) {$\dots$};
\draw[ultra thick] (-.7,-.6)\botlabel{\lambda_{i\!-\!1}\ } to (-.7,.6);
\draw[ultra thick] (0,-.6)\botlabel{\lambda_i} to (0,.6);
\draw[ultra thick] (.7,-.6)\botlabel{\ \lambda_{i\!+\!1}} to (.7,.6);
\node at (1.12,0) {$\dots$};
\draw[ultra thick] (1.5,-.6)\botlabel{\lambda_n} to (1.5,.6);
\ovaledinverse{0,.18}{u+i};
\ovaled{0,-.18}{u+i-1};
\end{tikzpicture}\end{align}
for $i=1,\dots,n$, and
\begin{align}\label{Eform}
E_i(u)&\mapsto
\!\sum_{\substack{\mu \in \Lambda(n,r)\\\mu_{i+1}>0}}
\begin{tikzpicture}[centerzero,scale=.9]
\draw[ultra thick] (-1.5,-.6)\botlabel{\mu_1} to (-1.5,.6);
\node at (-1.08,0) {$\dots$};
\draw[ultra thick] (-.7,-.6)\botlabel{\mu_{i\!-\!1}\ } to (-.7,.6);
\draw[ultra thick] (-.2,-.6)\botlabel{\mu_i} to (-.2,.6);
\draw[ultra thick] (0.8,-.6)\botlabel{\mu_{i\!+\!1}\ } to (0.8,.6);
\draw[ultra thick] (1.3,-.6)\botlabel{\ \mu_{i\!+\!2}} to (1.3,.6);
\draw (.77,-.6) to (.77,-.3) to (-0.17,.3) to (-.17,.6);
\node at (1.72,0) {$\dots$};
\draw[ultra thick] (2.1,-.6)\botlabel{\mu_n} to (2.1,.6);
\ovaled{0.3,0}{u+i};
\end{tikzpicture}\ ,&
F_i(u)&\mapsto\!
\sum_{\substack{\mu \in \Lambda(n,r)\\\mu_{i}>0}}
\begin{tikzpicture}[centerzero,scale=.9]
\draw[ultra thick] (-1.5,-.6)\botlabel{\mu_1} to (-1.5,.6);
\node at (-1.08,0) {$\dots$};
\draw[ultra thick] (-.7,-.6)\botlabel{\mu_{i\!-\!1}\ } to (-.7,.6);
\draw[ultra thick] (-.2,-.6)\botlabel{\mu_i} to (-.2,.6);
\draw[ultra thick] (0.8,-.6)\botlabel{\mu_{i\!+\!1}\ } to (0.8,.6);
\draw[ultra thick] (1.3,-.6)\botlabel{\ \mu_{i\!+\!2}} to (1.3,.6);
\draw (.77,.6) to (.77,.3) to (-0.17,-.3) to (-.17,-.6);
\node at (1.72,0) {$\dots$};
\draw[ultra thick] (2.1,-.6)\botlabel{\mu_n} to (2.1,.6);
\ovaled{0.3,0}{u+i};
\end{tikzpicture}
\end{align}
for $i=1,\dots,n-1$.
\end{theorem}

\begin{proof}
We first prove \cref{Dform}.
When $n=1$, 
it is true since
$D_1(u) = T_{1,1}(u)$, and we computed 
$\Drinfeld_{1,r}(T_{1,1}(u))$ in \cref{SlavaExample}.
Using \cref{shifty} for the natural embedding $\varphi$ and \cref{easyembed}, it follows
that \cref{Dform} holds for $i=1$ and all $n \geq 1$.
Using \cref{shifty} for the unnatural embedding $\psi$ and \cref{hardembed},
for any $0 \leq m \leq r$, $n > 1$ and $1 \leq i \leq n-1$,
we have that
$$
\Drinfeld_{n,r}(D_{i+1}(u))
=
\Drinfeld_{n,r} \circ \psi \circ \eta_{-1} (D_i(u+1))
=
\psi_m \circ \Drinfeld_{n-1,r-m}(D_i(u+1)).
$$
Using this identity, \cref{Dform}
for $i > 1$ follows by induction on $i$.

A similar induction argument can be used to prove \cref{Eform}, reducing the proofs of these to checking them just for
$\Drinfeld_{2,r}(E_1(u))$ and $\Drinfeld_{2,r}(F_1(u))$.
Since $E_1(u) = T_{1,1}(u)^{-1} T_{1,2}(u)$
and $F_1(u) = T_{2,1}(u) T_{1,1}(u)^{-1}$ by the definitions, 
these special cases may be checked using the $n=2$ examples in \cref{SlavaExample}, \cref{game,istanbul}.
\end{proof}

\begin{corollary}\label{EF}
For $1 \leq i < j \leq n$, the Drinfeld homomorphism $\Drinfeld_{n,r}$ maps 
\begin{align}\label{EF1}
E_{i,j}(u)&\mapsto \sum_{\substack{\mu\in\Lambda(n,r)\\\mu_j >0}}
\xi_{1/(u+i-x_{\mu_{\leq i}+1}),\ \diag(\mu_1,\dots,\mu_n)+e_{i,j}-e_{j,j}},\\\label{EF2}
F_{i,j}(u)&\mapsto \sum_{\substack{\mu\in\Lambda(n,r)\\\mu_i >0}}
\xi_{\diag(\mu_1,\dots,\mu_n)+e_{j,i}-e_{i,i},\ 1/(u+i-x_{\mu_{\leq i}})}.
\end{align}
\end{corollary}

The diagrams for \cref{EF1,EF2} are similar to the ones in \cref{Eform}, but the diagonal string connects the $i$th vertical string to the $j$th vertical string, like in the following examples:
\begin{align*}
E_{1,3}(u)&\mapsto
\sum_{\substack{\mu \in \Lambda(n,r)\\\mu_3>0}}
\begin{tikzpicture}[centerzero,scale=.9]
\draw[ultra thick] (-.5,-.6)\botlabel{\mu_1} to (-.5,.6);
\draw[ultra thick] (0.3,-.6)\botlabel{\mu_2} to (0.3,.6);
\draw[ultra thick] (1.1,-.6)\botlabel{\mu_3} to (1.1,.6);
\draw (1.07,-.6) to (1.07,-.3) to (-0.47,.3) to (-.47,.6);
\ovaled{-0.1,0.16}{u+1};
\end{tikzpicture}\ ,&
F_{1,3}(u)&\mapsto
\sum_{\substack{\mu \in \Lambda(n,r)\\\mu_1>0}}
\begin{tikzpicture}[centerzero,scale=.9]
\draw[ultra thick] (-.5,-.6)\botlabel{\mu_1} to (-.5,.6);
\draw[ultra thick] (0.3,-.6)\botlabel{\mu_2} to (0.3,.6);
\draw[ultra thick] (1.1,-.6)\botlabel{\mu_3} to (1.1,.6);
\draw (1.07,.6) to (1.07,.3) to (-0.47,-.3) to (-.47,-.6);
\ovaled{-0.1,-0.16}{u+1};
\end{tikzpicture}\ .
\end{align*}

\begin{proof}[Proof of \cref{EF}]
We prove this for $E_{i,j}(u)$ by induction on $j$.
The base case $j=i+1$ follows from \cref{Eform}.
The induction step uses the recursive formula\footnote{This is well known; e.g., see \cite[(4.9)]{BT} which gives some justification.}
\begin{equation}
E_{i,j}(u) = 
\big[E_{i,j-1}(u),E_{j-1,j}^{(1)}\big],
\end{equation}
and the following diagrammatic relation,
which is a special case of the merge-split relation from \cref{mergesplit}:
\begin{equation}
\begin{tikzpicture}[centerzero]
\draw[ultra thick](0,-.5) to (0,.5);
\draw[thin] (-.4,.5) to (-.03,.1) to (-.03,-.5);
\draw[thin] (.03,.5) to (.03,-.1) to (.4,-.5);
\end{tikzpicture}\ 
-
\begin{tikzpicture}[centerzero]
\draw[ultra thick](0,-.5) to (0,.5);
\draw[thin] (-.5,.5) to (-.03,-.14) to (-.03,-.5);
\draw[thin] (.03,.5) to (.03,.14) to (.5,-.5);
\end{tikzpicture}\ 
=
\begin{tikzpicture}[centerzero]
\draw[ultra thick](0,-.5) to (0,.5);
\draw[thin] (-.45,.5) to (.45,-.5);
\end{tikzpicture}\ .
\end{equation}
The result for $F_{i,j}(u)$ can be deduced 
from the one for $E_{i,j}(u)$ using \cref{trans},
noting also that $F_{i,j}(u) = \tau(E_{i,j}(u))$.
\end{proof}

%% file: s9-presenting.tex
\setcounter{section}{8}

\section{Presenting degenerate affine Schur algebras}\label{s9-presenting}

Let $\gl_n[x]$ be the current Lie algebra, that is, 
$\gl_n \otimes \kk[x]$. We use the notation
\begin{align}
e_{i,j;d} &:= e_{i,j} \otimes x^d.
\end{align}
The Lie bracket satisfies 
$[e_{i,j;a}, e_{k,l;b}] = \delta_{j,k} e_{i,l;a+b}
- \delta_{i,l} e_{k,j;a+b}$.
Let $V[x] := V \otimes \kk[x]$ be the natural $\gl_n[x]$-module
with basis $v_{i;d} := v_i \otimes x^d\:(i=1,\dots,n, d \geq 0)$.
The action of $\gl_n[x]$ on $V[x]$ is given explicitly 
by $e_{i,j;a} v_{k;b} = \delta_{j,k} v_{i;a+b}$.
The tensor space $V[x]^{\otimes r}$ is a
$(\gYn,\kk S_r \circltimes \P_r)$-bimodule in a natural way.
The action of $S_r$ is by permuting tensors, and the action of $x_i \in \P_r$ is by multiplication by $x$ on the $i$th tensor factor.
The proof of the following fundamental lemma depends on Maschke's theorem for the symmetric group.

\begin{lemma}\label{antslem}
If $\kk$ is a field of characteristic 0, 
the homomorphism 
$$
\rho_{n,r}:\gYn \rightarrow \End_{\kk S_r \circltimes P_r}\left(V[x]^{\otimes r}\right)
$$ induced by the natural action of $\gl_n[x]$ on $V[x]$ is surjective.
\end{lemma}

\begin{proof}
See the proof of \cite[Cor.~2.48]{Antor}, which proves the analogous result with the current algebra $\gl_n[x]$ replaced by the loop algebra $\gl_n[t,t^{-1}]$.
Exactly the same argument can be used in our polynomial setting.
Thus, one uses the isomorphism
$$
\End_{\kk S_r \circltimes P_r}\left(V[x]^{\otimes r}\right)
\cong
\left(\End_{\kk[x]}\big(V[x]\big)^{\otimes r}\right)^{S_r},
$$
which is analogous to \cite[(21)]{Antor},
plus \cite[Lem.~2.47]{Antor} applied to the algebra
$A := \End_{\kk[x]}(V[x]) \cong \mathfrak{gl}_n[x]$.
\end{proof}

To explain the relevance of \cref{antslem}, 
recall that there is a filtration on $\Yn$
in which the generator
$T_{i,j}^{(d+1)}$ is of filtered degree $d$. The associated 
graded $\gr \Yn$ is identified with the universal enveloping algebra $\gYn$ so that
$\gr_d T_{i,j}^{(d+1)}=e_{i,j;d}$.
We have also defined a filtration on $\aS(n,r)$
such that $\gr \aS(n,r)$ is identified with the current Schur algebra $\gS(n,r)$; cf. \cref{bots}.
The Drinfeld homomorphism is filtered, so it 
induces $\gr \Drinfeld_{n,r}:\gYn\rightarrow \gS(n,r)$.

\begin{theorem}\label{ontoness}
When $\kk$ is a field of characteristic 0, 
$\gr \Drinfeld_{n,r}$ is surjective.
\end{theorem}

\begin{proof}
From the formula \cref{crap} and \cref{jonsid}, one checks that 
$\gr \Drinfeld_{n,r}$ maps
\begin{equation}\label{friends}
e_{i,j;d} \mapsto 
\begin{dcases}
\sum_{\lambda \in \Lambda(n,r)}
p_d\left(x_{\lambda_{<i}+1},
\dots,x_{\lambda_{\leq i}}\right) 1_\lambda&\text{if $i=j$}\\
\sum_{\substack{\mu \in \Lambda(n,r)\\\mu_j > 0}}
\varsigma_{\diag(\mu_1,\dots,\mu_n)+e_{i,j}-e_{j,j},\
x_{\mu_{<j}+1}^d}&\text{if $i < j$}
\\
\sum_{\substack{\mu \in \Lambda(n,r)\\\mu_j > 0}}
\varsigma_{\diag(\mu_1,\dots,\mu_n)+e_{i,j}-e_{j,j},\  x_{\mu_{\leq j}}^d} &\text{if $i > j$.}
\end{dcases}
\end{equation}
Recalling \cref{nitpick2}, this is an endomorphism
of the right 
$\kk S_r \circltimes\P_r$-module 
$V^{\otimes r} \circltimes \P_r$.
There is an obvious isomorphism of $\kk S_r \circltimes \P_r$-modules
\begin{align}\label{whysolate3}
\theta
:V^{\otimes r} \circltimes \P_r 
&\stackrel{\sim}{\rightarrow}
V[x]^{\otimes r},&
v_\bi \otimes x_1^{d_1} \cdots x_r^{d_r}
&\mapsto
v_{i_1;d_1} \otimes\cdots\otimes v_{i_r;d_r}.
\end{align}
Using \cref{friends}, one checks that $\theta$ is also a left
$\gYn$-module homomorphism. Hence, the following diagram commutes:
\begin{equation}
\begin{tikzcd}
&\gYn\arrow[dl,"\gr\Drinfeld_{n,r}" above left]\arrow[dr,"\rho_{n,r}" above right]\\
\End_{\kk S_r \circltimesintikzcd \P_r}\big(V^{\otimes d}\circltimesintikzcd \P_r\big)
\arrow[rr,"\sim" above,"f \mapsto \theta\circ f \circ \theta^{-1}" below]&&\End_{\kk S_r \circltimesintikzcd \P_r}\left(V[x]^{\otimes d}\right)
\end{tikzcd}
\end{equation}
The surjectivity of $\gr\Drinfeld_{n,r}$ follows from this and \cref{antslem}.
\end{proof}

\begin{corollary}\label{tomorrow}
When $\kk$ is a field of characteristic zero,
$\Drinfeld_{n,r}:\Yn \rightarrow \aS(n,r)$ is surjective.
\end{corollary}

The {\em Harish-Chandra center} $Z_{HC}(\Yn)$ of the Yangian is the central subalgebra of $\Yn$ which is
freely generated by the
elements $C_n^{(d)}\:(d \geq 1)$ defined from
\begin{equation}\label{osaka}
C_n(u) = \sum_{d \geq 0} C_n^{(d)} u^{-d} := D_1(u) D_2(u-1)\cdots D_n(u-n+1).
\end{equation}
This can also be expressed as a certain quantum determinant; e.g., see \cite[Th.~8.6]{BT}.
If $\kk$ is a field of characteristic 0, it is known that the Harish-Chandra center is the entire center of $\Yn$; see \cite[Th.~2.13]{MNO} or \cite[Th.~7.2]{BKdrinfeld}.

\begin{lemma}\label{mamma}
We have that
\begin{equation}\label{debut}
\Drinfeld_{n,r}(C_n(u))
=
\frac{(u+1-x_1)(u+1-x_2)\cdots(u+1-x_r)}{(u-x_1)(u-x_2)\cdots(u-x_r)} 1_{n,r}.
\end{equation}
Hence, 
$\Drinfeld_{n,r}(C_n^{(d+1)}) = 
\tilde p_d(x_1,\dots,x_r) 1_{n,r}\in Z(\aS(n,r))$
for $d \geq 0$.
\end{lemma}

\begin{proof}
By \cref{Dform} and the definition \cref{osaka}, we have that
$$
\Drinfeld_{n,r}(C_n(u))=
\sum_{\lambda \in \Lambda(n,r)}
\begin{tikzpicture}[centerzero]
\draw[ultra thick] (-1.5,-.5)\botlabel{\lambda_1} to (-1.5,.5);
\draw[ultra thick] (-.8,-.5)\botlabel{\lambda_2} to (-.8,.5);
\node at (-.05,0) {$\dots$};
\draw[ultra thick] (.6,-.5)\botlabel{\lambda_n} to (.6,.5);
\ovaledinverse{-1.5,.15}{u+1};
\circled{-1.5,-.15}{u};
\ovaledinverse{.6,.15}{u+1};
\circled{.6,-.15}{u};
\ovaledinverse{-0.8,.15}{u+1};
\circled{-0.8,-.15}{u};
\end{tikzpicture}\ .
$$
This proves \cref{debut}. The last assertion follows from the definition of the deformed power sums in \cref{personofinterest}, together with \cref{interesting} which establishes the centrality.
\end{proof}

\begin{lemma}\label{ontocenter}
When $\kk$ is a field of characteristic 0, 
the Drinfeld homomorphism maps the center of $\Yn$ surjectively onto $Z(\aS(n,r))$.
\end{lemma}

\begin{proof}
When $\kk$ is a field of characteristic 0,
Newton's identity implies that
$\P^{(r)}$ is generated by the power
sums $p_1(x_1,\dots,x_r),\dots,p_r(x_1,\dots,x_r)$.
Hence, $\P^{(r)}$ is also generated by the deformed power sums
$\tilde p_1(x_1,\dots,x_r),\dots,\tilde p_r(x_1,\dots,x_r)$.
Using this, the result follows from
\cref{whatIexpect,mamma}.
\end{proof}

When $\kk$ is a 
field of characteristic $p>0$, $Z(\Yn)$ is much larger
than in characteristic 0. It is generated by the Harish-Chandra center $Z_{HC}(\Yn)$ together with the
{\em $p$-center} $Z_p(\Yn)$, which is the central subalgebra freely generated by the coefficients $B_i^{(pd)}$
for $1 \leq i \leq n$ and $d \geq 1$
defined by setting
\begin{equation}
B_i(u) = \sum_{d \geq 0} B_i^{(d)} u^{-d} := D_i(u) D_i(u-1) \cdots D_i(u-p+1)
\end{equation}
for $i=1,\dots,n$,
together with the coefficents $P_{i,j}^{(pd)}, Q_{i,j}^{(pd)}$
for $1 \leq i < j \leq n$ and $d \geq 1$
defined by
\begin{align}
P_{i,j}(u) &= \sum_{d \geq p} P_{i,j}^{(d)} u^{-d} := E_{i,j}(u)^p,
\\
Q_{i,j}(u) &= \sum_{t \geq p} Q_{i,j}^{(d)} u^{-d} := F_{i,j}(u)^p.
\end{align}
In fact, {\em all} of the coefficients $B_i^{(d)}, P_{i,j}^{(d)}$
and $Q_{i,j}^{(d)}$ for $d \geq 1$ belong to the $p$-center.
This is proved in \cite[Th.~5.4 and Th.~5.11(2)]{BT}.
The {\em restricted Yangian} $\resYn$ 
is the quotient of $Y_n$ by the two-sided ideal generated by $B_i^{(d)}, P_{i,j}^{(d)}$
and $Q_{i,j}^{(d)}$ for $d \geq 1$, i.e., 
the generators of the $p$-center.
This definition is due to Goodwin and Topley \cite[Sec.~4.3]{GTrestricted}. 

\begin{lemma}\label{poschar}
If $\kk$ is a field of characteristic $p > 0$,
the Drinfeld homomorphism $\Drinfeld_{n,r}$
maps $B_i(u)$ to $1_{n,r}$, and it maps $P_{i,j}(u)$ and $Q_{i,j}(u)$ to 0. Hence,
$\Drinfeld_{n,r}$ factors through the quotient to induce a homomorphism
\begin{equation}
\Drinfeld^{[p]}_{n,r}:\resYn \rightarrow \aS(n,r).
\end{equation}
\end{lemma}

\begin{proof}
This is straightforward. For ease of drawing diagrams, we just illustrate the idea by treating the case $p=3$.
By \cref{def}, the image of $B_i(u)$ is
$$
\sum_{\mu \in \Lambda(n,r)}
\begin{tikzpicture}[centerzero]
\draw[ultra thick] (-1.5,-1)\botlabel{\mu_1} to (-1.5,1);
\node at (-1.1,0) {$\dots$};
\draw[ultra thick] (-.7,-1)\botlabel{\mu_{i\!-\!1}\ } to (-.7,1);
\draw[ultra thick] (0,-1)\botlabel{\mu_i} to (0,1);
\draw[ultra thick] (.7,-1)\botlabel{\ \mu_{i\!+\!1}} to (.7,1);
\node at (1.1,0) {$\dots$};
\draw[ultra thick] (1.5,-1)\botlabel{\mu_n} to (1.5,1);
\circledinverse{0,.75}{u};
\ovaled{0,.45}{u+2};
\ovaledinverse{0,.15}{u+2};
\ovaled{0,-.15}{u+1};
\ovaledinverse{0,-.45}{u+1};
\circled{0,-.75}{u};
\end{tikzpicture}
$$
which equals $1_{n,r}$ because $\begin{tikzpicture}[centerzero]
\draw[ultra thick] (0,-.32) to (0,.32);
\circledinverse{0,.15}{u};
\circled{0,-.15}{u};
\end{tikzpicture} = \begin{tikzpicture}[centerzero]
\draw[ultra thick] (0,-.32) to (0,.32);
\end{tikzpicture}\ $.
The arguments for $P_{i,j}(u)$ and $Q_{i,j}(u)$ are similar. Again, we just illustrate with one example, namely, $P_{1,3}(u)\in\Y_3$ in characteristic 3. By \cref{EF}, 
$\Drinfeld_{3,r}\big(P_{1,3}(u)\big)$ equals
$$
\sum_{\substack{\mu \in \Lambda(3,r)\\\mu_3 \geq 3}}
\begin{tikzpicture}[centerzero]
\draw[ultra thick] (-.5,-.8)\botlabel{\mu_1} to (-.5,.9);
\draw[ultra thick] (0.3,-.8)\botlabel{\mu_2} to (0.3,.9);
\draw[ultra thick] (1.1,-.8)\botlabel{\mu_3} to (1.1,.9);
\draw (1.07,-.8) to (1.07,.1) to (-0.45,.7) to (-.45,.9);
\ovaled{-0.1,0.56}{u+1};
\draw (1.06,-.8) to (1.06,-.3) to (-0.46,.3) to (-.46,.9);
\ovaled{-0.1,0.16}{u+1};
\draw (1.05,-.8) to (1.05,-.7) to (-0.47,-.1) to (-.47,.9);
\ovaled{-0.1,-0.24}{u+1};
\end{tikzpicture}
=
\sum_{\substack{\mu \in \Lambda(3,r)\\\mu_3 \geq 3}}
\begin{tikzpicture}[centerzero]
\draw[ultra thick] (-.5,-.8)\botlabel{\mu_1} to (-.5,.9);
\draw[ultra thick] (1.5,-.8)\botlabel{\mu_2} to (1.5,.9);
\draw[ultra thick] (2,-.8)\botlabel{\mu_3} to (2,.9);
\draw[thick] (1.97,-.8) to (1.97,-.4) to (1.2,-.1);
\draw[thick] (-.47,.9) to (-.47,.52) to (-.17,0.41);
\draw (1.2,-.1) to (-.17,.41);
\draw (1.2,-.11) to [out=-150,in=-80,looseness=1.2] (-.17,.4);
\draw (1.2,-.09) to [out=100,in=30,looseness=1.2] (-.17,.42);
\ovaled{0.5,0.56}{u+1};
\ovaled{0.5,0.16}{u+1};
\ovaled{0.5,-0.24}{u+1};
\end{tikzpicture}
=
\sum_{\substack{\mu \in \Lambda(3,r)\\\mu_3 \geq 3}}
\begin{tikzpicture}[centerzero]
\draw[ultra thick] (-.5,-.8)\botlabel{\mu_1} to (-.5,.9);
\draw[ultra thick] (1.5,-.8)\botlabel{\mu_2} to (1.5,.9);
\draw[ultra thick] (2,-.8)\botlabel{\mu_3} to (2,.9);
\draw[thick] (1.97,-.8) to (1.97,-.4) to (1.2,-.1);
\draw[thick] (-.47,.9) to (-.47,.52) to (.31,0.24);
\draw (1.2,-.1) to (.3,.24);
\draw (1.2,-.11) to [out=-150,in=-80,looseness=1] (.3,.23);
\draw (1.2,-.09) to [out=100,in=30,looseness=1] (.31,.25);
\ovaled{-0.1,0.4}{u+1};
\end{tikzpicture},
$$
which equals 0 because 
$\begin{tikzpicture}[centerzero]
\draw[thick] (0,-.3) to (0,-.2);
\draw[thick] (0,.3) to (0,.2);
\draw (.01,.2) to[out=-45,in=45,looseness=1] (0.01,-.2);
\draw (0,.2) to (0,-.2);
\draw (-0.01,.2) to[out=-135,in=135,looseness=1] (-0.01,-.2);
\end{tikzpicture} = 3!\ \begin{tikzpicture}[centerzero]
\draw[thick] (0,-.3) to (0,.3);
\end{tikzpicture}$ by the split-merge relation.
\end{proof}

In the remainder of the section, we assume that $\kk$ is a field of characteristic 0, so that $\Drinfeld_{n,r}$ is surjective.
What can be said about its kernel?
To make a precise statement, we are going to
replace $\Yn$ with a modified form $\Y_{n,r}$
via a construction which is similar in spirit to the passage from $\Un$ to $\U_{n,r}$ in \cref{DGthm}.
Fix a choice of $r$ from now on, and let 
\begin{equation}
\P := \bigoplus_{\lambda \in \Lambda(n,r)} \P^\lambda 1_\lambda
\end{equation}
The root grading $\Yn = \bigoplus_{\alpha \in X(n)} \Y_\alpha$ is defined so that 
$T_{i,j}^{(d)}$ is of weight $\eps_i - \eps_j$;
equivalently, $E_i^{(d)}$ is of weight $\alpha_i$ and
$F_i^{(d)}$ is of weight $-\alpha_i$.
Let 
\begin{equation}\label{ponky}
\widetilde\Y_{n,r} := \bigoplus_{\lambda,\mu \in \Lambda(n,r)}
\P^\lambda \otimes \Y_{\lambda-\mu} \otimes \P^\mu
\end{equation}
viewed as a $(\P,\P)$-bimodule
so that 
$(f 1_\lambda) a (g 1_\mu) = ff_\lambda\otimes a_{\lambda,\mu}\otimes  g_\mu g$
for $f \in \P^\lambda, g \in \P^\mu$ and $a = \sum_{\lambda,\mu \in \Lambda(n,r)} f_\lambda \otimes a_{\lambda,\mu} \otimes g_\mu \in \widetilde\Y_{n,r}$.
Then we define
$\Y_{n,r}$ to be the quotient of the tensor algebra 
\begin{equation}\label{tellytubbies}
T_\P(\widetilde{\Y}_{n,r})
= \P \ \oplus\ \widetilde{\Y}_{n,r} \ \oplus\  \widetilde{\Y}_{n,r} \otimes_{\P} \widetilde{\Y}_{n,r}\  \oplus\  \widetilde{\Y}_{n,r} \otimes_{\P} \widetilde{\Y}_{n,r} \otimes_{\P} \widetilde{\Y}_{n,r}\ \oplus\ \cdots
\end{equation}
by the two-sided ideal generated by the relations
\begin{align}\label{tellyrellies}
1_\lambda (1 \otimes a \otimes 1) 1_\mu \otimes 1_\mu (1 \otimes b \otimes 1) 1_\nu
&= 1_\lambda (1 \otimes ab \otimes 1)1_\nu,\\
1_\lambda(1\otimes  D_i^{(d+1)} \otimes 1) 1_\lambda &= 
\tilde p_d\big(x_{\lambda_{<i}+1},
\dots,x_{\lambda_{\leq i}}\big) 1_\lambda,\label{bellyrellies}
\end{align}
for all $\lambda,\mu,\nu \in \Lambda(n,r), a \in \Y_{\lambda-\mu}, b \in \Y_{\mu-\nu}$, $i=1,\dots,n$ and $d \geq 0$.
As we are in characteristic 0, any symmetric polynomial can be expressed in terms of the deformed power sums. Hence, the relations \cref{tellyrellies,bellyrellies} imply that $1_\lambda \Y_{n,r} 1_\mu$ is spanned by images 
$1_\lambda(\overline{1 \otimes a \otimes 1})1_\mu$
of elements of the form $1_\lambda (1 \otimes a \otimes 1) 1_\mu$
for $a \in \Y_{\lambda-\mu}$.
The Drinfeld homomorphism $\Drinfeld_{n,r}$ induces a homomorphism
\begin{align}
\overline{\Drinfeld}_{n,r}:
\Y_{n,r}
&\rightarrow \aS(n,r),&
1_\lambda(\overline{f \otimes a \otimes g})1_\mu &\mapsto (f 1_\lambda) \Drinfeld_{n,r}(a) (g 1_\mu)
\end{align}
for $f \in \P^\lambda, a \in \Y_{\lambda-\mu}$ and $g \in \P^\mu$.
By \cref{tomorrow}, this homomorphism is surjective.

\begin{theorem}\label{startrek}
If  $n > r$ then $\overline{\Drinfeld}_{n,r}$
is an isomorphism.
\end{theorem}

\begin{proof}
We begin by defining another
algebra $\SY_{n,r}$ which is the analog of $\Y_{n,r}$ for $\SYn$.
Recall that $\SYn$ is
the subalgebra of $\Yn$ generated by the coefficients of $E_i(u)$, $F_i(u)$ and
\begin{equation}
H_i(u) = \sum_{d \geq 0} H_i^{(d)} u^{-d} := - \frac{D_{i+1}(u)}{D_i(u)}
\end{equation}
for $i=1,\dots,n-1$; e.g., see \cite[Sec.~6.1]{BT}.
The root grading of $\SYn$ is a grading 
$$
\SYn = \bigoplus_{\bar\lambda \in \overline{X}(n)} \SY_{\bar\lambda}
$$
by the quotient group $\overline{X}(n) := X(n) / (\eps_1+\cdots+\eps_n)$. 
We denote the image of $\lambda\in X(n)$ in $\overline{X}(n)$
by $\bar\lambda$.
Let $\KK$ be as in \cref{lotsofrain}. Let 
$$
\widetilde{\SY}_{n,r} := \bigoplus_{\lambda,\mu \in \Lambda(n,r)}
\SY_{\bar\lambda-\bar\mu}
$$
viewed as a $(\KK,\KK)$-bimodule so that
$1_\lambda a 1_\mu$ is the projection $a_{\lambda,\mu}$ 
of $a=\sum_{\lambda,\mu \in \Lambda(n,r)} a_{\lambda,\mu} \in \widetilde{\SY}_{n,r}$ onto the $(\lambda,\mu)$th summand.
Then we define $\SY_{n,r}$ to be the quotient of
the tensor algebra $T_\KK(\widetilde{\SY}_{n,r})$ by the two-sided ideal generated by the relations
\begin{align}\label{smellies}
1_\lambda a 1_\mu \otimes 1_\mu b 1_\nu
&= 1_\lambda ab1_\nu,&
1_\lambda H_i^{(1)} 1_\lambda &= (\lambda_i-\lambda_{i+1}) 1_\lambda,
\end{align}
for all $\lambda,\mu,\nu \in \Lambda(n,r), a \in \SY_{\bar\lambda-\bar\mu}, b \in \SY_{\bar\mu-\bar\nu}$, 
and $1,\dots,n-1$.

The inclusions of $\SYn$ into $\Yn$ and $\KK$ into $\P$ induce an algebra homomorphism
\begin{equation}
\inc:
\SY_{n,r} \rightarrow \Y_{n,r}.
\end{equation}
One can show directly from the definitions that $\inc$ is injective. We will not need to use this here so omit the details; in the case $n>r$, arguments in the next paragraphs prove more, namely, that $\inc$ is an isomorphism.

Assume from now on that $n > r$.
We claim that $\inc$ is surjective. 
To prove this, we know already that 
$\Y_{n,r}$ is generated by the coefficients of $1_{\mu+\alpha_i} (1 \otimes \overline{E_i(u)} \otimes 1) 1_\mu$ and
$1_\mu (1 \otimes \overline{F_i(u)} \otimes 1) 1_{\mu+\alpha_i}$
for $i=1,\dots,n-1$ and $\mu$ with $\mu_{i+1}>0$, which are obviously in the image of
$\inc$,
together with the coefficients of
$1_\lambda (1 \otimes \overline{D_i(u)} \otimes 1) 1_\lambda$
for $i=1,\dots,n$ and all $\lambda$. Thus, it suffices to show for
each $i$ and $\lambda\in\Lambda(n,r)$ that all coefficients of
$1_\lambda (1 \otimes \overline{D_i(u)} \otimes 1) 1_\lambda$
are in the image of $\inc$.
Given $\lambda$, we can choose $j$ so that $\lambda_j = 0$; this is the place that the assumption $n > r$ is required.
For this $j$, we have that $1_\lambda (1 \otimes \overline{D_j(u)} \otimes 1) 1_\lambda = 1_\lambda$. 
For any $i \neq j$,
all coefficients of $\frac{D_i(u)}{D_j(u)}$ are in $\SYn$.
From these two statements, it 
follows that all coefficients of $1_\lambda (1 \otimes \overline{D_i(u)} \otimes 1) 1_\lambda$ are in the
image of $\inc$ for all $i=1,\dots,n$.
This proves the claim.

By the claim, there is a surjective homomorphism $\overline{\SDrinfeld}_{n,r}$
fitting into the commutative diagram
\begin{equation}
\begin{tikzcd}
\SY_{n,r}\arrow[dr,twoheadrightarrow,"\overline{\SDrinfeld}_{n,r}" below left]\arrow[rr,twoheadrightarrow,"\inc"]&&\Y_{n,r}\arrow[dl,twoheadrightarrow,"\bar\Drinfeld_{n,r}"]\\
&\aS(n,r)
\end{tikzcd}
\end{equation}
We complete the proof of the theorem by showing that
$\overline{\SDrinfeld}_{n,r}$ is an isomorphism.
Equivalently, we show that the pull-back
functor $\overline{\SDrinfeld}_{n,r}^*:\aS(n,r)\mod\rightarrow
\SY_{n,r}\mod$ is an equivalence of categories.
Recall the functor $\Functor_{n,r}:\aH_r\mod
\rightarrow \aS(n,r)\mod$ from \cref{schurfunctor}. It is an equivalence of categories by \cref{equivcats}. 
Therefore, $\overline{\SDrinfeld}_{n,r}^*$
is an equivalence of categories if and only if
$\overline{\SDrinfeld}_{n,r}^* \circ \Functor_{n,r}$ is an equivalence of categories. The latter statement is proved in the next paragraph.

There is an $\sl_n$ analog $\SU_{n,r}$ 
of the algebra $\U_{n,r}$, and an $\sl_n$ analog of \cref{DGthm}
which gives an isomorphism $\overline{\sdrinfeld}_{n,r}:\SU_{n,r}
\stackrel{\sim}{\rightarrow} \S(n,r)$.
It follows that an 
$\SU_{n,r}$-module is the same thing as an $\SUn$-module whose restriction to $\sl_n$ is a polynomial representation of degree $r$. There is also an inclusion $\SU_{n,r} \hookrightarrow
\SY_{n,r}$ induced by the inclusion  $\SUn\hookrightarrow\SYn$;
this induced homomorphism is injective because its composition with $\overline{\SDrinfeld}_{n,r}$ is the isomorphism $\overline{\sdrinfeld}_{n,r}$
composed with the inclusion $\S(n,r)\hookrightarrow \aS(n,r)$ from \cref{easy}.
It follows that the category appearing in the
statement of Drinfeld's theorem in the introduction is 
identified with $\SY_{n,r}\mod$. Also
$\overline{\SDrinfeld}_{n,r}^* \circ \Functor_{n,r}$
is identified the Drinfeld functor $\Res^{\Yn}_{\SYn} \circ 
(V^{\otimes d} \otimes_{\kk S_r} -)$.
Hence, $\overline{\SDrinfeld}_{n,r}^* \circ \Functor_{n,r}$ is an equivalence by Drinfeld's theorem.
\end{proof}

\begin{conjecture}\label{ourconjecture}
$\overline{\Drinfeld}_{n,r}$ is an isomorphism for all values of $n$ and $r$.
\end{conjecture}

\begin{remark}
A consequence of \cref{ourconjecture,equivcats} is that
the Drinfeld functor $V^{\otimes r}\otimes_{\kk \S_r}-$
is an equivalence between $\aH_r\mod$ and $\Y_{n,r}\mod$
for all $n \geq r$.
Drinfeld's theorem for $n > r$ as stated in the introduction 
can be recovered from this by reversing the argument in the proof of \cref{startrek}. 
\end{remark}

%% file: s10-reps.tex
\setcounter{section}{9}

\section{Representation theory of \texorpdfstring{$\aS(n,r)$}{}}\label{s10-reps}

We assume in this section that $\kk$ is an algebraically closed field of characteristic 0, so that
$\Drinfeld_{n,r}$ is surjective
thanks to \cref{tomorrow}. 
By a {\em polynomial representation} of $\Yn$ of degree $r$,
we mean a $\Yn$-module which is the pull-back
$\Drinfeld_{n,r}^* M$ of a left $\aS(n,r)$-module $M$.
The category of polynomial representations of $\Yn$ of degree $r$ is naturally identified with $\aS(n,r)\mod$; if \cref{ourconjecture} is true
it is also the same as $\Y_{n,r}\mod$.
The goal is to classify 
irreducible polynomial representations of $\Yn$.

The characteristic 0 assumption means that there is a well-defined
partial order $\leq$
on $\kk$ defined by
$b \leq a \Leftrightarrow a-b \in \N$. 
It is also needed in order to be able to 
prove the following elementary lemma:

\begin{lemma}\label{mathcircle}
Let $f(u), g(u) \in \kk[u]$ be monic polynomials.
If
$\frac{f(u+1)}{f(u)} = \frac{g(u+1)}{g(u)}$
then $f(u) = g(u)$.
\end{lemma}

\begin{proof}
Exercise. 
\end{proof}

\cref{overcenter} implies that
every irreducible polynomial representation of degree $r$
is finite-dimensional. So, in order to classify them, 
we should start by recalling the classification of finite-dimensional irreducible representations of $\Yn$ from \cite{Dnew}.
Let $$
A(u) = (A_1(u), \dots, A_n(u))$$ 
be an $n$-tuple
of formal power series $A_1(u),\dots,A_n(u) \in 1 + u^{-1} \kk\llbracket u^{-1}\rrbracket$. There is a unique (up to isomorphism) irreducible $\Yn$-module $L(A(u))$ generated by a non-zero
vector $v_+$ such that
\begin{itemize}
\item $E_i(u) v_+ = 0$ for $i=1,\dots,n-1$;
\item $D_i(u) v_+ = A_i(u) v_+$ for $i=1,\dots,n$.
\end{itemize}
The module $L(A(u))$ may be constructed as the unique irreducible quotient of a Verma-type module, which is defined using the triangular decomposition of $\Yn$ arising from 
the Drinfeld presentation.

\begin{theorem}[Drinfeld]
For $A(u)$ as above, 
$L(A(u))$ is finite-dimensional if and only if
\begin{equation}\label{drinfeldpoly}
\frac{A_i(u)}{A_{i+1}(u)} = \frac{P_i(u+1)}{P_i(u)}
\end{equation}
for monic polynomials $P_1(u),\dots,P_{n-1}(u) \in \kk[u]$ (called {\em Drinfeld polynomials}). Moreover, every finite-dimensional irreducible $\Yn$-module is isomorphic to $L(A(u))$ for a unique such $A(u)$.
\end{theorem}

In view of this, the problem of classifying irreducible polynomial representations of $\Yn$ is thus reduced to the problem of determining
which $L(A(u))$ are polynomial of degree $r$,
which is the content of the next theorem:

\begin{theorem}\label{laugh}
For $A(u)$ as above,
the irreducible $\Yn$-module
$L(A(u))$
is a polynomial representation of degree $r$
if and only if
there exists a (necessarily unique) sequence
$$
\lambda(u) = (\lambda_1(u),\dots,\lambda_n(u))
$$ 
of monic polynomials $\lambda_i(u) \in \kk[u]$ such that
\begin{enumerate}
\item $A_i(u) = \frac{\lambda_i(u+1)}{\lambda_{i}(u)}$
for $i=1,\dots,n$;
\item
$\deg \lambda_1(u) + \cdots + \deg \lambda_n(u) = r$;
\item
$\lambda_{i+1}(u)\ \big|\ \lambda_i(u)$
for $i=1,\dots,n-1$.
\end{enumerate}
Hence, over a field of characteristic 0, isomorphism classes of irreducible polynomial representations of $\Yn$ of degree $r$ are naturally indexed by sequences
$\lambda(u) = (\lambda_1(u),\dots,\lambda_n(u))$ of monic polynomials satisfying (2) and (3).
\end{theorem}

\begin{proof}[Proof of the necessary condition $(\Rightarrow$)
in \cref{laugh}]
Suppose that $L(A(u))$ is a polynomial representation of degree $r$, i.e., it is an $\aS(n,r)$-module.
Let $\lambda \in \Lambda(n,r)$ be the weight of the
highest weight vector $v_+ \in L(A(u))$; explicitly,
$\lambda_i$ is the $u^{-1}$-coefficient of $A_i(u)$.
From $D_i(u) v_+ = A_i(u) v_+$ and \cref{Dform}, we deduce that
$(u+i-1-x_{\lambda_{<i}+1})
\cdots (u+i-1-x_{\lambda_{\leq i}})
1_\lambda \in \aS(n,r)$ acts on $v_+$
by multiplication by a monic polynomial
$\lambda_i(u) \in \kk[u]$ of degree $\lambda_i$ such that
$A_i(u) = \frac{\lambda_i(u+1)}{\lambda_i(u)}$
for $i=1,\dots,n$.
This proves (1) and (2). For (3), since $L(A(u))$ is finite-dimensional, there is are monic polynomials
$P_i(u)$ such that the equation \cref{drinfeldpoly} 
holds for $i=1,\dots,n-1$. This implies that
$$
\frac{\lambda_i(u+1)}{\lambda_i(u)} = \frac{\lambda_{i+1}(u+1)P_i(u+1)}{\lambda_{i+1}(u)P_i(u)}.
$$
Now \cref{mathcircle} gives that $\lambda_i(u) = \lambda_{i+1}(u) P_i(u)$, and (3) follows.
\end{proof}

The sufficient condition $(\Leftarrow)$ needed to complete the proof of \cref{laugh} will be proved a little later. 
To prepare for this, 
we need to recall some further results from \cite{Arakawa}.
As in \cref{equivcats}, let
\begin{equation}
\Functor_{n,r}:\aH_r\mod \rightarrow \aS(n,r)\mod \subset
\Yn\mod
\end{equation}
be the functor defined by tensoring over $\aH_r$ with
$V^{\otimes r} \otimes_{\kk S_r} \aH_r$ viewed as a $(\Yn,\aH_r)$-bimodule via the Drinfeld homomorphism. 
We refer to this as the {\em Drinfeld functor}.
For a left $\aH_r$-module $M$, there is the obvious isomorphism of vector spaces
\begin{equation}\label{freight}
V^{\otimes r} \otimes_{\kk S_r} \aH_r \otimes_{\aH_r}
M \cong V^{\otimes r} \otimes_{\kk S_r} M.
\end{equation}
It implies that $\Res^{\Yn}_{\Un} \circ \Functor_{n,r}
\cong \functor_{n,r} \circ \Res^{\aH_r}_{\kk S_r}$ where
$\functor_{n,r} := V^{\otimes r} \otimes_{\kk S_r} -:\kk S_r\mod \rightarrow \Un\mod$
is the usual Schur functor.

\begin{lemma}[Chari-Pressley]\label{cplem}
The natural tensor product on $\Yn\mod$ restricts to a functor
$$
-\otimes-:\aS(n,r)\mod \times \aS(n,s)\mod \rightarrow \aS(n,r+s)\mod.
$$
Moreover, there is an isomorphism 
$\Functor_{n,r}(-) \otimes \Functor_{n,s}(-)
\cong \Functor_{n,r+s} \circ (-\ostar-)$ of functors from
$\aH_{r}\mod\times \aH_s\mod$ to $\aS(n,r+s)\mod$.
\end{lemma}

\begin{proof}
There is an isomorphism
$$
\big(V^{\otimes r} \otimes_{\kk S_r} \aH_r\big)
\otimes \big(V^{\otimes s} \otimes_{\kk S_{s}} 
\aH_{s}\big) \cong V^{\otimes (r+s)} \otimes_{\kk S_{r+s}}
\aH_{r+s}
$$ 
of $(\Yn,\aH_{(r,s)})$-bimodules.
\end{proof}

For $b \leq a$ in $\kk$ with $r=a-b+1$, 
let $\kk_{[b,a]}$ be the one-dimensional 
left $\aH_r$-module
on which $x_i\:(1 \leq i \leq r)$ acts as $b+i-1$ and $w \in S_r$ acts as $(-1)^{\ell(w)}$.
This module is a {\em segment} in the terminology of \cite{zelevinsky}.

\begin{lemma}[Arakawa]\label{alem}
For $b \leq a$ with $r:=a-b+1 \leq n$, there are $\Yn$-module isomorphisms
\begin{equation}
\textstyle
\Functor_{n,r} \kk_{[b,a]}
\cong \ev_b^* \left(\bigwedge^r V\right)
\cong L(A(u))
\end{equation}
where $A(u) := \left(\frac{\lambda_1(u+1)}{\lambda_1(u)},\dots,\frac{\lambda_n(u+1)}{\lambda_n(u)}\right)$
with $\lambda_i(u) := \begin{cases}u-b&\text{if $1\leq i \leq r$}\\
1&\text{if $r+1\leq i \leq n$.}\end{cases}$
\end{lemma}

\begin{proof}
See \cite[Prop.~6]{Arakawa}.
Here is another proof.
By \cref{freight}, $\Res^{\Yn}_{\Un} \left( \Functor_{n,r} \kk_{[b,a]}\right)
\cong \bigwedge^r V$. Thus, $\Functor_{n,r} \kk_{[b,a]}$
is an irreducible $\Yn$-module. The vector
$v_+ := v_1\wedge\cdots\wedge v_r$ is a highest weight vector
of weight $\eps_1+\cdots+\eps_r$.
By \cref{Dform}, $D_i(u)$ acts on $v_+$ in the same way as
$\frac{u+i-x_i}{u+i-1-x_i}$,
which is 
by multiplication by $\frac{u+1-b}{u-b} = 1+\frac{1}{u-b}$
if $1 \leq i \leq r$,
or as $1$ if $r+1\leq i \leq n$.
This is the same as how $D_i(u)$ acts on this vector
in $\ev_b^* \left(\bigwedge^r V\right)$. Hence, the two modules are isomorphic,
We have also computed how each $D_i(u)$ acts on $v_+$,
identifying both modules with $L(A(u))$.
\end{proof}

Now suppose that we are given $m \geq 0$ and $\ba,\bb \in \kk^m$
such that $0 \leq a_j - b_j \leq n-1$ for each $j=1,\dots,m$.
Consider the $\Yn$-module
\begin{equation}
M(\ba,\bb) := \textstyle
\ev^*_{b_1}\left( \bigwedge^{a_1-b_1+1} V\right) \otimes \cdots \otimes
\ev^*_{b_m} \left(\bigwedge^{a_m-b_m+1} V\right).
\end{equation}
We call $M(\ba,\bb)$ a {\em standard module}.

\begin{lemma}\label{itspoly}
Assume that  
$0 \leq a_j - b_j \leq n-1$ for $j=1,\dots,m$.
Then the standard 
module $M(\ba,\bb)$ is a polynomial representation of degree
$\sum_{j=1}^m (a_j - b_j+1)$.
\end{lemma}

\begin{proof}
\cref{cplem,alem} imply that $M(\ba,\bb)$ is
isomorphic to the image under $\Drinfeld_{n,r}$
of the multisegment
$\kk_{[b_1,a_1]} \ostar
\cdots\ostar \kk_{[b_m,a_m]}$.
\end{proof}

We say that $\ba \in \kk^m$ is {\em dominant} 
if $1 \leq i < j \leq m \Rightarrow a_i \not< a_j$.
The following theorem was proved originally by Nazarov and Tarasov \cite{NT2}; see \cite[Th.~8]{Arakawa} for another proof exploiting the Drinfeld functor.

\begin{theorem}[Nazarov-Tarasov, Arakawa]\label{nta}
If $\ba \in \kk^m$ is dominant and $0 \leq a_j - b_j \leq n-1$ for $j=1,\dots,m$
then the standard module
$M(\ba,\bb)$ has a unique irreducible quotient
$L(\ba,\bb)$.
Moreover,
$L(\ba,\bb) \cong L(A(u))$
where $A(u) := \left(\frac{\lambda_1(u+1)}{\lambda_1(u)},\dots,\frac{\lambda_n(u+1)}{\lambda_n(u)}\right)$
with
\begin{equation}
\lambda_i(u) := 
\prod_{\substack{1 \leq j \leq m\\ i \leq a_j - b_j+1}} (u-b_j).
\end{equation}
\end{theorem}

Using this, we can complete the proof of \cref{laugh}.

\begin{proof}[Proof of the sufficient condition $(\Leftarrow$)
in \cref{laugh}]
Given $A(u)$ with $A_i(u) = \frac{\lambda_i(u+1)}{\lambda_i(u)}$
as in \cref{laugh}(1)--(3), we need to show that 
$L(A(u))$ is a polynomial representation of degree $r$.
Let $m := \deg \lambda_1(u)$.
We define $\ba,\bb \in \kk^m$ as follows:
\begin{itemize}
\item Let $b_1$ be any root of $\lambda_1(u)$.
\item 
Let $b_1 \leq a_1 \leq b_1+n-1$ 
be maximal such that $b_1$ is a root of $\lambda_i(u)$ for all
$1 \leq i \leq a_1-b_1+1$.
\item Divide $\lambda_i(u)$ by $(u-b_1)$ for each 
$i=1,\dots,a_1-b_1+1$, then iterate $(m-1)$ more times with the new polynomials to obtain $a_2 \leq b_2,\dots,a_m\leq b_m$.
\end{itemize}
This ensures that $0 \leq a_j -b_j \leq n-1$ for each $j=1,\dots,m$.
Finally, we simultaneously rearrange the $m$-tuples
$\ba,\bb$ to ensure that $\ba$ dominant.
\cref{nta} implies that $L(A(u)) \cong L(\ba,\bb)$, and
this is a polynomial representation of degree $r$ by \cref{itspoly}.
\end{proof}

Now that \cref{laugh} is proved, we switch to using the notation
$L\big(\lambda_1(u),\dots,\lambda_n(u)\big)$ to denote the irreducible polynomial representation of $\Yn$ indexed by a sequence 
$\lambda_1(u),\dots,\lambda_n(u)$ of 
monic polynomials satisfying \cref{laugh}(3).
Using also \cref{equivcats}, it is easy to deduce
the following convenient parametrization of irreducible representations of the degenerate affine Hecke algebra $\aH_r$, which is different from the usual parametrization by multisegments:

\begin{corollary}
Isomorphism classes of irreducible left $\aH_r$-modules
are in bijection with
sequences $\lambda(u) = (\lambda_1(u),\lambda_2(u),\dots)$
of monic polynomials in $\kk[u]$
whose degrees sum to $r$ and
$\lambda_{i+1}(u) \ \big|\ \lambda_i(u)$ for each $i\geq 1$.
The irreducible module $D(\lambda(u))$ labelled by such a sequence
may be constructed explicitly by setting
\begin{equation}
D(\lambda(u)) := 1_\omega
L\big(\lambda_1(u),\dots,\lambda_n(u)\big)
\end{equation}
in the setup of \cref{equivcats}.
Alternatively, letting $\ba,\bb \in \kk^m$ be the sequences constructed from $\lambda(u)$ following the algorithm in the proof of the sufficient condition of \cref{laugh} just explained (with $\ba$ dominant), 
$D(\lambda(u))$ is the irreducible head of
$\kk_{[b_1,a_1]} \ostar \cdots \ostar \kk_{[b_m,a_m]}$
\end{corollary}

\begin{remark}
There is a remarkable
explicit formula for the composition multiplicities of the standard modules $M(\ba,\bb)$ in terms of Kazhdan-Lusztig polynomials. It is closely 
related to the degenerate analog of Zelevinsky's $p$-adic analog of the Kazhdan-Lusztig conjecture for $GL_n$.
See \cite[Th.~15]{Arakawa}, which is proved using results from \cite{AS} deduced ultimately from the Kazhdan-Lusztig conjecture for the Lie algebra $\mathfrak{gl}_r$.
\end{remark}

%% file: das.bbl
\begin{thebibliography}{BEAEO20}

\bibitem[Ant20]{Antor}
Jonas Antor.
\newblock Affine versions of {S}chur-{W}eyl duality.
\newblock Master's thesis, Universtät Bonn, 2020.
\newblock URL:
  \url{https://www.math.uni-bonn.de/ag/stroppel/Masterarbeit_Antor-4-1.pdf}.

\bibitem[Ara99]{Arakawa}
Tomoyuki Arakawa.
\newblock Drinfeld functor and finite-dimensional representations of {Y}angian.
\newblock {\em Comm. Math. Phys.}, 205(1):1--18, 1999.
\newblock \href {https://doi.org/10.1007/s002200050664}
  {\path{doi:10.1007/s002200050664}}.

\bibitem[AS98]{AS}
Tomoyuki Arakawa and Takeshi Suzuki.
\newblock Duality between {$\sl_n({\C})$} and the degenerate affine {H}ecke
  algebra.
\newblock {\em J. Algebra}, 209(1):288--304, 1998.
\newblock \href {https://doi.org/10.1006/jabr.1998.7530}
  {\path{doi:10.1006/jabr.1998.7530}}.

\bibitem[BEAEO20]{BEEO}
Jonathan Brundan, Inna Entova-Aizenbud, Pavel Etingof, and Victor Ostrik.
\newblock Semisimplification of the category of tilting modules for {$GL_ n$}.
\newblock {\em Adv. Math.}, 375:107331, 37, 2020.
\newblock \href {https://doi.org/10.1016/j.aim.2020.107331}
  {\path{doi:10.1016/j.aim.2020.107331}}.

\bibitem[BK05]{BKdrinfeld}
Jonathan Brundan and Alexander Kleshchev.
\newblock Parabolic presentations of the {Y}angian {$Y({\mathfrak{gl}}_n)$}.
\newblock {\em Comm. Math. Phys.}, 254(1):191--220, 2005.
\newblock \href {https://doi.org/10.1007/s00220-004-1249-6}
  {\path{doi:10.1007/s00220-004-1249-6}}.

\bibitem[BK08]{BKschur}
Jonathan Brundan and Alexander Kleshchev.
\newblock Schur-{W}eyl duality for higher levels.
\newblock {\em Selecta Math. (N.S.)}, 14(1):1--57, 2008.
\newblock \href {https://doi.org/10.1007/s00029-008-0059-7}
  {\path{doi:10.1007/s00029-008-0059-7}}.

\bibitem[BLM90]{BLM}
Alexander Beilinson, George Lusztig, and Robert MacPherson.
\newblock A geometric setting for the quantum deformation of {${\rm GL}_n$}.
\newblock {\em Duke Math. J.}, 61(2):655--677, 1990.
\newblock \href {https://doi.org/10.1215/S0012-7094-90-06124-1}
  {\path{doi:10.1215/S0012-7094-90-06124-1}}.

\bibitem[Bru25]{BSchur}
Jonathan Brundan.
\newblock The {$q$}-{S}chur category and polynomial tilting modules for quantum
  {${\rm GL}_n$}.
\newblock {\em Pacific J. Math.}, 336(1-2):63--112, 2025.
\newblock \href {https://doi.org/10.2140/pjm.2025.336.63}
  {\path{doi:10.2140/pjm.2025.336.63}}.

\bibitem[BT18]{BT}
Jonathan Brundan and Lewis Topley.
\newblock The {$p$}-centre of {Y}angians and shifted {Y}angians.
\newblock {\em Mosc. Math. J.}, 18(4):617--657, 2018.
\newblock \href {https://doi.org/10.17323/1609-4514-2018-18-4-617-657}
  {\path{doi:10.17323/1609-4514-2018-18-4-617-657}}.

\bibitem[CKM14]{CKM}
Sabin Cautis, Joel Kamnitzer, and Scott Morrison.
\newblock Webs and quantum skew {H}owe duality.
\newblock {\em Math. Ann.}, 360(1-2):351--390, 2014.
\newblock \href {https://doi.org/10.1007/s00208-013-0984-4}
  {\path{doi:10.1007/s00208-013-0984-4}}.

\bibitem[CP96]{CP}
Vyjayanthi Chari and Andrew Pressley.
\newblock Quantum affine algebras and affine {H}ecke algebras.
\newblock {\em Pacific J. Math.}, 174(2):295--326, 1996.
\newblock URL: \url{http://projecteuclid.org/euclid.pjm/1102365173}.

\bibitem[DF15]{du-fu}
Jie Du and Qiang Fu.
\newblock Quantum affine {$\mathfrak{gl}_n$} via {H}ecke algebras.
\newblock {\em Adv. Math.}, 282:23--46, 2015.
\newblock \href {https://doi.org/10.1016/j.aim.2015.06.007}
  {\path{doi:10.1016/j.aim.2015.06.007}}.

\bibitem[DG02]{doty-giaquinto}
Stephen Doty and Anthony Giaquinto.
\newblock Presenting {S}chur algebras.
\newblock {\em Int. Math. Res. Not.}, 36:1907--1944, 2002.
\newblock \href {https://doi.org/10.1155/S1073792802201026}
  {\path{doi:10.1155/S1073792802201026}}.

\bibitem[DG07]{green-doty}
Stephen Doty and Richard Green.
\newblock Presenting affine {$q$}-{S}chur algebras.
\newblock {\em Math. Z.}, 256(2):311--345, 2007.
\newblock \href {https://doi.org/10.1007/s00209-006-0076-1}
  {\path{doi:10.1007/s00209-006-0076-1}}.

\bibitem[DJ86]{DJgl}
Richard Dipper and Gordon James.
\newblock Representations of {H}ecke algebras of general linear groups.
\newblock {\em Proc. London Math. Soc. (3)}, 52(1):20--52, 1986.
\newblock \href {https://doi.org/10.1112/plms/s3-52.1.20}
  {\path{doi:10.1112/plms/s3-52.1.20}}.

\bibitem[DKM25]{DKMZ2}
Nicholas Davidson, Jonathan Kujawa, and Robert Muth.
\newblock Superalgebra deformations of web categories: affine and cyclotomic
  webs, 2025.
\newblock \href {https://arxiv.org/abs/2511.21671} {\path{arXiv:2511.21671}}.

\bibitem[DKMZ23]{DKMZ1}
Nicholas Davidson, Jonathan Kujawa, Robert Muth, and Jieru Zhu.
\newblock Superalgebra deformations of web categories: finite webs, 2023.
\newblock \href {https://arxiv.org/abs/2302.04073} {\path{arXiv:2302.04073}}.

\bibitem[Dot03]{doty}
Stephen Doty.
\newblock Presenting generalized {$q$}-{S}chur algebras.
\newblock {\em Represent. Theory}, 7:196--213, 2003.
\newblock \href {https://doi.org/10.1090/S1088-4165-03-00176-6}
  {\path{doi:10.1090/S1088-4165-03-00176-6}}.

\bibitem[Dri86]{Drinfeld}
Vladimir Drinfeld.
\newblock Degenerate affine {H}ecke algebras and {Y}angians.
\newblock {\em Funktsional. Anal. i Prilozhen.}, 20(1):69--70, 1986.

\bibitem[Dri87]{Dnew}
Vladimir Drinfeld.
\newblock A new realization of {Y}angians and of quantum affine algebras.
\newblock {\em Dokl. Akad. Nauk SSSR}, 296(1):13--17, 1987.

\bibitem[Eli15]{Elias}
Ben Elias.
\newblock Light ladders and clasp conjectures.
\newblock \arxiv{1510.06840}, 2015.

\bibitem[FM02]{FM}
Edward Frenkel and Evgeny Mukhin.
\newblock The {H}opf algebra {${\rm Rep}\,U_q\widehat{\mathfrak{gl}}_\infty$}.
\newblock {\em Selecta Math. (N.S.)}, 8(4):537--635, 2002.
\newblock \href {https://doi.org/10.1007/PL00012603}
  {\path{doi:10.1007/PL00012603}}.

\bibitem[Gre99]{Green-affine}
Richard Green.
\newblock The affine {$q$}-{S}chur algebra.
\newblock {\em J. Algebra}, 215(2):379--411, 1999.
\newblock \href {https://doi.org/10.1006/jabr.1998.7753}
  {\path{doi:10.1006/jabr.1998.7753}}.

\bibitem[Gre07]{Green}
James~A. Green.
\newblock {\em Polynomial {R}epresentations of {${\rm GL}_{n}$}}, volume 830 of
  {\em Lecture Notes in Mathematics}.
\newblock Springer, Berlin, augmented edition, 2007.
\newblock With an appendix on Schensted correspondence and Littelmann paths by
  K. Erdmann, Green and M. Schocker.
\newblock URL:
  \url{https://link.springer.com/content/pdf/10.1007/3-540-46944-3.pdf}.

\bibitem[GT21]{GTrestricted}
Simon Goodwin and Lewis Topley.
\newblock Restricted shifted {Y}angians and restricted finite {$W$}-algebras.
\newblock {\em Trans. Amer. Math. Soc. Ser. B}, 8:190--228, 2021.
\newblock \href {https://doi.org/10.1090/btran/63}
  {\path{doi:10.1090/btran/63}}.

\bibitem[Kle05]{Kleshchev}
Alexander Kleshchev.
\newblock {\em Linear and {P}rojective {R}epresentations of {S}ymmetric
  {G}roups}, volume 163 of {\em Cambridge Tracts in Mathematics}.
\newblock Cambridge University Press, Cambridge, 2005.
\newblock \href {https://doi.org/10.1017/CBO9780511542800}
  {\path{doi:10.1017/CBO9780511542800}}.

\bibitem[Kup96]{Kuperberg}
Greg Kuperberg.
\newblock Spiders for rank {$2$} {L}ie algebras.
\newblock {\em Comm. Math. Phys.}, 180(1):109--151, 1996.
\newblock URL: \url{http://projecteuclid.org/euclid.cmp/1104287237}.

\bibitem[LM25]{LM}
Chun-Ju Lai and Alexandre Minets.
\newblock Schurification of polynomial quantum wreath products, 2025.
\newblock \href {https://arxiv.org/abs/2502.02108} {\path{arXiv:2502.02108}}.

\bibitem[MNO96]{MNO}
Alexander Molev, Maxim Nazarov, and Grigori Olshanski.
\newblock Yangians and classical {L}ie algebras.
\newblock {\em Uspekhi Mat. Nauk}, 51(2(308)):27--104, 1996.
\newblock \href {https://doi.org/10.1070/RM1996v051n02ABEH002772}
  {\path{doi:10.1070/RM1996v051n02ABEH002772}}.

\bibitem[MOY98]{MOY}
Hitoshi Murakami, Tomotada Ohtsuki, and Shuji Yamada.
\newblock Homfly polynomial via an invariant of colored plane graphs.
\newblock {\em Enseign. Math. (2)}, 44(3-4):325--360, 1998.

\bibitem[MS19]{MS}
Vanessa Miemietz and Catharina Stroppel.
\newblock Affine quiver {S}chur algebras and {$p$}-adic {$GL_n$}.
\newblock {\em Selecta Math. (N.S.)}, 25(2):Paper No. 32, 66, 2019.
\newblock \href {https://doi.org/10.1007/s00029-019-0474-y}
  {\path{doi:10.1007/s00029-019-0474-y}}.

\bibitem[MS21]{MakS}
Ruslan Maksimau and Catharina Stroppel.
\newblock Higher level affine {S}chur and {H}ecke algebras.
\newblock {\em J. Pure Appl. Algebra}, 225(8):Paper No. 106442, 44, 2021.
\newblock \href {https://doi.org/10.1016/j.jpaa.2020.106442}
  {\path{doi:10.1016/j.jpaa.2020.106442}}.

\bibitem[MSV13]{MSV}
Marco Mackaay, Marko Stosić, and Pedro Vaz.
\newblock A diagrammatic categorification of the {$q$}-{S}chur algebra.
\newblock {\em Quantum Topol.}, 4(1):1--75, 2013.
\newblock \href {https://doi.org/10.4171/QT/34} {\path{doi:10.4171/QT/34}}.

\bibitem[NT98a]{NT2}
Maxim Nazarov and Vitaly Tarasov.
\newblock On irreducibility of tensor products of {Y}angian modules.
\newblock {\em Internat. Math. Res. Notices}, 3:125--150, 1998.
\newblock \href {https://doi.org/10.1155/S1073792898000129}
  {\path{doi:10.1155/S1073792898000129}}.

\bibitem[NT98b]{NT}
Maxim Nazarov and Vitaly Tarasov.
\newblock Representations of {Y}angians with {G}elfand-{Z}etlin bases.
\newblock {\em J. Reine Angew. Math.}, 496:181--212, 1998.
\newblock \href {https://doi.org/10.1515/crll.1998.029}
  {\path{doi:10.1515/crll.1998.029}}.

\bibitem[SSW24]{ShenSongWang}
Yaolong Shen, Linliang Song, and Weiqiang Wang.
\newblock Affine and cyclotomic $q$-{S}chur categories via webs, 2024.
\newblock \href {https://arxiv.org/abs/2504.10270} {\path{arXiv:2504.10270}}.

\bibitem[SW11]{SW}
Catharina Stroppel and Benjamin Webster.
\newblock Quiver {S}chur algebras and $q$-{F}ock space, 2011.
\newblock \href {https://arxiv.org/abs/1110.1115} {\path{arXiv:1110.1115}}.

\bibitem[SW24]{SongWang2}
Linliang Song and Weiqiang Wang.
\newblock Affine and cyclotomic {S}chur categories, 2024.
\newblock \href {https://arxiv.org/abs/2407.10119} {\path{arXiv:2407.10119}}.

\bibitem[SW25]{SongWang}
Linliang Song and Weiqiang Wang.
\newblock Affine and cyclotomic webs.
\newblock {\em J. Lond. Math. Soc. (2)}, 112(3):Paper No. e70278, 43, 2025.
\newblock \href {https://doi.org/10.1112/jlms.70278}
  {\path{doi:10.1112/jlms.70278}}.

\bibitem[Vig03]{vigneras}
Marie-France Vign\'eras.
\newblock Schur algebras of reductive {$p$}-adic groups. {I}.
\newblock {\em Duke Math. J.}, 116(1):35--75, 2003.
\newblock \href {https://doi.org/10.1215/S0012-7094-03-11612-9}
  {\path{doi:10.1215/S0012-7094-03-11612-9}}.

\bibitem[Zel80]{zelevinsky}
A.~V. Zelevinsky.
\newblock Induced representations of reductive $p$-adic groups. {II}. {O}n
  irreducible representations of $\mathrm{GL}(n)$.
\newblock {\em Ann. Sci. École Norm. Sup. (4)}, 13(2):165--210, 1980.
\newblock URL: \url{http://www.numdam.org/item?id=ASENS_1980_4_13_2_165_0}.

\end{thebibliography}
